\documentclass{article}

\usepackage[utf8]{inputenc} 
\usepackage[T1]{fontenc}

\usepackage{amsmath}
\usepackage{amsfonts}
\usepackage{amssymb}
\usepackage{amsthm}
\usepackage{graphicx}
\usepackage{fancyhdr}
\usepackage{geometry}
\usepackage{hyperref}
\usepackage{mathrsfs}
\usepackage[nottoc,numbib]{tocbibind}
\usepackage{mathtools}
\usepackage{color}
\usepackage[normalem]{ulem}

\allowdisplaybreaks

\makeatletter
\DeclareRobustCommand{\format@sec@number}[2]{{\normalfont\upshape#1}#2}

\makeatother

\def\e{\varepsilon}

\def\a{\alpha}
\def\b{\beta}

\def\d{\delta}

\def\l{\lambda}

\def\s{\sigma}

\def\R{\mathbb R}
\def\N{{\mathbb N}}

\def\Z{\mathbb Z}
\def\T{\mathbb T}

\def\C{\mathbb C}

\def\Q{\mathbb Q}

\def\({\biggl(}
\def\){\biggr)}
\def\<{\mathbf{\langle}}
\def\>{\mathbf{\rangle}}

\newcommand{\ee}{\mathrm{e}}

\newcommand{\Meng}[2]{\left\{#1\mathrel{}\middle|\mathrel{}#2\right\}}
\newcommand{\abs}[1]{\left\lvert#1\right\rvert}

\numberwithin{equation}{section}

\newtheorem{theorem}[equation]{Theorem}
\newtheorem{proposition}[equation]{Proposition}
\newtheorem{lemma}[equation]{Lemma}
\newtheorem{corollary}[equation]{Corollary}

\newtheorem{maintheorem}{Theorem}

\theoremstyle{definition}
\newtheorem{definition}[equation]{Definition}

\theoremstyle{definition}

\theoremstyle{remark}
\newtheorem{remark}{Remark}

\title{Real-analytic AbC constructions on the torus}
\author{Shilpak Banerjee and Philipp Kunde}
\date{}

\graphicspath{{images/}} % Folder where images are stored. Remove if not in sharelatex

\begin{document}

\maketitle

\begin{abstract}
In this article we demonstrate a way to extend the AbC (approximation by conjugation) method invented by Anosov and Katok from the smooth category to the category of real-analytic diffeomorphisms on the torus. We present a general framework for such constructions and prove several results. In particular, we construct minimal but not uniquely ergodic diffeomorphisms and nonstandard real-analytic realizations of toral translations. %Moreover, we build a real-analytic diffeomorphism with disjoint convolutions and a homogeneous spectrum of multiplicity $2$ for its Cartesian square.
\end{abstract}

\tableofcontents

\section{Introduction}
An important question in Smooth Ergodic Theory asks if there are smooth versions to the objects and concepts of abstract ergodic theory. One of the most powerful tools of constructing volume preserving $C^{\infty}$-diffeomorphisms with prescribed ergodic or topological properties on any compact connected manifold $M$ of dimension $m\geq 2$ admitting a non-trivial circle action $\mathcal{S} = \left\{\phi^t\right\}_{t \in \mathbb{S}^1}$ is the so called approximation by conjugation-method developed by D.V. Anosov and A. Katok in their fundamental paper \cite{AK}. These diffeomorphisms are constructed as limits of conjugates $T_n = H^{-1}_n \circ \phi^{\alpha_{n}} \circ H_n$, where $\alpha_{n} = \frac{p_n}{q_n}= \alpha_{n-1} + \frac{1}{s_{n-1} \cdot k_{n-1} \cdot l_{n-1} \cdot q^2_{n-1}} \in \mathbb{Q}$, $H_n = h_n \circ H_{n-1}$ and $h_n$ is a measure-preserving diffeomorphism satisfying $\phi^{\alpha_{n-1}} \circ h_n = h_n \circ \phi^{\alpha_{n-1}}$. In each step the conjugation map $h_n$ and the parameters $k_{n-1}, l_{n-1}$ are chosen such that the diffeomorphism $f_n$ imitates the desired property with a certain precision. Then the parameter $s_{n-1}$ is chosen large enough to guarantee closeness of $f_{n}$ to $f_{n-1}$ in the $C^{\infty}$-topology and so the convergence of the sequence $\left(f_n\right)_{n \in \mathbb{N}}$ to a limit diffeomorphism is provided. This method enables the construction of smooth diffeomorphisms with specific ergodic properties (e. g. weak mixing ones in \cite[section 5]{AK}) or non-standard smooth realizations of measure-preserving systems (e. g. \cite[section 6]{AK} and \cite{FSW}). See also the very interesting survey article \cite{FK} for more details and other results of this method. 

Unfortunately, there are great challenging differences in the real-analytic category as discussed in \cite[section 7.1]{FK}: Since maps with very large derivatives in the real domain or its inverses are expected to have singularities in a small complex neighbourhood, for a real analytic family $S_t$, $0 \leq t \leq t_0$, $S_0 = \text{id}$, the family $h^{-1} \circ S_t \circ h$ is expected to have singularities very close to the real domain for any $t>0$. So, the domain of analycity for maps of our form $f_n = H^{-1}_n \circ \phi^{\alpha_{n}} \circ H_n$ will shrink at any step of the construction and the limit diffeomorphism will not be analytic. Thus, it is necessary to find conjugation maps of a special form which may be inverted more or less explicitly in such a way that one can guarantee analycity of the map and its inverse in a large complex domain. Using very explicit conjugation maps Saprykina was able to construct examples of volume-preserving uniquely ergodic real-analytic diffeomorphims on $\mathbb{T}^2$ (\cite{S}). Fayad and Katok designed such examples on any odd-dimensional sphere in \cite{FK-ue}.

The goal of this article is to reproduce some examples of smooth dynamical systems obtained by the AbC (approximation by conjugation) scheme in the category of real-analytic diffeomorphisms on the torus $\mathbb{T}^d$, $d \geq 2$. For this purpose, we introduce the concept of block-slide type maps on the torus and demonstrate that these maps in a certain sense can be approximated well enough by measure preserving real-analytic diffeomorphisms. This allows us to carry out many AbC constructions in the real-analytic category. We briefly summarise certain previously past results and prove several new ones using real-analytic approximation of block-slide type maps.  Note that all constructions in this article are done on the torus. Real-analytic AbC constructions on arbitrary real-analytic manifolds continue to remain an intractable problem. 

Throughout this article $\T^d$ will denote the $d$ dimensional torus and $\mu $ will stand for the usual Lebesgue measure on $\T^d$. We will use $\text{Diff }^\omega_\rho(\T^d\, \mu )$ to denote the set of real-analytic $\mu$-measure preserving diffeomorphisms of the $d$ dimensional torus whose lift can be extended holomorphically to a complex neighbourhood of diameter at least $\rho$.

First we tackle the problem of non-standard realizations (i.e. to find a diffeomorphism which is metrically but not smoothly isomorphic to a given measure-preserving transformation). The first author used the AbC method and the concept of real-analytic approximation of block-slide type maps to find examples of measure preserving real-analytic, ergodic diffeomorphisms on the torus that are metrically isomorphic to some irrational rotation of the circle. The precise theorem can be stated as follows:

\begin{theorem}[\cite{Ba-Ns}] \label{nsr circle rotation}
For any $\rho>0$ and any integer $d\geq 1$, there exist real-analytic diffeomorphisms $T\in\text{Diff }^\omega_\rho(\T^d\, \mu )$ which are metrically isomorphic to some irrational rotations of the circle.
\end{theorem}

%If we restrict our attention to the two dimensional torus only, there are two possible ways to improve the above result. The original construction by Anosov and Katok was just able to exercise control over almost every orbit of the initial $\T^1$ action, but later the technology evolved and in some cases one was able to exercise control over every orbit of the $\T^1$ action. {\bb To mention: Fathi-Herman, Windsor, Fayad-saprykina-windso, (Saprykina?)} We can use such techniques and prove that the limiting diffeomorphims obtained are in fact uniquely ergodic with respect to the Lebesgue measure (\cite[Theorem 1.1.]{Ba-Ns}): For any $\rho>0$ , there exists uniquely ergodic real-analytic diffeomorphisms $T\in\text{Diff }^\omega_\rho(\T^2, \mu )$ which are metrically isomorphic to some irrational rotations of the circle. We note that minor modifications will extend the above result to the $d$ dimensional torus but we do not do so in this article. 

The diffeomorphisms constructed by Anosov and Katok in \cite[section 4]{AK} realized circle rotations smoothly with Liouvillean rotation numbers. However, it was not clear from this
construction which Liouvillean rotations were realized. Later, Fayad, Saprykina and Windsor extended this result in \cite{FSW} and proved that any Liouvillean rotation of the circle can be realized. In the analytic category, we can not expect realization of every Liouvillean rotation of the circle but we can give a precise description of a subset of some of the Liouvillean rotations we realize. We introduce the set $\mathcal{L}_{\ast}$ of numbers contained in the set of Non-Brjuno numbers: $\mathcal{\alpha} \in \mathbb{R}$ is in $\mathcal{L}_{\ast}$ if for every $k \in \mathbb{N}$ there is $\left( p,q\right) \in \mathbb{Z} \times \mathbb{N}$ with $p,q$ relatively prime satisfying
\begin{align}
\Big|{\alpha - \frac{p}{q}}\Big| < \frac{1}{\mathrm{e}^{\mathrm{e}^{k^q}}}.
\end{align}
In a later section we will examine this set of numbers. In particular, we will show that $\mathcal{L}_{\ast}$ is a dense $G_{\delta}$-subset of $\mathbb{R}$. We will prove,

\begin{maintheorem} \label{nsr circle rotation estimated}
For any $\rho>0$ and every $\a \in \mathcal{L}_{\ast}$ there exists a real-analytic diffeomorphism $T\in\text{Diff }^\omega_\rho(\T^2\, \mu )$ which is metrically isomorphic to the rotation $S_\a$ of the circle.
\end{maintheorem}

In the realm of non-standard realizations, there is another set of questions dedicated to the realization of ergodic translations of a torus on another manifold. In the original paper Anosov and Katok showed that certain ergodic translations on a $d$ dimensional torus can be realized as measure preserving smooth diffeomorphisms on any smooth manifold admitting an effective $\T^1$ action (see \cite[section 6]{AK}). We should note that this result was further improved by Benhenda and it was shown that one can realize any ergodic translation on $\T^d$ with one arbitrary Liouvillean coordinate (see \cite{Mb-ts}).  

It appears that the block-slide type maps allow enough flexibility for us to realize some of these ergodic translations analytically on another torus of arbitrary dimension. We prove,

\begin{maintheorem} \label{theorem nsr total translations}
For any $\rho>0$ and any two integers, $h\geq 1$ and $d\geq 2$, there exists an ergodic real-analytic diffeomorphism $T\in\text{Diff }^\omega_\rho(\T^d)$ which is metrically isomorphic to an ergodic translation of $\T^h$.
\end{maintheorem}

And the obvious corollary follows:

\begin{maintheorem}
For any $\rho>0$ and any two integers, $h\geq 1$ and $d\geq 2$, there exists an ergodic real-analytic diffeomorphism $T\in\text{Diff }^\omega_\rho(\T^d)$ such that $T$ has a discrete spectrum generated (over $\Z$) by $h$ linearly independent eigenvalues.
\end{maintheorem}

There is a conjecture of Kolmogorov in \cite{Kol} stating that on a $d$ dimensional real-analytic manifold an ergodic real-analytic diffeomorphism preserving an analytic measure may have a discrete spectrum with only $d$ distinct eigenvalues. Our result falsifies this conjecture.

Another aspect of the approximation by conjugation scheme deals with the problem of finding diffeomorphisms with a prescribed dynamical property. Originally Anosov and Katok produced examples of measure preserving smooth diffeomorpshims that are weakly mixing on any manifold admitting an effective $\T^1$ action. Later Fayad and Saprykina constructed weakly mixing diffeomorphisms in the restricted space $\mathcal{A}_{\alpha}\left(M\right)= \overline{\left\{h \circ R_{\alpha} \circ h^{-1} \ : h \in \text{Diff}^{\infty}\left(M, \mu \right)\right\}}^{C^{\infty}}$ for every Liouvillean number $\alpha$ (\cite{FS}) in case of dimension $2$. In case of the disc $\mathbb{D}^2$ and the annulus $\mathbb{A}$ this gives the dichotomy that a number is Diophantine if and only if there is not ergodic $C^{\infty}$-diffeomorphism with that rotation number. In that paper \cite{FS}, the authors were even able to construct examples of weakly mixing real analytic diffeomorphims of the two dimensional torus for rotation numbers $\alpha$ that satisfy a condition of similar type as our above one (namely that for some $\delta>0$ the equation $\abs{\alpha- \frac{p}{q}} < \exp(-q^{1+\delta})$ has an infinite number of relatively prime integer solution $p,q$). We should note that the method of reparametrization of linear flows as in \cite{F} is more appropriate to get weakly mixing analytic diffeomorphisms on $\T^d$. 

Using the AbC method and the concept of real-analytic approximation of block-slide type maps, the second author showed that on a torus of any dimension greater than one there are examples of weakly mixing real-analytic diffeomorphims preserving a measurable Riemannian metric.

\begin{theorem}[\cite{Ku-Wm}]
For any $\rho > 0$ and any integer $d \geq  2$, there are weakly mixing real-analytic diffeomorphisms $T \in \text{Diff }^\omega_\rho(\T^d,\mu)$ preserving a measurable Riemannian metric.
\end{theorem}

This result solved \cite{GK}, Problem 3.9., about the existence of real-analytic volume-preserving
IM-diffeomorphisms (i. e. diffeomorphisms preserving an absolutely continuous probability measure
and a measurable Riemannian metric) in the case of tori. In this before-mentioned paper \cite{GK}, Gunesch and Katok constructed volume-preserving weakly mixing $C^{\infty}$-diffeomorphisms preserving a measurable Riemannian metric (see also \cite{KG} for the same result in the restricted spaces $\mathcal{A}_{\alpha}(M)$ for arbitrary Liouville number $\alpha$) and gave a comprehensive consideration of IM-diffeomorphisms and IM-group actions. In particular, the existence of a measurable invariant metric for a diffeomorphism is equivalent to the existence
of an invariant measure for the projectivized derivative extension which is absolutely continuous
in the fibers. Recently, the second author examined the ergodic behaviour of the derivative extension with respect to such a
measure (\cite{Ku-Der}). It provides the only known examples of measure-preserving diffeomorphisms
whose differential is ergodic with respect to a smooth measure in the projectivization of
the tangent bundle. It is an interesting open problem to exhibit such an examination in the real-analytic case.

In another version (called ``toplogical version'' in \cite{FK}) of the AbC-method one tries to exercise control over every orbit of the initial $\T^1$ action, while the original construction by Anosov and Katok was only able to exercise control over almost every orbit of the $\T^1$ action. Such topological constructions deal with minimality and the number of ergodic invariant measures (e. g. unique ergodicity) for intance (see e. g. \cite{FH}, \cite{FSW}, \cite{Win}). We can use such techniques and prove that the limiting diffeomorphims obtained in Theorem \ref{nsr circle rotation} are in fact uniquely ergodic with respect to the Lebesgue measure (\cite[Theorem 1.1.]{Ba-Ns}): For any $\rho>0$ , there exist uniquely ergodic real-analytic diffeomorphisms $T\in\text{Diff }^\omega_\rho(\T^2, \mu )$ which are metrically isomorphic to some irrational rotations of the circle. We note that minor modifications will extend the above result to the $d$ dimensional torus but we do not do so in this article. Instead we can produce more exotic examples. We show that there are minimal but not uniquely ergodic measure preserving real-analytic diffeomorphisms. 

\begin{maintheorem}\label{theorem prescribed no of measures}
For any $\rho > 0$, and any natural number $r$, there exists a real-analytic diffeomorphism $T \in \text{Diff }^\omega_\rho(\T^2,\mu)$ which is minimal and has exactly $r$ ergodic invariant measures each of which are absolutely continuous with respect to the Lebesgue measure.
\end{maintheorem}

This result parallels a result of Windsor in the smooth category (\cite{Win}). While conversely a uniquely ergodic transformation on a compact metric space preserving a Borel measure is minimal on the support of the measure (e.g. \cite{KH}, Proposition 4.1.18), the first example that minimality does not imply unique ergodicity is due to Markov (see \cite{NS}, section 9.35.). In the analytic category Furstenberg constructed skew-products admitting uncountably many ergodic measures (\cite{Fu} or see \cite{KH}, Corollary 12.6.4.). In fact, these counterexamples bear a great meaning in the history of Ergodic Theory: They showed that the so-called \textit{quasi-ergodic hypothesis} (i. e. each orbit is dense in each surface of constant energy) does not imply the equality of space means and time means and so helped to find the right notion of ergodicity. 

In a forthcoming paper the authors use the AbC-method with the real-analytic approximation of block-slide type maps to construct $T \in \text{Diff}^{\omega}_{\rho}\left( \T^d,\mu\right)$ with disjoint convolutions and a homogeneous spectrum of multiplicity $2$ for its Cartesian square $T \times T$. In \cite{Ku-Dc}, $C^{\infty}$-diffeomorphisms in $\mathcal{A}_{\alpha}(M)$ with these properties were constructed.

\section{Preliminaries}

Here we introduce the basic concepts and establish notations that we will use for the rest of this article. 

For a natural number $d$, we will denote the $d$ dimensional torus by $\T^d:=\R^d/\Z^d$. The standard Lebesgue measure on $\T^d$ will be denoted by  $\mu$. We define $\phi$, a measure preserving $\T^1$ action on the torus $\T^d$ as follows:
\begin{align}
\phi^t(x_1,\ldots, x_d)=(x_1+t,x_2,\ldots, x_d)
\end{align}

\subsection{The topology of real-analytic diffeomorphisms on the torus}

We give a description of the space of diffeomorphisms that are interesting to us.
Any real-analytic diffeomorphism on $\T^d$ homotopic to the identity admits a lift to a map from $\R^d$ to $\R^d$ and has the following form 
\begin{align} 
F(x_1,\ldots , x_d)=(x_1+f_1(x_1,\ldots, x_d),\ldots,x_d+f_d(x_1,\dots,x_d))
\end{align}
where $f_i:\R^d\to \R$ are $\Z^d$-periodic real-analytic functions. Any real-analytic $\Z^d$-periodic function defined on $\R^d$ can be extended to some complex neighbourhood \footnote{we identify $\R^d$ inside $\C^d$ via the natural inclusion $(x_1,\ldots , x_d)\mapsto (x_1+i0,\ldots ,x_d+i0)$.} of $\R^d$  as a holomorphic (complex analytic) function. For a fixed $\rho>0$, let
\begin{align}
\Omega_\rho:=\{(z_1,\ldots,z_d)\in\C^d:|\text{Im}(z_1)|<\rho ,\ldots, |\text{Im}(z_d)|<\rho\}
\end{align}
 and for a function $f$ defined on this set, put 
 \begin{align}
 \|f\|_\rho:=\sup_{(z_1,\ldots, z_d)\in\Omega_\rho}|f(z_1,\ldots, z_d)|
 \end{align}
 We define  $C^\omega_\rho(\T^d)$ to be the space of all $\Z^d$-periodic real-analytic functions on $\R^d$ that extends to a holomorphic function on $\Omega_\rho$ and $\|f\|_\rho<\infty$.

We define, $\text{Diff }^\omega_\rho(\T^d,\mu)$ to be the set of all measure preserving real-analytic diffeomorphisms of $\T^d$ homotopic to the identity, whose lift $F(x)=(x_1+f_1(x),\ldots,x_d+f_d(x))$ to $\R^d$ satisfies $f_i\in C^\omega_\rho(\T^d)$ and we also require the lift $\tilde{F}(x)=(x_1+\tilde{f}_1(x),\ldots,x_d+\tilde{f}_d(x))$ of its inverse to $\R^d$ to satisfies $\tilde{f}_i\in C^\omega_\rho(\T^d)$.
The metric $d$ in $\text{Diff }^\omega_\rho(\T^d,\mu)$ is defined by 
\begin{align*}
d_\rho(f,g)=\max\{\tilde{d}_\rho(f,g),\tilde{d}_\rho(f^{-1},g^{-1})\} \qquad\text{where}\qquad \tilde{d}_\rho(f,g)=\max_{i=1,\ldots, d}\{\inf_{n\in\Z}\|f_i-g_i+n\|_\rho\}
\end{align*}
Let $F=(F_1,\ldots, F_d)$ be the lift of a diffeomorphism in $\text{Diff }^\omega_\rho(\T^d,\mu)$, we define the norm of the total derivative
\begin{align*}
\|DF\|_\rho:=\max_{\substack{i=1,\ldots, d\\j=1,\ldots, d}}\Big\|\frac{\partial F_i}{\partial x_j}\Big\|_\rho
\end{align*}

Next, with some abuse of notation, we define the following two spaces 
\begin{align}
C^\omega_\infty (\T^d)  := & \cap_{n=1}^\infty C^\omega_n(\T^d) \label{6.789} \\
\text{Diff }^\omega_\infty (\T^d,\mu)  :=  &  \cap_{n=1}^\infty \text{Diff }^\omega_n(\T^d,\mu) \label{4.569}
\end{align}
Note that the functions in \ref{6.789} can be extended to $\C^n$ as entire functions. We also note that $\text{Diff }^\omega_\infty (\T^d,\mu)$ is closed under composition. To see this, let $f,g\in \text{Diff }^\omega_\infty (\T^d,\mu)$ and $F$ and $G$ be their corresponding lifts. Then note that $F\circ G$ is the lift of $f\circ g$ (with $\pi:\R^2\to\T^2$ as the natural projection, $\pi\circ F\circ G=f\circ\pi\circ G=f\circ g\circ \pi$). Now for the complexification of $F$ and $G$ note that the composition $F\circ G(z)=(z_1+g_1(z)+f_1(G(z)),\ldots, z_d+g_d(z)+f_d(G(z)) )$. Since $g_i\in C^\omega_\infty (\T^d) $, we have for any $\rho,$ $\sup_{z\in\Omega_\rho}|\text{Im}(G(z))|\leq \max_i (\sup_{z\in\Omega_\rho}|\text{Im}(z_i)+\text{Im}(g_i(z))|) \leq \max_i (\sup_{z\in\Omega_\rho}|\text{Im}(z_i)|+\sup_{z\in\Omega_\rho}|\text{Im}(g_i(z))|)\leq \rho + \max_i (\sup_{z\in\Omega_\rho}|g_i(z)|)<\rho + const<\rho'<\infty$ for some $\rho'$. So, $\sup_{z\in\Omega_\rho}|z_i+g_i(z)+f_i(G(z))|\leq |z_i|+|g_i(z)|+|f_i(G(z))|<\infty$ since $z\in\Omega_\rho, g_i\in C^\omega_\infty (\T^d)$ and $G(z)\in \Omega_{\rho'}$. An identical treatment gives the result for the inverse.

All intermediate diffeomorphisms constructed during the AbC method in this paper will belong to this category. \footnote{We note that the existence of such real-analytic functions whose complexification is entire or as in this case, the complexification of their lift is entire is central to a real-analytic AbC method. As of now we only know how to construct such functions on the torus, odd dimensional spheres and certain homogeneous spaces. }

This completes the description of the analytic topology necessary for our construction. Also throughout this paper, the word ``diffeomorphism" will refer to a real-analytic diffeomorphism. Also, the word ``analytic topology" will refer to the topology of $\text{Diff }^\omega_\rho(\T^d,\mu)$ described above. See \cite{S} for a more extensive treatment of these spaces.

\subsection{Some partitions of the torus}

First we recall a few definitions. A sequence of partitions $\{\mathcal{P}_n\}_n$ of a \emph{Lebesgue space}\footnote{Also known as a \emph{standard probability space} or a \emph{Lebesgue-Rokhlin space}. We consider those spaces which are isomorphic mod $0$ to the unit interval with the usual Lebesgue measure.} $(M,\mu)$ is called \emph{generating} if there exists a measurable subset $M'$ of full measure such that $\{x\}=\cap_{n=1}^\infty\mathcal{P}_n(x)\;\;\forall x\in M'$. We say that the sequence $\{\mathcal{P}_n\}_n$ is \emph{monotonic} if $\mathcal{P}_{n+1}$ is a refinement of $\mathcal{P}_n$. 

There are some partitions of $\T^d$ that are of special interest to us. They appear repeatedly in this article and we summarize them here.

Assume that we are given three natural numbers $l,k,q$ and a function $a:\{0,1,\ldots,k-1\}\to\{0,1,\ldots,q-1\}$. We define the following three partitions of $\T^d$:
\begin{align}
& \mathcal{T}_{q}:=\Big\{\Delta_{i,q}:=\big[\frac{i}{q},\frac{i+1}{q}\big)\times \T^{d-1}: i = 0,1,\ldots,q-1\Big\}\label{partition T}\\
& \mathcal{G}_{l,q}:=\Big\{\big[\frac{i_1}{lq},\frac{i_1+1}{lq}\big)\times\big[\frac{i_2}{l},\frac{i_2+1}{l}\big)\times\ldots\times\big[\frac{i_d}{l},\frac{i_d+1}{l}\big):i_1 = 0,1,\ldots,lq-1,\nonumber\\
&\qquad\qquad\qquad\qquad\qquad\qquad\qquad\qquad\qquad\qquad (i_2,\ldots,i_{d}) \in \{0,1,\ldots,l-1\}^{d-1}\Big\}\label{partition G}\\
& \mathcal{G}_{j,l,q}:=\Big\{\big[\frac{i_1}{l^jq},\frac{i_1+1}{l^jq}\big)\times\big[\frac{i_2}{l},\frac{i_2+1}{l}\big)\times\ldots\times\big[\frac{i_{d-j+1}}{l},\frac{i_{d-j+1}+1}{l}\big)\times\T^{j-1}:i_1 = 0,1,\ldots,l^jq-1,\nonumber\\
&\qquad\qquad\qquad\qquad\qquad\qquad\qquad\qquad\qquad\qquad (i_2,\ldots,i_{d-j+1}) \in \{0,1,\ldots,l-1\}^{d-j+1}\Big\}\label{partition Gj}\\
&\mathcal{R}_{a,k,q}:=\Big\{R_{j,q}:=\phi^{j/q}\Big(\bigcup_{i=0}^{k-1}\Delta_{a(i)k+i,kq}\Big), j=0,\ldots,q-1\Big\} \label{partition R}
\end{align} 
We note $\phi^\a$ acts on the partitions \ref{partition T}, \ref{partition G}, \ref{partition Gj} and \ref{partition R} as a permutation for any choice of $p$ when $\a=p/q$.

\subsection{Block-slide type maps and their analytic approximations}

We recall that a \emph{step function} on the unit interval is a finite linear combination of indicator functions on intervals. We define for $1\leq i,j\leq d$ and $i\neq j$, the following piecewise continuous map on the $d$ dimensional torus,
\begin{align}
\mathfrak{h}:\T^d\to\T^d\qquad\text{defined by}\qquad\mathfrak{h}(x_1,\ldots,x_d):=(x_1,\ldots, x_i + s(x_j)\mod 1,\ldots, x_d)
\end{align}
where $s$ is a step function on the unit interval. We refer to any finite composition of maps of the above kind as a \emph{block-slide type of map} on the torus. The nomenclature is motivated from the fact that a finite composition of maps of the above kind has the effect of sliding solid blocks on the torus much like the game of nine.

Inspired by \cite{BK} the purpose of the section is to demonstrate that a block-slide type of map can be approximated extremely well by measure preserving real analytic diffeomorphisms outside a set of arbitrarily small measure. This can be achieved because step function can be approximated well by real analytic functions whose complexification is entire.

\begin{lemma} \label{lemma approx}
Let $k$ and $N$ be two positive integer and $\b=(\b_0,\ldots,\b_{k-1})\in [0,1)^k$. Assume $k$ is even. Consider a step function of the form 
\begin{align*}
\tilde{s}_{\b,N}:[0,1)\to \R\quad\text{ defined by}\quad \tilde{s}_{\b,N}(x)=\sum_{i=0}^{kN-1}\tilde{\b}_i\chi_{[\frac{i}{kN},\frac{i+1}{kN})}(x)
\end{align*}
Here $\tilde{\b}_i:=\b_j$ where $j:=i\mod k$. Then, given any $\e>0$ and $\d>0$, there exists a periodic real-analytic function $s_{\b,N}:\R\to\R$ satisfying the following properties:
\begin{enumerate}
\item Entirety: The complexification of $s_{\b,N}$ extends holomorphically to $\C$.
\item Proximity criterion: $s_{\b,N}$ is $L^1$ close to $\tilde{s}_{\b,N}$. We can say more, \begin{align}\label{nearness}
\sup_{x\in[0,1)\setminus F}|s_{\b,N}(x)-\tilde{s}_{\b,N}(x)|<\e
\end{align}
\item Periodicity: $s_{\b,N}$ is $1/N$ periodic. More precisely, the complexification will satisfy,
\begin{align}\label{boundedness} 
s_{\b,N}(z+n/N)=s_{\b,N}(z)\qquad\forall\; z\in\C\text{ and }n\in\Z
\end{align}
\item Bounded derivative: The derivative is small outside a set of small measure,
\begin{align} \label{derivative bound}
\sup_{x\in[0,1)\setminus F}|s_{\b,N}'(x)|<\e 
\end{align}
\end{enumerate}
Where $F=\cup_{i=0}^{kN-1}I_i\subset [0,1)$ is a union of intervals centred around $\frac{i}{kN},\;i=1,\ldots, kN-1$ and $I_0=[0,\frac{\d}{2kN}]\cup[1-\frac{\d}{2kN},1)$ and $\l(I_i)=\frac{\d}{kN}\;\forall\; i$. 
\end{lemma}

\begin{proof}
See \cite[Lemma 4.7]{Ba-Ns} and \cite[Lemma 3.6]{Ku-Wm}.
\end{proof}

Note that the condition \ref{boundedness} in particular implies 
\begin{align*}
\sup_{z: \text{Im}(z)<\rho}s_{\b,N}(z)<\infty\quad\forall\; \rho>0
\end{align*}
Indeed, for any $\rho>0$, put $\Omega'_\rho=\{z=x+iy:x\in [0,1], |y|<\rho\}$ and note that entirety of $s_{\b,N}$ combined with compactness of $\overline{\Omega'_\rho}$ implies $\sup_{z\in\Omega'_\rho}|s_{\b,N}(z)|<C$ for some constant $C$. Periodicity of $s_{\a,N}$ in the real variable and the observation $\Omega_\rho=\cup_{n\in\Z}\left(\Omega'_\rho + n\right)$ implies that $\sup_{z\in\Omega_\rho}|s_{\b,N}(z)|<C$. We have essentially concluded that  $s_{\b,N}\in C^\omega_\infty(\T^1)$. 

We also make the observation that the condition $k$ is even is not really a necessary one. One can drop the condition after replacing $F$ with a different error set. 

In order to prove convergence with a prescribed rotation number we require the following refinement of the lemma:

\begin{lemma} \label{lem:approx}
Let $l, N \in \mathbb{N}$, $l$ even, and $\beta=\left(\beta_0,...,\beta_{l-1}\right) \in \left[0,1\right]^l$. We consider a step function of the form
\begin{equation*}
\tilde{s}_{\beta,N}:\left[0,1\right) \rightarrow \mathbb{R} \text{ defined by } \tilde{s}_{\beta,N}(x) = \sum^{lN-1}_{i=0} \tilde{\beta}_i \cdot \chi_{\left[\frac{i}{lN},\frac{i+1}{lN}\right)}(x),
\end{equation*}
where $\tilde{\beta}_i \coloneqq \beta_j$ in case of $j \equiv i \mod l$. Given any $\varepsilon \in \left( 0, \frac{1}{8} \right)$ and $\delta\in (0,1)$ let the number $A>0$ fulfil the conditions
\begin{equation} \tag{A1} \label{eq:app1}
A > - \frac{2l}{\pi \cdot \delta} \cdot \ln \left(- \ln \left( 1-\frac{\varepsilon}{8} \right) \right)
\end{equation}
and
\begin{equation} \tag{A2} \label{eq:app2}
A> \frac{2l}{\pi \cdot \delta} \cdot \ln \left(-\ln\left(\frac{\varepsilon}{2l}\right)\right).
\end{equation}
Then the $\frac{1}{N}$-periodic real entire function $s_{\beta, N, \varepsilon, \delta}$ given by
\begin{align*}
& s_{\beta, N, \varepsilon, \delta}(z) =  \\
& \left(\sum^{\frac{l}{2}-1}_{i=0} \beta_i \cdot \left(\ee^{-\ee^{-A \cdot \sin\left(2\pi\left(Nz- \frac{i}{l}\right)\right)}} - \ee^{-\ee^{-A \cdot \sin\left(2\pi\left(Nz- \frac{i+1}{l}\right)\right)}}\right)\right) \cdot \ee^{-\ee^{-A \cdot \sin\left(2\pi Nz\right)}} \\
& +\left(\sum^{l-1}_{i=\frac{l}{2}} \beta_{i} \cdot \left(\ee^{-\ee^{-A \cdot \sin\left(2\pi\left(Nz- \frac{i}{l}\right)\right)}} - \ee^{-\ee^{-A \cdot \sin\left(2\pi\left(Nz- \frac{i+1}{l}\right)\right)}}\right)\right) \cdot \ee^{-\ee^{A \cdot \sin\left(2\pi Nz\right)}}.
\end{align*}
satisfies
\begin{equation} \label{eq:condapprox}
\sup_{x \in [0,1) \setminus F} \left| s_{\beta, N, \varepsilon, \delta}(x) - \tilde{s}_{\beta,N}(x) \right| < \varepsilon,
\end{equation}
where $F= \bigcup^{lN-1}_{i=0} I_i \subset [0,1)$ is a union of intervals centered around $\frac{i}{lN}$, $i=1,...,lN-1$, $I_0 = \left[0, \frac{\delta}{2lN}\right]\cup \left[1-\frac{\delta}{2lN},1\right)$ and $\lambda\left(I_i\right)= \frac{\delta}{lN}$ for every $i$.
\end{lemma}

\begin{proof}
First of all, we point out that $s_{\beta, N, \varepsilon, \delta}$ is a $\frac{1}{N}$-periodic real entire function. Let $x \in [0,1) \setminus F$, namely $x \in \left[ \frac{j}{lN}+\frac{\delta}{2lN}, \frac{j+1}{lN}-\frac{\delta}{2lN} \right]$ for some $j \in \mathbb{Z}$, $0 \leq j \leq lN-1$. We write $x=\frac{j}{lN}+ \Delta$. Exploiting the fact $\sin(x) > \frac{x}{2}$ for $0<x < \frac{\pi}{2}$ we get 
\begin{equation*}
\ee^{-\ee^{-A \cdot \sin(2\pi N \Delta)}} \geq \ee^{-\ee^{-A \cdot \pi N \Delta}}.
\end{equation*}
Using equation \ref{eq:app1} this implies
\begin{equation} \label{eq:est1}
\ee^{-\ee^{-A \cdot \sin\left( 2 \pi \left( Nx- \frac{s}{l} \right) \right)}}> 1 - \frac{\varepsilon}{8}
\end{equation}
in case of $0 \leq j-s < \frac{l}{2}$ or $-l < j-s < -\frac{l}{2}$. On the other hand, we use the fact $\sin(x) < \frac{x}{2}$ for $-\frac{\pi}{2} < x < 0$ and get
\begin{equation*}
\ee^{-\ee^{-A \cdot \sin(-2\pi N \Delta)}} \leq \ee^{-\ee^{A \cdot \pi N \Delta}}.
\end{equation*}
By applying condition \ref{eq:app2} this yields
\begin{equation} \label{eq:est2}
\ee^{-\ee^{-A \cdot \sin \left( 2 \pi \left( Nx- \frac{s}{l}\right)\right)}}< \frac{\varepsilon}{2l}
\end{equation}
in case of $\frac{l}{2} \leq j-s < l$ or $-\frac{l}{2} \leq j-s <0$. By the above estimates in equation \ref{eq:est1} and \ref{eq:est2} we get
\begin{equation*}
s_{\beta, N, \varepsilon, \delta}(x) \geq \beta_j \cdot \left( \left( 1- \frac{\varepsilon}{8} \right) \cdot \left( 1- \frac{\varepsilon}{8} \right) - \frac{\varepsilon}{2l} \right)  - (l -1 ) \cdot \frac{\varepsilon}{2l} 
\end{equation*}
and 
\begin{equation*}
s_{\beta, N, \varepsilon, \delta}(x) \leq \beta_j + (l -1 ) \cdot \frac{\varepsilon}{2l}.  
\end{equation*}
Altogether, we conclude
\begin{equation*}
\left| s_{\beta, N, \varepsilon, \delta}(x) - \beta_j \right| < \varepsilon.
\end{equation*}
\end{proof}

Finally we piece together everything and demonstrate how a block-slide type of map on the torus can be approximated by a measure preserving real-analytic diffeomorphism.

\begin{proposition} \label{proposition approximation}
Let $\mathfrak{h}:\T^d\to\T^d$ be a block-slide type of map which commutes with $\phi^{1/q}$ for some natural number $q$. Then for any $\e>0$ and $\d>0$, there exists a real-analytic diffeomorphism $h\in\text{Diff }^\omega_\infty(\T^d,\mu)$ satisfying the following conditions:
\begin{enumerate}
\item Proximity property: There exists a set $E\subset\T^d$ such that $\mu(E)<\d$ and $\sup_{x\in\T^d\setminus E}\|h(x)-\mathfrak{h}(x)\|<\e$. 
\item Commuting property: $h\circ\phi^{1/q}=\phi^{1/q}\circ h$
\end{enumerate} 
In this case we say the the diffeomorphism $h$ is $(\e,\d)$-close to the block-slide type map $\mathfrak{h}$. 
\end{proposition}

\begin{proof}
First we assume that for some step function $\tilde{s}_{\b,1}$ and any integer $i$ with $1<i\leq d$, the block-slide map $\mathfrak{h}$ is of the following type:
\begin{align}
\mathfrak{h}:\T^d\to\T^d\qquad\qquad\text{defined by}\qquad\mathfrak{h}(x)=(x_1+\tilde{s}_{\b,1}(x_i),x_2,\ldots,x_d)
\end{align}
Then we define the following function using $s_{\b,1}$ as in lemma \ref{lemma approx}:
\begin{align}
h:\T^d\to\T^d\qquad\qquad\text{defined by}\qquad h(x)=(x_1+s_{\b,1}(x_i),x_2,\ldots,x_d)
\end{align}
With $F$ as lemma \ref{lemma approx}, we put $E=\T^{i-1}\times F\times \T^{d-i}$ and observe that $h$ satisfies all the conditions of the proposition.

Similarly, if for some step function $\tilde{s}_{\b,q}$ and any integer $i$ with $1<i\leq d$, we have a block-slide map $\mathfrak{h}$ of the following type:
\begin{align}
\mathfrak{h}:\T^d\to\T^d\qquad\qquad\text{defined by}\qquad\mathfrak{h}(x)=(x_1,\ldots,x_{i-1},x_i+\tilde{s}_{\b,q}(x_1),x_{i+1},\ldots,x_d)
\end{align}
Then we define the following function using $s_{\b,q}$ as in lemma \ref{lemma approx}:
\begin{align}
h:\T^d\to\T^d\qquad\qquad\text{defined by}\qquad h(x)=(x_1,\ldots,x_{i-1},x_i+s_{\b,q}(x_1),x_{i+1},\ldots,x_d)
\end{align}
With $F$ as lemma \ref{lemma approx}, we put  $E=F\times \T^{d-1}$ and observe that $h$ satisfies all the conditions of the proposition.

So for a general block-slide type map which is obtained by a composition of several maps of the above type, we just take a composition of the individual approximations and a union of all the component error sets.
\end{proof}

\subsection{Analytic AbC method} \label{subsection abc method}

Our objective now is to recall the approximation by conjugation scheme developed by Anosov and Katok in \cite{AK}. Though we modify this scheme slightly to be more suitable for our purpose and fit the notations of our article we insist that the method presented here is almost identical to the original construction. In more modern works, this method is often presented in a less formal way (see \cite{FK}) avoiding most technicalities but for our purpose we find the original scheme to be most suitable and we stick close to it.

The AbC method is an inductive process where a sequence of diffeomorphisms $T_n\in\text{Diff }^\omega_\infty(\T^d,\mu)$ is constructed inductively. The diffeomorphisms $T_n$ converge to some diffeomorphism $T$ $\in$ $\text{Diff }^\omega_\rho(\T^d,\mu)$. Additionally $T_n$ s are chosen carefully so that they satisfy some finite version of the desired property of $T$. 

We now give an explicit description. At the beginning of the construction we fix a constant $\rho>0$ and note that all parameters chosen will depend on this $\rho$. 

Assume that the construction has been carried out up to the $n$ th stage and we have the following information available to us:

\begin{enumerate}
\item We have sequences of natural numbers $\{p_m\}_{m=1}^n$, $\{q_m\}_{m=1}^n$, $\{k_m\}_{m=1}^{n-1}$, $\{l_m\}_{m=1}^{n-1}$, $\{s_m\}_{m=1}^{n-1}$, a sequence of functions $\{a_m:\{0,\ldots, k_m\}\to\{0,\ldots, q_m-1\}\}_{m=1}^{n-1}$ and a sequence of numbers $\{\e_m\}_{m=1}^n$ . They satisfy the following condition:
\begin{align}
p_{m}=s_{m-1}k_{m-1}l_{m-1}q_{m-1}p_{m-1} + 1\qquad\quad q_{m}=s_{m-1}k_{m-1}l_{m-1}q_{m-1}^2\qquad\quad \e_m< 2^{-q_m}
\end{align} 

\item The sequence of diffeomorphisms $\{T_m\}_{m=1}^n$ is constructed as conjugates of a periodic translation. More precisely,
\begin{align}
T_m:=H_m^{-1}\circ\phi^{\a_{m}}\circ H_m\qquad\qquad H_m:=h_m\circ H_{m-1}\qquad\qquad h_m\in\text{Diff }^{\omega}_\infty(\T^d,\mu)
\end{align}
The diffeomorphisms $\{h_{m}\}_{m=1}^n$ satisfy the following commuting condition:
\begin{align}
h_{m}\circ\phi^{\a_{m-1}}=\phi^{\a_{m-1}}\circ h_{m}
\end{align}

\item For $m =1,\ldots, n$, the diffeomorphism $T_m$ preserves and permutes two sequences of partitions, namely, $H_m^{-1}\mathcal{R}_{a_m,k_m,q_m}$ and $\mathcal{F}_{q_m}:=H_m^{-1}\mathcal{T}_{q_m}$.

\item For $m =1,\ldots, n$, $\mu(h_{m}^{-1}R_{i,q_{m-1}}\triangle\Delta_{i,q_{m-1}})<\e_{m-1}$ for any $R_{i,q_{m-1}}\in\mathcal{R}_{a_{m-1},k_{m-1},q_{m-1}}$ and $\Delta_{i,q_{m-1}}\in \mathcal{T}_{q_{m-1}}$ with the same $i$.

\item For $m =1,\ldots, n$, $\text{diam} (\mathcal{F}_{q_m}\cap E_{m})<\e_m$ \footnote{ This means that the diameter of the intersection of any atom of $\mathcal{F}_{q_m}$ and $E_m$ is less that $\e_m$.} for some measurable set $E_m$ satisfying $\mu(E_m)>1-\e_m$. (Note that this means $\mathcal{F}_{q_m}$  is a generating but not necessarily monotonic sequence of partitions.)

\item For $m =1,\ldots, n$: $d_\rho(T_m,T_{m-1})<\e_m$ and $d_0 \left( T^i_m, T^i_{m-1} \right) < \frac{1}{2^{m-1}}$ for $0 \leq i < q_{m-1}-1$.
\end{enumerate}

Now we show how to do the construction at the $n+1$ th stage of this induction process. We proceed in the following order:
\begin{enumerate}
\item We choose $k_n $ and our function $a_n:\{0,\ldots, k_n\}\to\{0,\ldots, q_n-1\}$. This choice will depend on the construction we are doing and the specific properties we are targeting to prove.

\item We choose $l_n$ to be a large enough integer so that the following condition is satisfied:
\begin{align}\label{ln criterion}
l_n>2^n\|DH_n\|_0
\end{align}

\item Find a block-slide type map $\mathfrak{h}_{a_n,k_n,l_n,q_n}$ which commutes with $\phi^{\a_n}$, maps the partition $\mathcal{G}_{l_nk_n,q_n}$ to $\mathcal{T}_{l_n^dk_n^d,q_n}$ and it maps the partition $\mathcal{T}_{q_n}$ to the partition $\mathcal{R}_{a_n,k_n,q_n}$. 

\item Use proposition \ref{proposition approximation} to construct $h_{n+1}$ which is $(\e_n/2^{l_nk_nq_n},\e_n/2^{l_nk_nq_n})$ close to $\mathfrak{h}_{a_n,k_n,l_n,q_n}$. Put $E_n$ to be the error set in proposition \ref{proposition approximation}.

\item Ensure $|\a_{n+1}-\a_n|$ is small enough to guarantee $d_\rho(T_{n+1},T_{n})<\e_m$ and $d_0 \left( T^i_{n+1}, T^i_{n} \right) < \frac{1}{2^{n}}$ for $0 \leq i < q_{n}-1$. If either $l_n$ or $k_n$ above is chosen to be very large and this condition is satisfied, we put $s_n=1$. If our choice of $l_n$ or $k_n$ is too restrictive then we choose $s_n$ to be large enough so that convergence is guaranteed.  

\end{enumerate}
This completes the construction at the $n+1$ th stage. Note that this way convergence of $T_n$ to some $T\in\text{Diff }^\omega_\rho(\T^d,\mu)$ is guaranteed.  

\begin{remark} \label{close iterates}
By $d_0 \left( T^i_{n+1}, T^i_{n} \right) < \frac{1}{2^{n}}$ for $0 \leq i < q_{n}-1$ and every $n \in \N$ we get $d_0 \left(T^i, T^i_{n+1} \right) < \frac{1}{2^n}$ for $0 \leq i < q_{n+1}-1$. 
\end{remark}

We need another important constructions which is very handy for some application. Note that the partitions $\mathcal{F}_{q_n}$ are not necessarily monotonic. However the following proposition shows that a generating monotonic partition which is cyclically permuted by $T_n$ can be constructed from $\mathcal{F}_{q_n}$. This is identical to proposition 3.1 in \cite{AK}.

\begin{proposition} \label{proposition monotonic generating cyclic partition}
With notations as in the approximation by conjugation scheme, we can find a sequence of partitions $\mathcal{M}_n$ of $\T^d$ satisfying the following three properties:
\begin{enumerate}
\item Monotonicity condition: $\mathcal{M}_{n+1}>\mathcal{M}_n$
\item Cyclic permutaion: The diffeomorphims $T_n$ cyclically permutes the atoms of $\mathcal{M}_n$.
\item Generating condition: $\mathcal{M}_n\to\e$ as $n\to\infty$. 
\end{enumerate}
\end{proposition}

\begin{proof}
For any $n$ we define a measurable map  $\mathfrak{c}_{n+1,n}^{(\mathfrak{1})}:\T^d/\mathcal{T}_{q_n}\to \T^d/\mathcal{R}_{a_n,k_n,q_n}$ by  $\mathfrak{c}_{n+1,n}^{(\mathfrak{1})}(\Delta_{i,q_n})=R_{i,q_n}$ and another measurable map $\mathfrak{c}_{n+1,n}^{(\mathfrak{2})}:\T^d/\mathcal{T}_{q_{n+1}}\to\T^d/\mathcal{R}_{a_n,k_n,q_n}$ by $\mathfrak{c}_{n+1,n}^{(\mathfrak{2})}(\Delta_{i,q_{n+1}})=R_{j,q_n}$   where $j$ is such that $\Delta_{i,q_{n+1}}\subset R_{j,q_n}$. So at the level of $\T^d$ the composition $\mathfrak{c}_{n+1,n}=(\mathfrak{c}_{n+1,n}^{(\mathfrak{2})})^{-1}\circ\mathfrak{c}_{n+1,n}^{(\mathfrak{1})}$ gives us a correspondence which associates each atom of $\mathcal{T}_{q_n}$ with a union of atoms of $\mathcal{T}_{q_{n+1}}$. So we can define a new partition $\mathcal{T}_{q_{n+1},q_{n}}:=\{\mathfrak{c}_{n+1,n}(\Delta_{i,q_n}): 0\leq i<q_n\}$. Continuing this procedure, we obtain for any $m>n$, a partition  $\mathcal{T}_{q_m,q_n}:=\{\mathfrak{c}_{m,m-1}\circ\ldots\circ\mathfrak{c}_{n+2,n+1}\circ\mathfrak{c}_{n+1,n}(\Delta_{i,q_n}): 0\leq i<q_n\}$. We note that this partition satisfies the following three conditions for any three integers $n$, $l$ and $m$ with $m>l>n$:
\begin{enumerate}
\item Monotonicity condition: $\mathcal{T}_{q_{m},q_{n+1}}>\mathcal{T}_{q_m,q_n}$ 
\item Cyclic permutation: $\phi^{1/q_n}$ cyclically permutes the atoms of $\mathcal{T}_{q_m,q_n}$. 
\item $\mathfrak{c}_{m,l}\circ\mathfrak{c}_{l,n}=\mathfrak{c}_{m,n}$.
\end{enumerate} 
Now we define two new partitions as follows:
\begin{align}
& \mathcal{F}_{q_m,q_n}:=H_m^{-1}\mathcal{T}_{q_m,q_n}\\
& \mathcal{F}_{q_n}:=H_n^{-1}\mathcal{T}_{q_n}
\end{align}
We define the correspondence $\mathfrak{p}_{m,n}:\mathcal{F}_{q_n}\to\mathcal{F}_{q_m,q_n}$ by $\mathfrak{p}_{m,n}(H_n^{-1}(\Delta_{i,q_n}))=H_m^{-1}(\mathfrak{c}_{m,n}(\Delta_{i,q_n}))$. Now note that for any three integers $m>l>n$ we have the following three properties:
\begin{enumerate}
\item Monotonicity: $\mathcal{F}_{q_{m},q_{n+1}}>\mathcal{F}_{q_m,q_n}$.
\item Cyclic permutation: $T_n$ cyclically permutes the atoms of $\mathcal{F}_{q_m,q_n}$ (since $\phi^{1/q_n}$ commutes with $h_j$ for $j>n$).
\item $\mathfrak{p}_{m,l}\circ\mathfrak{p}_{l,n}$ $=\mathfrak{p}_{m,n}$. Indeed, for any $\Delta_{i,q_n}$ we observe,
\begin{align*}
\mathfrak{p}_{m,l}\circ\mathfrak{p}_{l,n}(H_n^{-1}(\Delta_{i,q_n}))= &\; \mathfrak{p}_{m,l}(H_l^{-1}(\mathfrak{c}_{l,n}(\Delta_{i,q_n})))\\
= &\; H_m^{-1}(\mathfrak{c}_{m,l}(\mathfrak{c}_{l,n}(\Delta_{i,q_n}))))\\
= &\; H_m^{-1}(\mathfrak{c}_{m,n}(\Delta_{i,q_n}))\\
= &\; \mathfrak{p}_{m,n}(H_n^{-1}(\Delta_{i,q_n}))
\end{align*}
\end{enumerate}
Our next goal is to prove that the limit 
\begin{align}\label{6.567}
\mathfrak{p}_{\infty,n}(H_n^{-1}(\Delta_{i,q_n}))):=\lim_{m\to\infty}(\mathfrak{p}_{m,n}(H_n^{-1}(\Delta_{i,q_n})))
\end{align}
exists for any $n$ and $i$. In order to see this we note at first that $\mathfrak{c}_{m,n}(\Delta_{i,q_n})=\cup_{l\in\sigma}\Delta_{l,q_m}$ where $\sigma$ is some indexing set of size $q_m/q_n$. This implies $\mathfrak{c}_{m+1,n}(\Delta_{i,q_n})=\cup_{l\in\sigma}R_{l,q_m}$. And hence $h_{m+1}^{-1}(\mathfrak{c}_{m+1,n}(\Delta_{i,q_n}))=\cup_{l\in\sigma}h_{m+1}^{-1}(R_{l,q_n})$. Now note the following estimates:
\begin{align*}
& \; \mu(h_{m+1}^{-1}(R_{l,q_m})\triangle \Delta_{l,q_m})<\e_m\\
\Rightarrow & \; \mu(\bigcup_{l\in\sigma}h_{m+1}^{-1}(R_{l,q_m})\triangle \Delta_{l,q_m})<\frac{q_m}{q_n}\e_m\\
\Rightarrow & \; \mu(\bigcup_{l\in\sigma}h_{m+1}^{-1}(R_{l,q_m})\triangle \bigcup_{l\in\sigma}\Delta_{l,q_m})<\frac{q_m}{q_n}\e_m\\
\Rightarrow & \; \mu(h_{m+1}^{-1}(\mathfrak{c}_{m+1,n}(\Delta_{i,q_n}))\triangle \mathfrak{c}_{m,n}(\Delta_{i,q_n})<\frac{q_m}{q_n}\e_m\\
\Rightarrow & \; \mu(H_{m+1}^{-1}(\mathfrak{c}_{m+1,n}(\Delta_{i,q_n}))\triangle H_m^{-1}(\mathfrak{c}_{m,n}(\Delta_{i,q_n}))<\frac{q_m}{q_n}\e_m\\
\Rightarrow & \; \mu(\mathfrak{p}_{m+1,n}(H_n^{-1}(\Delta_{i,q_n}))\triangle \mathfrak{p}_{m,n}(H_n^{-1}(\Delta_{i,q_n}))<\frac{q_m}{q_n}\e_m\\
\Rightarrow & \; \sum_{i=0}^{q_n-1}\mu(\mathfrak{p}_{m+1,n}(H_n^{-1}(\Delta_{i,q_n}))\triangle \mathfrak{p}_{m,n}(H_n^{-1}(\Delta_{i,q_n}))<q_m\e_m\\
\Rightarrow & \; \sum_{i=0}^{q_n-1}\mu(\mathfrak{p}_{m_1,n}(H_n^{-1}(\Delta_{i,q_n}))\triangle \mathfrak{p}_{m_2,n}(H_n^{-1}(\Delta_{i,q_n}))<\sum_{m=m_1}^{m_2}q_m\e_m
\end{align*}
This shows that the sequence $\mathcal{F}_{q_m,q_n}$ converges as $m$ goes to infinity and shows the existence of \ref{6.567}. So we define the partition 
\begin{align}
\mathcal{M}_n:=\{\mathfrak{p}_{\infty,n}(H_n^{-1}(\Delta_{i,q_n}))):0\leq i<q_n\}
\end{align}
And we note that this partition has the required properties.
\end{proof}

\section{Non standard analytic realization of some ergodic rotations of the circle}

Non-standard realization problems demonstrate how a dynamical system can live on a non native manifold. In particular we are interested in exploring when an ergodic rotation of the circle can be measure theoretically isomorphic to a measure preserving real-analytic ergodic diffeomorphisms on a torus.

\subsection{Some measure theory}

Our goal here is to prove an abstract lemma from measure theory which is a slight generalization of Lemma 4.1 in \cite{AK}. Essentially we formulate a finite version of the conjugacy we will eventually prove.

Lemma 4.1 from \cite{AK} gave us an easily checkable finite version of the conjugacy that one can use to prove the existence of a metric isomorphism of the limiting diffeomorphisms. Since the generating partitions used in the $C^\infty$ non-standard realization problem can easily made to be monotonic, this Lemma was sufficient. But in the real-analytic case, our construction is not flexible enough to guarantee monotonicity, so we need a modified version. Let us recall Lemma 4.1 from \cite{AK} since we will need it for our version.

\begin{lemma} \label{4.1}
Let $\{M^{(i)},\mu^{(i)}\},\;i=1,2$ be two Lebesgue spaces. Let $\mathcal{P}_n^{(i)}$ be a \emph{monotonic} sequences of generating finite partitions of $M^{(i)}$. Let $T_n^{(i)}$ be a sequence of automorphisms of $M^{(i)}$ satisfying $T_n^{(i)}\mathcal{P}_n^{(i)}=\mathcal{P}_n^{(i)}$ and suppose $ \lim_{n\to\infty}T_n^{(i)}=T^{(i)}$ weakly.
Suppose that there are  metric isomorphisms $K_n:M^{(1)}/\mathcal{P}_n^{(1)}\to M^{(2)}/\mathcal{P}_n^{(2)}$ satisfying:
\begin{align}
& K_n^{-1}T_n^{(2)}|_{\mathcal{P}_n^{(2)}}K_n=T_n^{(1)}|_{\mathcal{P}_n^{(1)}}\\
& K_{n+1}(\mathcal{P}_{n}^{(1)})=K_n(\mathcal{P}_{n}^{(1)})
\end{align}
Then the automorphisms $T^{(1)}$ and $T^{(2)}$ are metrically isomorphic.
\end{lemma}
We would also like to point out at this point of time that the metric isomorphism $K$ in the proof was defined to be
\begin{align}\label{K(x)}
K(x):=\cap_{n=1}^\infty K_n(P_n^{(1)}(x))\qquad\text{a.e.}\;\; x\in M^{(1)}
\end{align} 

Now we prove the following variation which will allow us to accommodate a marginal ``twist" that will appear in our construction.

\begin{lemma} \label{lemma mtl}
Let $\{M^{(i)},\mu^{(i)}\},\;i=1,2$ be two Lebesgue spaces. Let $\mathcal{P}_n^{(i)}$ be a sequence of generating finite partitions of $M^{(i)}$. Let $\{\e_n\}$ be a sequence of positive numbers satisfying $\sum_{n=1}^\infty\e_n<\infty$. In addition, assume that there exists a sequence of sets $\{E_n^{(i)}\}$ in $M^{(i)}$ satisfying:
\begin{align}
& \mu^{(1)}(E_n^{(1)})=\mu^{(2)}(E_n^{(2)})<\e_n\label{mtl 1}\\
& P_{n+1}^{(1)}(x)\setminus E_{n+1}^{(1)}\subset P_{n}^{(1)}(x)\quad\forall\; x\in M^{(1)}\setminus E_{n+1}^{(1)}\label{mtl 2}\\
& P_{n+1}^{(2)}(y)\subset P_{n}^{(2)}(y)\quad \forall\; y\in M^{(2)}\label{mtl 3}
\end{align}
 Let $T_n^{(1)}$ and $T_n^{(2)}$ be two sequences of automorphisms of the spaces $M^{(1)}$ and $M^{(2)}$ satisfying:
\begin{align}
& T_n^{(i)}\mathcal{P}_n^{(i)}=\mathcal{P}_n^{(i)}\quad\quad i=1,2\label{mtl 4}\\
& \lim_{n\to\infty}T_n^{(i)}=T^{(i)}\quad\quad i=1,2\label{mtl 5}\\
& T_n^{(i)}(\cup_{m=n}^\infty E_m^{(i)})=\cup_{m=n}^\infty E_m^{(i)}\quad\quad i=1,2\label{mtl 6}
\end{align}
Note that the limit in \ref{mtl 5} is taken in the weak topology. Suppose additionally that there exists a sequence of metric isomorphisms $K_n:M^{(1)}/\mathcal{P}_n^{(1)}\to M^{(2)}/\mathcal{P}_n^{(2)}$ satisfying:
\begin{align}
& K_n^{-1}T_n^{(2)}|_{\mathcal{P}_n^{(2)}}K_n=T_n^{(1)}|_{\mathcal{P}_n^{(1)}}\label{mtl 7}\\
& K_{n+1}(\mathcal{P}_{n+1}^{(1)}(x))\subset K_{n}(\mathcal{P}_{n}^{(1)}(x))\quad\forall\; x\in M^{(1)}\setminus E_{n+1}^{(1)}\label{mtl 8}
\end{align}
Then the automorphisms $T^{(1)}$ and $T^{(2)}$ are metrically isomorphic.
\end{lemma}

\begin{proof}
Put $F^{(i)}_N:=\cup_{n=N}^\infty E_{n}^{(i)}$. Consider the sequence of Lebesgue spaces $M^{(i)}_N:=M^{(i)}\setminus F^{(i)}_N$.

\vspace{2mm}

\noindent \emph{Claim 1: $\exists$ a metric isomorphism $K_{(N)}:M^{(1)}_N\to M^{(2)}_N$, satisfying $K_{(N)}^{-1}T^{(2)}|_{M_N^{(2)}}K_{(N)}=T^{(1)}|_{M_N^{(1)}}$.}

We define $\mathcal{P}^{(i)}_{N,k}$, a finite measurable partition of $M^{(i)}_N$ by $\mathcal{P}^{(i)}_{N,k}(x):=\mathcal{P}^{(i)}_{N+k}(x)\setminus F^{(i)}_N$. We note that the sequence of partition $\{\mathcal{P}^{(i)}_{N,k}\}_{k}$ is generating because $\{\mathcal{P}^{(i)}_{k}\}_k$ is generating. Additionally, condition \ref{mtl 2} makes $\{\mathcal{P}^{(i)}_{N,k}\}_k$ a monotonic sequence of partition. We define $K_{N,k}(\mathcal{P}^{(1)}_{N,k}(x)):=K_{N+k}(\mathcal{P}^{(1)}_{N+k}(x))\setminus F^{(2)}_{N}$. 
With this definition we claim that $K_{N,k+1}(\mathcal{P}^{(1)}_{N,k}(x))=K_{N,k}(\mathcal{P}^{(1)}_{N,k}(x))$ $\forall x\in M^{(1)}_N$. (Indeed, from \ref{mtl 8} we get for a.e. $x\in M^{(1)}_N$, $K_{N,k+1}(\mathcal{P}^{(1)}_{N,k+1}(x))=K_{N+k+1}(\mathcal{P}^{(1)}_{N+k+1}(x))\setminus F^{(2)}_{N}\subset K_{N+k}(\mathcal{P}^{(1)}_{N+k}(x))\setminus F^{(2)}_N=K_{N,k}(\mathcal{P}^{(1)}_{N,k}(x))$. This with the fact that $K_{N,k+1}(\mathcal{P}^{(1)}_{N,k+1}(x))\in \mathcal{P}^{(2)}_{N,k+1}$ and $\{\mathcal{P}^{(2)}_{N,k}\}_k$ is a monotonic sequence of partitions helps us in concluding the claim). 
Observe that \ref{mtl 4}, \ref{mtl 6} and \ref{mtl 7} guarantees that $K_{N,k}^{-1}T_{N+k}^{(2)}|_{\mathcal{P}_{N,k}^{(2)}}K_{N,k}=T_{N+k}^{(1)}|_{\mathcal{P}_{N,k}^{(1)}}$. So we can apply Lemma \ref{4.1} to guarantee the existence of a metric isomorphism $K_{(N)}:M^{(1)}_N\to M^{(2)}_N$ defined for a.e. $x\in M^{(1)}_N$ by $K_{(N)}(x)=\cap^\infty_{k=1} K_{N,k}(\mathcal{P}^{(1)}_{N,k}(x))$. This finishes claim 1. 

\vspace{2mm}

\noindent\emph{Claim 2: $K_{(N+1)}(x)=K_{(N)}(x)$ for a.e. $x\in M^{(1)}_{N}$}

Follows from the definition of $K_{(N)}$. Indeed, note that $K_{(N+1)}(x)=\cap^\infty_{k=1} K_{N+1,k}(\mathcal{P}^{(1)}_{N+1,k}(x))=\cap^\infty_{k=1}K_{N+k+1}(\mathcal{P}^{(1)}_{N+k+1}(x))\setminus F^{(2)}_{N+1}=\cap^\infty_{k=0} K_{N+k+1}(\mathcal{P}^{(1)}_{N+k+1}(x))\setminus F^{(2)}_{N+1}$. The last equality follows from \ref{mtl 8}.

\vspace{2mm}

\noindent\emph{Claim 3: There exists a metric isomorphism $K:M^{(1)}\to M^{(2)}$ satisfying $K^{-1} T^{(2)}  K=T^{(1)}$}

Note that condition \ref{mtl 1} implies that a.e. $x\in E_n^{(1)}$ for at most finitely many $n$. Indeed, if $E=\{x:x\in E_n\text{ for infinitely many } n\}$, then $E\subset\cap_{n=m}^\infty E_n\;\forall\; m$ and $\lim_{m\to\infty}\mu^{(1)}(\cap_{n=m}^\infty E_n)=0$. So we can define for a.e. $x\in M^{(1)}$, $K(x):=K_{(N)}(x)$ if $x\in M^{(1)}_N$. Now claim 3 easily follows from claim 2.
\end{proof}

\subsection{Construction of the conjugation map} \label{constr nsr}
Let $\alpha \in \mathcal{L}_{\ast}$. We construct successively a sequence of measure-preserving diffeomorphisms $T_n = H^{-1}_n \circ \phi^{\alpha_{n}} \circ H_n$, where the conjugation maps $H_n = h_n \circ H_{n-1}$ and rational numbers $\alpha_{n} = \frac{p_{n}}{q_{n}} \in \Q$ are chosen in such a way that the functions $T_n$ converge to a diffeomorphism in $\text{Diff}^{\omega}_{\rho}\left( \T^2, \mu\right)$ with the desired properties. We present step $n+1$ of the construction, i.\,e. we assume that we have already defined $H_{n}$ and the numbers $\alpha_1,...,\alpha_{n-1}$. Additionally, we choose an even integer $l_n \in \mathbb{N}$ satisfying the condition
\begin{equation} \label{eq:l1}
l_n > 2^{n+1} \cdot \|DH_{n}\|_{\rho_{n}+1},
\end{equation}
where $\rho_{n}= \| H_n \|_{\rho}$. This condition on $l_n$ will be used to ensure that the sequence of partitions we construct later is generating (see the proof of [Ba15], Proposition 6.3). In this step of the construction we have to define the conjugation map $h_{n+1}$ and to choose $\alpha_n$. \begin{figure}
\centering
\begin{tabular}{c}
{\includegraphics[width=7.5cm, height=15cm]{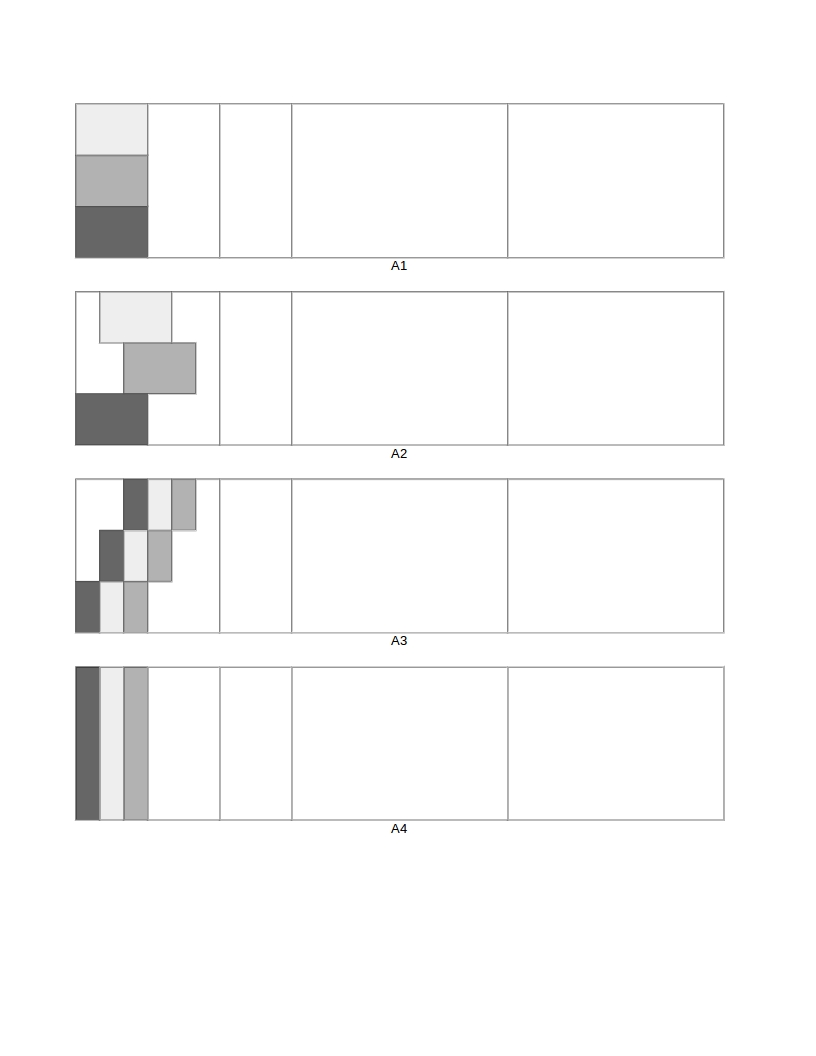}}
\end{tabular}
\caption{Illustration of the action of $\mathfrak{g}_{i,l,q}$ on $\T^2$ with $i=2,l=3$ and $q=3$. We note how the blocks are moved by the intermediate maps: $A1\xrightarrow{\mathfrak{g}_{2,3,3}^{(\mathfrak{1})}} A2\xrightarrow{\mathfrak{g}_{2,3,3}^{(\mathfrak{2})}} A3\xrightarrow{\mathfrak{g}_{2,3,3}^{(\mathfrak{3})}}A4$.}
\label{figure generating}
\end{figure}

We start by describing the main combinatorial idea behind the proof of theorems \ref{nsr circle rotation} and \ref{nsr circle rotation estimated}. Given any two integers $l$ and $q$, there exists a block-slide type map which allows us to break down the partition $\mathcal{T}_{l^dq}$ and reform it into a partition $\mathcal{G}_{l,q}$ whose atoms have diameter less than $d/l$.

First we consider the following three step functions:
\begin{align}
&\psi_{l,q}^{(\mathfrak{1})}:[0,1)\to \R &\text{ defined by}\quad &\psi_{l,q}^{(\mathfrak{1})}(x)=\sum_{i=1}^{l-1}\frac{l-i}{l^2q}\chi_{[\frac{i}{l},\frac{i+1}{l}]}(x)\nonumber\\
&\psi_{l,q}^{(\mathfrak{2})}:[0,1)\to \R &\text{ defined by}\quad &\psi_{l,q}^{(\mathfrak{2})}(x)=\sum_{i=0}^{l^2q-1}\(\frac{i}{l}-\Big\lfloor\frac{i}{l}\Big\rfloor\)\chi_{[\frac{i}{l^2q},\frac{i+1}{l^2q}]}(x)\nonumber\\
&\psi_{l,q}^{(\mathfrak{3})}:[0,1)\to \R &\text{ defined by}\quad &\psi_{l,q}^{(\mathfrak{3})}(x)=\sum_{i=0}^{l-1}\frac{i}{l^2q}\chi_{[\frac{i}{l},\frac{i+1}{l}]}(x)\label{tilde psi}
\end{align}
Then we define the following three types of block slide map:
\begin{align*}
& \mathfrak{g}_{i,l,q}^{(\mathfrak{1})}:\T^d\to\T^d\qquad\text{defined by}\qquad\mathfrak{g}_{i,l,q}^{(\mathfrak{1})}\big((x_1,\ldots,x_d)\big)=(x_1+\psi_{l,q}^{(\mathfrak{1})}(x_i),x_2,\ldots,x_d)\\
& \mathfrak{g}_{i,l,q}^{(\mathfrak{2})}:\T^d\to\T^d\qquad\text{defined by}\qquad\mathfrak{g}_{i,l,q}^{(\mathfrak{2})}\big((x_1,\ldots,x_d)\big)=(x_1,\ldots,x_{i-1},x_i+\psi_{l,q}^{(\mathfrak{2})}(x_1),x_i,\ldots,x_d)\\
& \mathfrak{g}_{i,l,q}^{(\mathfrak{3})}:\T^d\to\T^d\qquad\text{defined by}\qquad\mathfrak{g}_{i,l,q}^{(\mathfrak{3})}\big((x_1,\ldots,x_d)\big)=(x_1-\psi_{l,q}^{(\mathfrak{3})}(x_i),x_2,\ldots,x_d)
\end{align*}
Note that the composition 
\begin{align}
\mathfrak{g}_{i,l,q}:\T^d\to\T^d\qquad\qquad\text{defined by}\qquad\mathfrak{g}_{i,l,q}=\mathfrak{g}_{i,l,q}^{(\mathfrak{3})}\circ\mathfrak{g}_{i,l,q}^{(\mathfrak{2})}\circ\mathfrak{g}_{i,l,q}^{(\mathfrak{1})}
\end{align}
maps the partition $\mathcal{G}_{l^{i},q}$ to $\mathcal{G}_{l^{i+1},q}$. Where 
\begin{align}
& \mathcal{G}_{l^{j},q}:=\Big\{\big[\frac{i_1}{l^{j}q},\frac{i_1+1}{l^{j}q}\big)\times\big[\frac{i_2}{l},\frac{i_2+1}{l}\big)\times\ldots\times\big[\frac{i_{d-j+1}}{l},\frac{i_{d-j+1}+1}{l}\big)\times\T^{j-1}:i = 0,1,\ldots,lq-1,\nonumber\\
&\qquad\qquad\qquad\qquad\qquad\qquad\qquad\qquad\qquad\qquad (i_2,\ldots,i_d) \in \{0,1,\ldots,l-1\}^{d-1}\Big\}
\end{align}
So the composition
\begin{align}
\mathfrak{g}_{l,q}:\T^d\to\T^d\qquad\qquad\text{defined by}\qquad\mathfrak{g}_{l,q}=\mathfrak{g}_{2,l,q}\circ\ldots\circ\mathfrak{g}_{d,l,q}
\end{align}
maps the partition $\mathcal{G}_{l,q}$ to $\mathcal{T}_{l^dq}$.

Let $3 \varepsilon_n= \delta_n = \frac{1}{2^{n+1}}$. With the aid of Lemma \ref{lem:approx} we construct the following entire functions 
\begin{align*}
&\psi_{1,n+1} = s_{\beta^{(1)}, N^{(1)}, \varepsilon_n, \delta_n}, \text{ where } \beta^{(1)}_0=0, \ \beta^{(1)}_i = \frac{l_n-i}{l^{2}_n \cdot q_n} \text{ for } i=1,...,l_n-1, \ N^{(1)}=1 \\
&\psi_{2,n+1} = s_{\beta^{(2)}, N^{(2)}, \varepsilon_n, \delta_n}, \text{ where }  \beta^{(2)}_i=\frac{i}{l_n} \text{ for } i=0,..., l_n-1, \ N^{(2)}= l_nq_n  \\
& \psi_{3,n+1} = s_{\beta^{(3)}, N^{(3)}, \varepsilon_n, \delta_n}, \text{ where } \beta^{(3)}_i = \frac{i}{l^{2}_n \cdot q_n} \text{ for } i=0,...,l_n-1, \ N^{(3)}=1
\end{align*}
Let $A_{i,n+1}$ denote the corresponding number in the construction of $\psi_{i,n+1}$ from Lemma \ref{lem:approx}. Using these functions we define the conjugation maps
\begin{align*}
h_{1,n+1}(x_1,x_2) & = \left( x_1 + \psi_{1,n+1}(x_2) \mod 1, x_2 \right) \\
h_{2,n+1}(x_1,x_2) & = \left( x_1, x_2+ \psi_{2,n+1}(x_1) \mod 1 \right) \\
h_{3,n+1}(x_1,x_2) & = \left( x_1 - \psi_{3,n+1}(x_2) \mod 1, x_2 \right) 
\end{align*}
approximating the maps $\mathfrak{g}_{2,l_n,q_n}^{(\mathfrak{j})}$ form above. Finally, we put
\begin{equation*}
h_{n+1} = h_{3,n+1} \circ h_{2,n+1} \circ h_{1,n+1}.
\end{equation*}
We point out that $\phi^{\alpha_n} \circ h_{n+1} = h_{n+1} \circ \phi^{\alpha_n}$. 

\begin{lemma} \label{lem:A}
Let $l_n \geq 4$. In the concrete situation of our constructions we can choose
\begin{equation*}
A_{i,n+1} = 2^{2n+5} \cdot l^2_n.
\end{equation*}
\end{lemma}

\begin{proof}
Using Taylor expansion and the notation $x=\frac{\varepsilon_n}{2l_n}$ we calculate
\begin{equation*}
-\ln(-\ln(1-x))=-\ln\left( x + \frac{x^2}{2}+\frac{x^3}{3} + O(x^4) \right) \leq - \ln(x) = \ln(x^{-1}).
\end{equation*}
By our explicit definition of the number $\varepsilon_n$ we get
\begin{equation*}
\ln \left(2l_n \cdot \varepsilon^{-1}_n \right) = \ln \left( 2 l_n \cdot 3 \cdot 2^{n+1} \right) < \ln \left(  2^{n+4} \cdot l_n \right).
\end{equation*}
Then condition \ref{eq:app1} yields the requirement
\begin{equation*}
A_{i,n+1} \geq 2^{n+1} \cdot l_n \cdot \ln \left(  2^{n+4} \cdot l_n \right).
\end{equation*}
We note that condition \ref{eq:app2} is satisfied automatically.
\end{proof}

\subsection{Proof of Convergence}
Let $\rho>0$ be arbitrary. We want to prove convergence of $\left(T_n\right)_{n \in \N}$ in Diff$^{\omega}_{\rho}\left(\mathbb{T}^m\right)$. For this purpose, we introduce the numbers
\begin{equation*}
\rho_k = \|H_k\|_{\rho} \text { for any } k \in \N.
\end{equation*}
Using the definitions of the conjugation maps we compute
\begin{align*}
& h^{-1}_{n+1} \circ \phi^{\alpha_{n+1}} \circ h_{n+1} \left(x_1,x_2 \right) =
\Bigg{(} x_1+\alpha_{n+1}+\psi_{1,n+1}(x_2) - \\
&\qquad\qquad\qquad \psi_{1,n+1} \bigg{(} x_2+ \psi_{2,n+1}\left( x_1 + \psi_{1,n}(x_2) \right) - \psi_{2,n+1} \Big{(} x_1 + \alpha_{n+1} + \psi_{1,n+1}(x_2)\Big{)} \bigg{)} , \\
& \qquad\qquad\qquad\qquad\qquad \ x_2 + \psi_{2,n+1}\left( x_1 + \psi_{1,n+1}(x_2) \right)- \psi_{2,n+1} \bigg{(} x_1 + \alpha_{n+1}+ \psi_{1,n+1}(x_2)  \bigg{)} \Bigg{)}
\end{align*}
We recall that $\psi_{2,n}$ is $\frac{1}{q_n}$-periodic and get
\begin{align*}
& h^{-1}_{n+1} \circ \phi^{\alpha_{n+1}} \circ h_{n+1} \left(x_1,x_2 \right) = \Bigg{(} x_1+\alpha_{n+1}+\psi_{1,n+1}(x_2) - \\
&\qquad\qquad\qquad \psi_{1,n+1} \bigg{(} x_2+ \psi_{2,n+1}\left( x_1 + \alpha_n + \psi_{1,n+1}(x_2) \right) - \psi_{2,n+1} \Big{(} x_1 + \alpha_{n+1} + \psi_{1,n+1}(x_2)\Big{)} \bigg{)} , \\
&\qquad\qquad\qquad\qquad\qquad \ x_2 + \psi_{2,n+1}\left( x_1 + \alpha_n+ \psi_{1,n+1}(x_2) \right)- \psi_{2,n+1} \bigg{(} x_1 + \alpha_{n+1}+ \psi_{1,n+1}(x_2)  \bigg{)} \Bigg{)}
\end{align*}
Hereby, we conclude
\begin{align*}
& \left( h^{-1}_{n+1} \circ \phi^{\alpha_{n+1}} \circ h_{n+1}  - \phi^{\alpha_n}\right) \left(x_1,x_2 \right) =  \Bigg{(} \alpha_{n+1} - \alpha_n +\psi_{1,n+1}(x_2) -\\ &\qquad\qquad\qquad \psi_{1,n+1} \bigg{(} x_2+ \psi_{2,n+1}\left( x_1 + \alpha_n + \psi_{1,n+1}(x_2) \right) - \psi_{2,n+1} \Big{(} x_1 + \alpha_{n+1} + \psi_{1,n+1}(x_2)\Big{)} \bigg{)} , \\
&\qquad\qquad\qquad\qquad\qquad \ \psi_{2,n+1}\left( x_1 + \alpha_n+ \psi_{1,n+1}(x_2) \right)- \psi_{2,n+1} \bigg{(} x_1 + \alpha_{n+1}+ \psi_{1,n+1}(x_2)  \bigg{)} \Bigg{)}
\end{align*}

In the next step, we exploit the closeness of $h^{-1}_{n+1} \circ \phi^{\alpha_{n+1}} \circ h_{n+1} \left(x_1,x_2 \right)$ to $\phi^{\alpha_n}$ and find
\begin{align*}
d_{\rho} \left(f_{n+1}, f_{n} \right) = & d_{\rho} \left(H^{-1}_{n}\circ h^{-1}_{n+1} \circ \phi^{\alpha_{n+1}} \circ h_{n+1} \circ H_{n}, H^{-1}_{n}\circ  \phi^{\alpha_{n}} \circ H_{n} \right) \\
\leq & \|DH_{n}\|_{\rho_{n}+1} \cdot \|h^{-1}_{n+1} \circ \phi^{\alpha_{n+1}} \circ h_{n+1} - \phi^{\alpha_n}\|_{\rho_{n}}.
\end{align*}

In order to estimate $\|h^{-1}_{n+1} \circ \phi^{\alpha_{n+1}} \circ h_{n+1} - \phi^{\alpha_n}\|_{\rho_{n}}$ we will use the subsequent result:

\begin{lemma} \label{lem:est}
Let $\rho>0$ and $B^{\rho}= \left\{ z \in \mathbb{C} \; : \; \left|\text{im}\left(z\right)\right|< \rho\right\}$. Then we have
\begin{equation*}
\sup_{z \in B^{\rho}} \abs{s_{\beta, N, \varepsilon, \delta}(z)} \leq 2\pi \cdot N \cdot A \cdot \ee^{2 \ee^{A \cdot \ee^{2\pi N \rho}} +A \cdot \ee^{2\pi N \rho} +  2\pi N \rho}
\end{equation*}
and
\begin{equation*}
\sup_{z_1,z_2 \in B^{\rho}} \abs{s_{\beta, N, \varepsilon, \delta}(z_1)-s_{\beta, N, \varepsilon, \delta}(z_2)} \leq C \cdot A \cdot l \cdot N \cdot \ee^{4 \cdot \ee^{A \cdot \ee^{2 \pi N \rho}}} \cdot \abs{z_1-z_2},
\end{equation*}
where $C$ is a constant independent of $n$, $l$ and $N$.
\end{lemma}

\begin{proof}
First of all, we observe
\begin{equation*}
\frac{\mathrm{d}}{\mathrm{d}z}\ee^{-\ee^{-A \cdot \sin\left(2\pi\left(Nz- \frac{i}{l}\right)\right)}} = \ee^{-\ee^{-A \cdot \sin\left(2\pi\left(Nz- \frac{i}{l}\right)\right)}-A  \cdot \sin\left(2\pi\left(Nz- \frac{i}{l}\right)\right)} \cdot 2 \pi \cdot A \cdot N \cdot \cos \left(2\pi \left(Nx-\frac{i}{l}\right) \right) 
\end{equation*}
Using the mean value theorem this yields
\begin{align*}
\sup_{z \in B^{\rho}} \abs{s_{\beta, N, \varepsilon, \delta}(z)} & \leq l \cdot \ee^{\ee^{A \cdot \ee^{2\pi N \rho}}} \cdot 2\pi \cdot N \cdot A \cdot \ee^{A \cdot \ee^{2\pi N \rho}} \cdot \ee^{2\pi N \rho} \cdot \frac{1}{l} \cdot \ee^{\ee^{A \cdot \ee^{2\pi N \rho}}} \\
& \leq 2\pi \cdot N \cdot A \cdot \ee^{2 \ee^{A \cdot \ee^{2\pi N \rho}} +A \cdot \ee^{2\pi N \rho} +  2\pi N \rho}
\end{align*}
Additionally, we get
\begin{align*}
& \|Ds_{\beta,N,\varepsilon, \delta}\|_{\rho} \\
\leq & l \cdot \ee^{2 \ee^{A\cdot \ee^{2\pi N \rho}}} \cdot 4\pi \cdot A \cdot N \cdot \ee^{A\cdot \ee^{2\pi N \rho}} \cdot \ee^{2\pi \rho N} +  2\pi \cdot A \cdot N \cdot \ee^{2 \ee^{A\cdot \ee^{2\pi N \rho}}} \cdot \ee^{2A\cdot \ee^{2\pi N \rho}} \cdot \ee^{4\pi \rho N}  \\
\leq & 6\pi \cdot A \cdot l \cdot N \cdot \ee^{4 \cdot \ee^{A \cdot \ee^{2 \pi N \rho}}}.
\end{align*}
By applying the mean value theorem we obtain the second statement of the Lemma.
\end{proof}

In addition to the before mentioned conditions we require the number $l_n$ to satisfy
\begin{equation} \label{eq:l2}
l_n > \ee^{2\pi \cdot (\rho_{n}+1)}.
\end{equation}
Finally, we are able to prove convergence of the sequence $\left( f_n \right)_{n \in \mathbb{N}}$:

\begin{lemma} \label{lem:conv}
Fix $\rho >0$. Then there is a sequence $\left(\alpha_n\right)_{n \in \N}$ of rational numbers converging to $\alpha$ monotonically such that the sequence $\left( T_n \right)_{n \in \mathbb{N}}$ converges to $T$ in Diff$^{\omega}_{\rho}\left(\mathbb{T}^2\right)$.
\end{lemma}

\begin{proof}
First of all, we introduce the number
\begin{equation*}
\rho^{\prime}_n = \rho_{n}+2\pi \cdot A_{1,n+1} \cdot \ee^{2 \ee^{A_{1,n+1} \cdot \ee^{2\pi \rho_{n}}} +A_{1,n+1} \cdot \ee^{2\pi \rho_{n}} +  2\pi \rho_{n}}
\end{equation*}
We recall that $C$ as well as the requirements on $l_n$ in equations \ref{eq:l1} and \ref{eq:l2} are independent of $q_n$. Hence, we can state the subsequent condition on $q_n$:
\begin{equation} \label{eq:condq}
q_n \geq 2 C^2 \cdot l_n \cdot \ee^{4 \cdot \ee^{2^{2n+5}l^3_n}}
\end{equation}
Under this restriction on the number $q_n$ we find $\alpha_n = \frac{p_n}{q_n}$ with $p_n,q_n$ relatively prime such that
\begin{equation*}
\abs{\alpha - \alpha_n } < \frac{1}{\ee^{\ee^{\left(2^{2n+6} \cdot l^3_n \cdot \ee^{2 \pi \cdot \rho^{\prime}_n} \right)^{q_n}}}},
\end{equation*}
because $\alpha \in \mathcal{L}_{\ast}$. By using $\rho^{\prime}_n$ and by applying Lemma \ref{lem:est} twice we get
\begin{align*}
& \sup_{(x_1,x_2) \in A^{\rho_{n}}} \abs{\psi_{2,n+1}\left( x_1 + \alpha_n + \psi_{1,n+1}(x_2) \right) - \psi_{2,n} \Big{(} x_1 + \alpha_{n+1} + \psi_{1,n+1}(x_2)\Big{)} } \\
\leq & \sup_{ y \in B^{\rho^{\prime}_n}} \abs{  \psi_{2,n+1}\left( y + \alpha_n\right) - \psi_{2,n+1} \left( y + \alpha_{n+1} \right) } \\
\leq & C \cdot \ee^{\ln\left( A_{2,n+1} \cdot l_n \cdot q_n\right) + 4 \cdot \ee^{A_{2,n+1} \cdot \ee^{2 \pi \cdot q_n \cdot \rho^{\prime}_n}}} \cdot \abs{\alpha_{n+1} - \alpha_n}
\end{align*}
Under our conditions on the numbers $\alpha_k$ its value is less than $1$. Hereby, we conclude using equation \ref{eq:l2}
\begin{align*}
& \|h^{-1}_{n+1} \circ \phi^{\alpha_{n+1}} \circ h_{n+1} - \phi^{\alpha_n}\|_{\rho_{n}} 
\leq  \abs{\alpha_{n+1}-\alpha_n} +\\
& \qquad\qquad C \cdot \ee^{\ln\left( A_{1,n+1} \cdot l_n\right) + 4 \cdot \ee^{A_{1,n+1} \cdot \ee^{2 \pi \cdot (\rho_{n}+1)}}} \cdot C \cdot \ee^{\ln\left( A_{2,n+1} \cdot l_n \cdot q_n\right) + 4 \cdot \ee^{A_{2,n+1} \cdot \ee^{2 \pi \cdot q_n \cdot \rho^{\prime}_n}}} \cdot \abs{\alpha_{n+1} - \alpha_n} \\
& \qquad\qquad\qquad \leq  2 C^2 \cdot 2^{2n+5} \cdot l^3_n \cdot \ee^{4 \cdot \ee^{2^{2n+5}l^3_n}} \cdot \ee^{\ln\left( 2^{2n+5} \cdot l^3_n \cdot q_n\right) + 4 \cdot \ee^{2^{2n+5} \cdot l^2_n \cdot \ee^{2 \pi \cdot q_n \cdot \rho^{\prime}_n}}} \cdot \abs{\alpha_{n+1} - \alpha_n}
\end{align*}
With the aid of condition \ref{eq:condq} we can continue the former estimates in the following way:
\begin{align*}
& 2^{n+1} \cdot \|DH_{n}\|_{\rho_{n}+1} \cdot \|h^{-1}_{n+1} \circ \phi^{\alpha_{n+1}} \circ h_{n+1} - \phi^{\alpha_n} \|_{\rho_{n}} \\
\leq & \ee^{2 \cdot \ln\left( 2^{2n+5} \cdot l^3_n \cdot q_n\right) + 4 \cdot \ee^{2^{2n+5} \cdot l^2_n \cdot \left(\ee^{2 \pi \cdot \rho^{\prime}_n}\right)^{q_n}}} \cdot \abs{\alpha_{n+1} - \alpha_n} \\
\leq & \ee^{\ee^{2^{2n+6} \cdot l^3_n \cdot \left(\ee^{2 \pi \cdot \rho^{\prime}_n}\right)^{q_n}}} \cdot \abs{\alpha_{n+1} - \alpha_n} \\
\leq & \ee^{\ee^{\left(2^{2n+6} \cdot l^3_n \cdot \ee^{2 \pi \cdot \rho^{\prime}_n} \right)^{q_n}}} \cdot \abs{\alpha_{n+1} - \alpha_n} \\
\end{align*}
Using the above estimates we conclude:
\begin{align*}
d_{\rho}\left(T_{n+1}, T_{n} \right) & \leq \| DH_{n}\|_{\rho_{n}+1} \cdot \|h^{-1}_{n+1} \circ \phi^{\alpha_{n+1}} \circ h_{n+1} - \phi^{\alpha_n}\|_{\rho_{n}} \\
& \leq \ee^{\ee^{\left(2^{2n+6} \cdot l^3_n \cdot \ee^{2 \pi \cdot \rho^{\prime}_n} \right)^{q_n}}} \cdot \frac{1}{2^{n+1}} \abs{\alpha_{n+1} - \alpha_n} \\
& \leq \ee^{\ee^{\left(2^{2n+6} \cdot l^3_n \cdot \ee^{2 \pi \cdot \rho^{\prime}_n} \right)^{q_n}}} \cdot \frac{1}{2^{n+1}} \cdot 2 \cdot \abs{\alpha-\alpha_n} \\
& < \frac{1}{2^n}.
\end{align*}
Hence, $\left(T_n\right)_{n \in \mathbb{N}}$ is a Cauchy sequence in Diff$^{\omega}_{\rho}\left( \mathbb{T}^2 \right)$. Since this is a complete space, we obtain convergence of $\left(T_n\right)_{n \in \mathbb{N}}$ to a real-analytic diffeomorphism $T \in \text{Diff}^{\omega}_{\rho}\left( \mathbb{T}^2 \right)$. 
\end{proof}

\subsection{Proof of conjugacy of \texorpdfstring{$T$}{TEXT} to the rotation \texorpdfstring{$R_{\alpha}$}{TEXT} of the circle}

% {\rr We should define $\mathcal{F}_{q}$ here and include Proposition 6.3 of your paper} \\
This section is identical to section 6 of [Ba15] and we omit very detailed proofs which are available in that paper. We have a sequence of real-analytic diffeomorphisms $T_n$ converging to a real-analytic diffeomorphism $T$, a generating sequence of partitions $\mathcal{F}_{q_n}$  of $\T^2$ (see subsection \ref{subsection abc method}) and each $\mathcal{F}_{q_n}$ is cyclically permuted by $T_n$. We also know the convergence of $\left(\alpha_n\right)_{n \in \mathbb{N}}$ to the prescribed number $\alpha$.

On the other hand we approximate an irrational rotation of the circle by rational rotations. Let $\a_n=\frac{p_n}{q_n}$ be as in the approximation by conjugation scheme described above and consider a sequence of partitions of the circle as follows:
\begin{align*}
\mathcal{C}_{q_n}:=\Big\{\Gamma_{i,q_n}:=\Big[\frac{i}{q_n},\frac{i+1}{q_n}\Big):\; i=0,1,\ldots q_n-1\Big\}
\end{align*}
Clearly this is a sequence of partitions are monotonic and generating. We also, define a sequence of maps:
\begin{align*}
R_{\a_n}:S^1\to S^1,\quad\quad \text{ defined by }x\mapsto x+\a_n
\end{align*}
So, we have $R_{\a_n}\to R_\a$. We also define 
\begin{align*}
\tilde{E}_{{n+1}}:=\bigcup_{i=0}^{q_{n+1}}\Big[\frac{i}{l_n^2q_n}-\frac{\mu(E_{{n+1}})}{2l_n^2q_n},\frac{i}{l_n^2q_n}+\frac{\mu(E_{{n+1}})}{2l_n^2q_n}\Big] 
\end{align*}

Following the notation of lemma \ref{lemma mtl} we let $M^{(1)}:=\T^2,\mu^{(1)}:=\mu,\mathcal{P}^{(1)}_n:=\mathcal{F}_{q_n}, E^{(1)}_n:=E_{n}, M^{(2)}:=\T^1,\mu^{(2)}:=\l, \mathcal{P}^{(2)}_{n}:=\mathcal{C}_{q_n}$ and $E^{(2)}_n:=\tilde{E}_{{n+1}}$. Finally we define the conjugacy $K_n$ by $K_n(H_n^{-1}\Delta_{i,q_n})=\Gamma_{i,q_n}$. This gives us that $\mathcal{P}^{(i)}_n$ is generating and  conditions \ref{mtl 1}, \ref{mtl 2}, \ref{mtl 4} and \ref{mtl 5} in lemma \ref{lemma mtl}. Conditions \ref{mtl 7} and \ref{mtl 8} follows from the definition. Now note that $\phi^{\alpha_n}$ preserves $E_{{n+1}}^{(v)}$ and $E_{{n+1}}^{(d)}$ and hence $T_{n+1}$ preserves $E_{{n+1}}$. This gives us \ref{mtl 6} and completes the proof of Theorem \ref{nsr circle rotation estimated}.

\subsection{Set of numbers \texorpdfstring{$\mathcal{L}_{\ast}$}{TEXT}} \label{subsection:numbers}
As announced in the introduction we examine the set of obtained rotation numbers. By well known arguments for spaces like $\mathcal{L}_{\ast}$ (e.\,g. \cite{Br}, Appendix A.2) we prove
\begin{lemma}
$\mathcal{L}_{\ast}$ is a dense $G_{\delta}$-subset of $\mathbb{R}$.
\end{lemma}

\begin{proof}
For each pair $\left(p,q\right) \in \mathbb{Z} \times \mathbb{N}$ with $p,q$ relatively prime and every $k \in \mathbb{N}$ we define the following open set
\begin{equation*}
O_{k}(p,q) = \Meng{ x \in \R}{ 0 < \abs{x - \frac{p}{q}} < \frac{1}{\ee^{\ee^{k^q}}}}.
\end{equation*}
Then the countable union
\begin{equation*}
U_k = \bigcup_{\left(p,q\right) \in \mathbb{Z} \times \mathbb{N} \text{ with } p,q \text{ relatively prime}} O_{k}(p,q)
\end{equation*}
is also open in $\mathbb{R}$ for every $k \in \mathbb{N}$. In the next step, we fix $k \in \mathbb{N}$. Obviously, each rational number $\omega \in \mathbb{Q}$ written in its lowest form $\omega= \frac{p}{q}$ lies in the closure of $O_{k}(p,q)$. Hence, each rational number lies in the closure of $U_k$. Since the rational numbers are dense in $\mathbb{R}$, $U_k$ is dense in $\mathbb{R}$. This applies to all $k \in \mathbb{N}$. Moreover, we observe
\begin{equation*}
\mathcal{L}_{\ast} = \bigcap_{k \in \mathbb{N}} U_k.
\end{equation*}
By the Baire category Theorem, $\mathcal{L}_{\ast}$ as a countable intersection of open dense sets is dense in $\mathbb{R}$.
\end{proof}

\section{Non standard analytic realization of some ergodic translations of the torus}

Our goal in this section is to produce a proof of theorem \ref{theorem nsr total translations}. The proof of this theorem in the smooth category was done by Anosov and Katok in \cite{AK}. Later Benhenda in \cite{Mb-ts} produced an estimated version of this result showing that every ergodic toral translation with one arbitrary Liouvillian coordinate can be realized smoothly on any manifold admitting a circle action.

In our article we prove a real analytic version of this result. Unfortunately we do not have the techniques to prove results like theorem 1.1 and theorem 1.2 in \cite{AK} for arbitrary real-analytic manifolds. However the concept of block-slide type maps and their real-analytic approximations allow us just about enough flexibility on the torus to produce pretty much any combinatorial picture the approximation by conjugation scheme requires. So we prove that there exist real-analytic diffeomorphisms on $\T^d$ which are metrically conjugated to ergodic translations on $\T^h$ for arbitrary $d\geq 2$ and $h\geq 1$.

As of now we do not know if this concept of block-slide type maps can be successfully generalized to other types of real-analytic manifolds. It is possible that these maps can be generalized to odd dimensional spheres using transitive flows like those described in \cite{FK-ue}. The main obstruction to this generalizations seems to be the fact that the analytic flows they use commutes with the Hopf fibration but not with each other.

\subsection{Description of the required combinatorics} \label{constr transl}

Our objective here is to demonstrate that one can reproduce the approximation by conjugation scheme as described in \cite{AK} in its full generality in the real-analytic category on $\T^d$, $d \geq 2$. We show three basic kind of rearrangement techniques in this section. Together, these three kind of rearrangements will be sufficient to produce all constructions done in \cite{AK}.

\subsubsection*{Periodic interchange of two consecutive atoms}

\begin{figure}
\centering
\begin{tabular}{c c}
{\includegraphics[width=7cm, height=15cm]{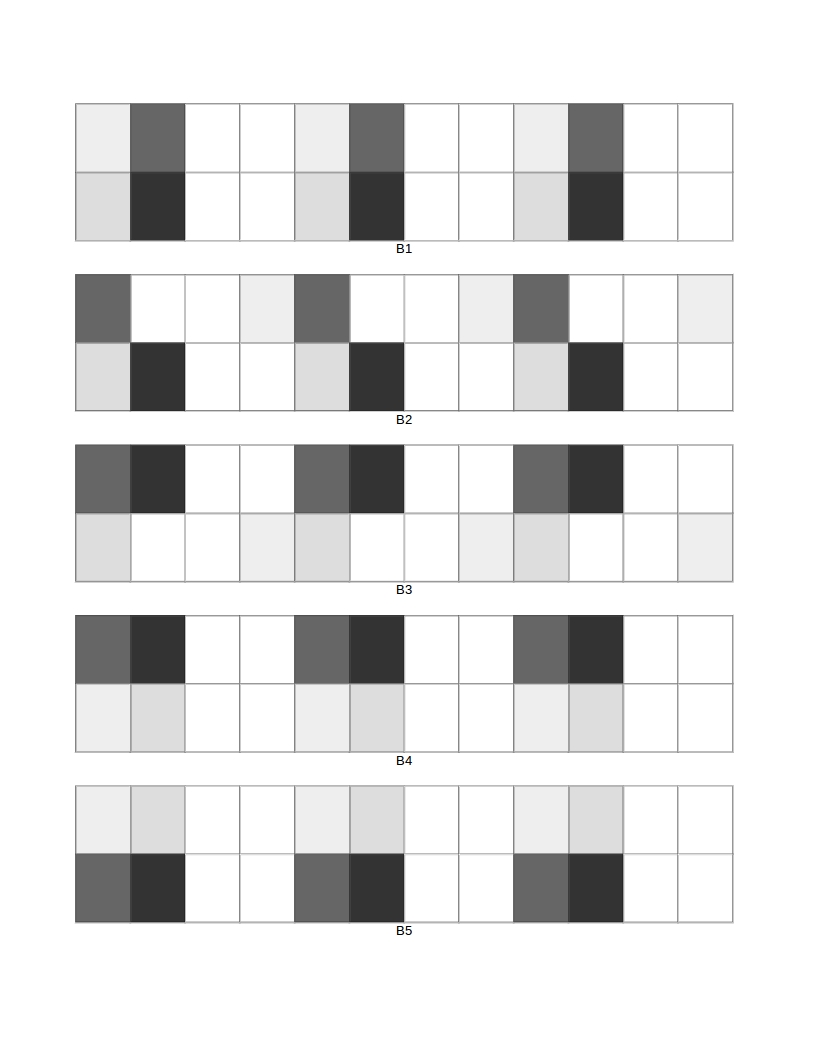}} & {\includegraphics[width=7cm, height=15cm]{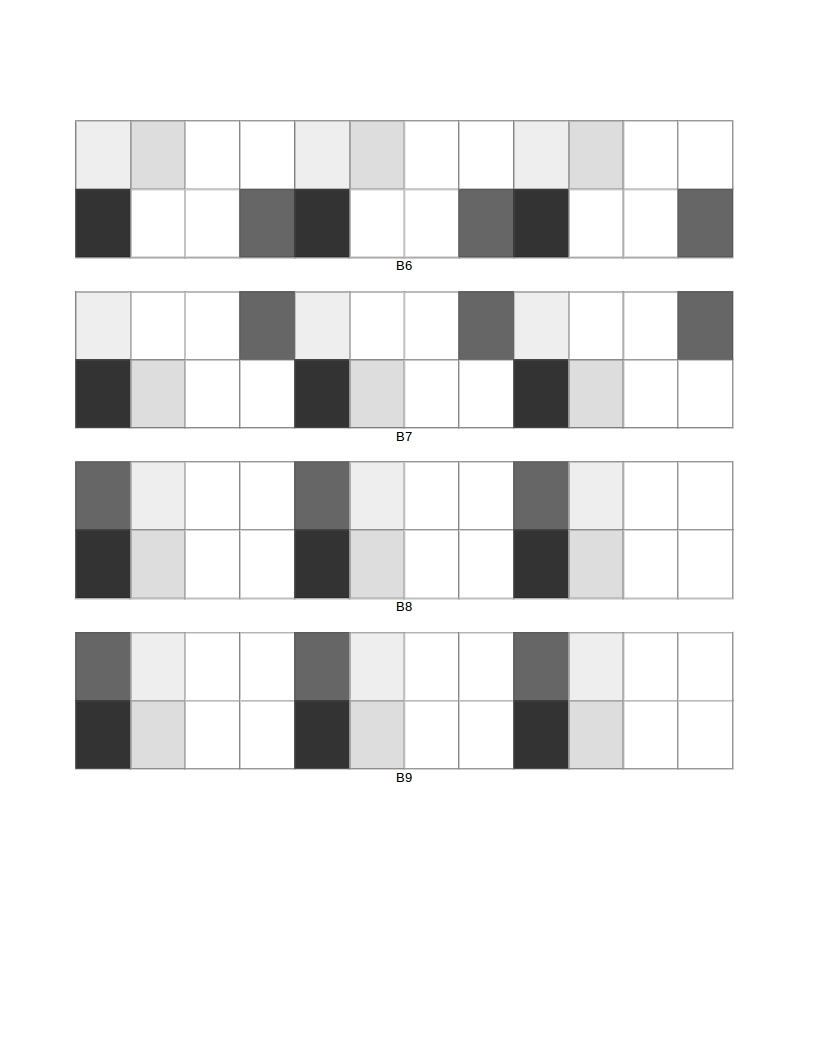}}
\end{tabular}
\caption{Illustration of the action of $\mathfrak{f}_{kq}$ on $\T^2$ with $k=4$ and $q=3$. We note how the blocks are moved by the intermediate maps: $
     B1\xrightarrow{\mathfrak{f}_{kq}^{(\mathfrak{1})}} B2\xrightarrow{\mathfrak{f}_{kq}^{(\mathfrak{2})}} B3\xrightarrow{\mathfrak{f}_{kq}^{(\mathfrak{3})}} B4\xrightarrow{\mathfrak{f}_{kq}^{(\mathfrak{4})}} B5\xrightarrow{\mathfrak{f}_{kq}^{(\mathfrak{5})}} B6\xrightarrow{\mathfrak{f}_{kq}^{(\mathfrak{6})}} B7\xrightarrow{\mathfrak{f}_{kq}^{(\mathfrak{7})}} B8\xrightarrow{\mathfrak{f}_{kq}^{(\mathfrak{8})}} B9  
$. Note that at the end the first two columns are interchanged. $\mathfrak{f}_{kq}$ acts as an identity on every other column. We also note that $B8$ and $B9$ appear to be identical but the unmarked rectangles are flipped.}
\label{figure interchange}
\end{figure}

Fix any two integers $k$ and $q$. Our objective here is to show that one can interchance two consecutive atoms of $\mathcal{T}_{kq}$ periodically inside each atom of $\mathcal{T}_q$. More precisely we show that there exist a block-slide type map $\mathfrak{f}_{k,q}$ that interchanges the atom $\Delta_{ik,kq}$ with the atom $\Delta_{ik+1,kq}$ for $i=0,\ldots,q-1$ and leaves all other atoms of $\mathcal{T}_{kq}$ unchanged. 

We begin by considering the following step functions (or more appropriately piecewise constant functions):
\begin{align}
& \s^{(\mathfrak{1})}_{kq}:(0,1]\to\R\qquad\qquad\text{defined by}\qquad \s^{(\mathfrak{1})}_{kq}(t)=\begin{cases} 0 &  \quad\text{if } t \in (0,1/2] \\ 1/(kq) & \quad\text{if } t\in (1/2,1]\end{cases} \\
& \s^{(\mathfrak{2})}_{kq}:(0,1]\to\R\qquad\qquad\text{defined by}\qquad \s^{(\mathfrak{2})}_{kq}(t)=\begin{cases} 1/(kq) &  \quad\text{if } t \in (0,1/2] \\ 0 & \quad\text{if } t \in (1/2,1]\end{cases}\\
& \s^{(\mathfrak{3})}_{kq}:(0,1]\to\R\qquad\qquad\text{defined by}\qquad \s^{(\mathfrak{3})}_{kq}(t)=\begin{cases} 0 & \quad\text{if } qt\mod 1 \in (0,\frac{1}{k}]\\ 1/2 & \quad\text{if } qt\mod 1 \in (\frac{1}{k},1]\end{cases}\\
& \s^{(\mathfrak{4})}_{kq}:(0,1]\to\R\qquad\qquad\text{defined by}\qquad \s^{(\mathfrak{4})}_{kq}(t)=\begin{cases} 0 & \quad\text{if } qt\mod 1 \in (0,\frac{2}{k}]\\ 1/2 & \quad\text{if } qt\mod 1 \in (\frac{2}{k},1]\end{cases} 
\end{align}
Note that $\s^{(\mathfrak{3})}_{kq}$ is $1/q$ periodic and we can define the following piecewise continuous functions on $\T^d$:
\begin{align}
& \mathfrak{f}_{kq}^{(\mathfrak{1})}:\T^d\to\T^d\qquad\qquad\text{defined by}\qquad\mathfrak{f}_{kq}^{(\mathfrak{1})}\big((x_1,\ldots,x_d)\big)=(x_1-\s_{kq}^{(\mathfrak{1})}(x_2),x_2,\ldots,x_d)\\
& \mathfrak{f}_{kq}^{(\mathfrak{2})}:\T^d\to\T^d\qquad\qquad\text{defined by}\qquad\mathfrak{f}_{kq}^{(\mathfrak{2})}\big((x_1,\ldots,x_d)\big)=(x_1,x_2+\s_{kq}^{(\mathfrak{3})}(x_1),\ldots,x_d)\\
& \mathfrak{f}_{kq}^{(\mathfrak{3})}:\T^d\to\T^d\qquad\qquad\text{defined by}\qquad\mathfrak{f}_{kq}^{(\mathfrak{3})}\big((x_1,\ldots,x_d)\big)=(x_1+\s_{kq}^{(\mathfrak{2})}(x_2),x_2,\ldots,x_d)\\
& \mathfrak{f}_{kq}^{(\mathfrak{4})}:\T^d\to\T^d\qquad\qquad\text{defined by}\qquad\mathfrak{f}_{kq}^{(\mathfrak{4})}\big((x_1,\ldots,x_d)\big)=(x_1,x_2+1/2,\ldots,x_d)\\
& \mathfrak{f}_{kq}^{(\mathfrak{5})}:\T^d\to\T^d\qquad\qquad\text{defined by}\qquad\mathfrak{f}_{kq}^{(\mathfrak{5})}\big((x_1,\ldots,x_d)\big)=(x_1-\s_{kq}^{(\mathfrak{2})}(x_2),x_2,\ldots,x_d)\\
& \mathfrak{f}_{kq}^{(\mathfrak{6})}:\T^d\to\T^d\qquad\qquad\text{defined by}\qquad\mathfrak{f}_{kq}^{(\mathfrak{6})}\big((x_1,\ldots,x_d)\big)=(x_1,x_2+\s_{kq}^{(\mathfrak{3})}(x_1),\ldots,x_d)\\
& \mathfrak{f}_{kq}^{(\mathfrak{7})}:\T^d\to\T^d\qquad\qquad\text{defined by}\qquad\mathfrak{f}_{kq}^{(\mathfrak{7})}\big((x_1,\ldots,x_d)\big)=(x_1-\s_{kq}^{(\mathfrak{1})}(x_2),x_2,\ldots,x_d)\\
& \mathfrak{f}_{kq}^{(\mathfrak{8})}:\T^d\to\T^d\qquad\qquad\text{defined by}\qquad\mathfrak{f}_{kq}^{(\mathfrak{8})}\big((x_1,\ldots,x_d)\big)=(x_1,x_2+\s_{kq}^{(\mathfrak{4})}(x_2),\ldots,x_d)
\end{align}
We compose all the functions above into the following function:
\begin{align}
\mathfrak{f}_{k,q}:\T^d\to\T^d\qquad\qquad\text{defined by}\qquad \mathfrak{f}_{k,q}:=\mathfrak{f}_{kq}^{(\mathfrak{8})}\circ\mathfrak{f}_{kq}^{(\mathfrak{7})}\circ\mathfrak{f}_{kq}^{(\mathfrak{6})}\circ\mathfrak{f}_{kq}^{(\mathfrak{5})}\circ\mathfrak{f}_{kq}^{(\mathfrak{4})}\circ\mathfrak{f}_{kq}^{(\mathfrak{3})}\circ\mathfrak{f}_{kq}^{(\mathfrak{2})}\circ\mathfrak{f}_{kq}^{(\mathfrak{1})}
\end{align}

\subsubsection*{Periodic rearrangement of atoms}

\begin{figure}
\centering
\begin{tabular}{c}
{\includegraphics[width=7.5cm, height=15cm]{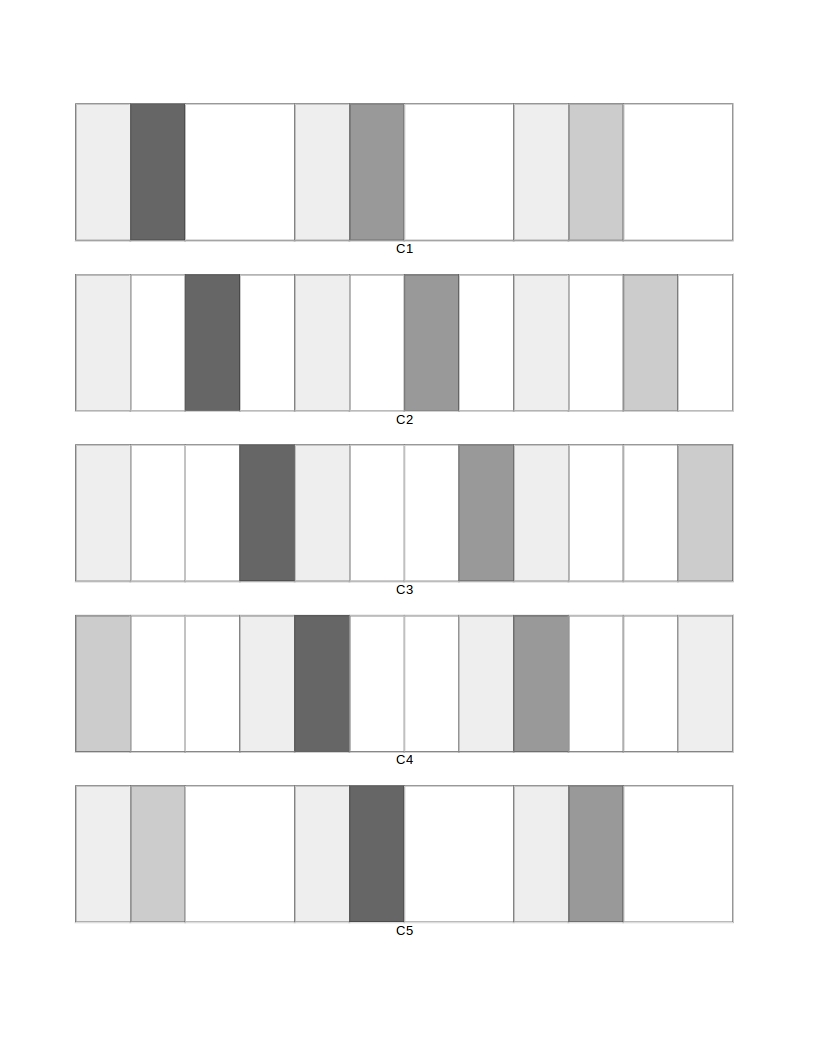}}
\end{tabular}
\caption{Illustration of the action of $\mathfrak{w}_{i,l,k,q}$ on $\T^2$ with $l=1$, $i=1$, $k=4$ and $q=3$. We note how the blocks are moved by the intermediate maps: $C1\xrightarrow{\mathfrak{w}_{1,4,3}} C2\xrightarrow{\mathfrak{w}_{2,4,3}} C3\xrightarrow{\mathfrak{w}_{3,4,3}} C4\xrightarrow{\phi^{1/(12)}} C5 $. Note that every other columns are kept constant at the end.
}
\label{figure walking}
\end{figure}

Now we show that for for any given $l$ and $i$ with $0\leq l< k$ and $0\leq i<q$, there exists a map of the block-slide kind which will allow us to rearrange the atoms of $\mathcal{T}_{k,q}$ so that for any $i$ with $0\leq i<q$, the atom $\Delta_{l+kj,kq}$ is moved to $\Delta_{l+k(j+i),kq}$ while any atom that is not of the form $\Delta_{l+kj',kq}$ is left invariant.

Now we describe this map. Consider the following block-slide type map,
\begin{align}
& \mathfrak{w}_{l,k,q}:\T^d\to\T^d\qquad\qquad\text{defined by}\qquad\mathfrak{w}_{l,k,q}=\phi^{l/(kq)}\circ\mathfrak{f}_{k,q}\circ\phi^{-l/(kq)}
\end{align}
and consider the following composition,
\begin{align}
& \mathfrak{w}_{i,l,k,q}:\T^d\to\T^d\qquad\qquad\text{defined by}\qquad\mathfrak{w}_{i,l,k,q}=\phi^{1/(kq)}\circ\mathfrak{w}_{l+ki-2,k,q}\circ\ldots\mathfrak{w}_{l+1,k,q}\circ\mathfrak{w}_{l,k,q}
\end{align}
%{\rr Next try: $\phi^{1/(kq)}\circ\mathfrak{w}_{l+ki-2,k,q}\circ\ldots\mathfrak{w}_{l+1,k,q}\circ\mathfrak{w}_{l,k,q}$} Yep!
Note that the above is the map we desired at the begining.

\subsubsection*{Recovering  $\mathcal{T}_{l^dq}$ from a generating partition}

We show that given any two integers $l$ and $q$, there exists a block-slide type map which allows us to break down the partition $\mathcal{T}_{l^dq}$ and reform it into a partition $\mathcal{G}_{l,q}$ whose atoms have diamter less than $d/l$.

We consider the following three types of block slide map:
\begin{align*}
& \mathfrak{g}_{i,l,q}^{(\mathfrak{1})}:\T^d\to\T^d\qquad\text{defined by}\qquad\mathfrak{g}_{i,l,q}^{(\mathfrak{1})}\big((x_1,\ldots,x_d)\big)=(x_1+\psi_{l,q}^{(\mathfrak{1})}(x_i),x_2,\ldots,x_d)\\
& \mathfrak{g}_{i,l,q}^{(\mathfrak{2})}:\T^d\to\T^d\qquad\text{defined by}\qquad\mathfrak{g}_{i,l,q}^{(\mathfrak{2})}\big((x_1,\ldots,x_d)\big)=(x_1,\ldots,x_{i-1},x_i+\psi_{l,q}^{(\mathfrak{2})}(x_1),x_{ i+1},\ldots,x_d)\\
& \mathfrak{g}_{i,l,q}^{(\mathfrak{3})}:\T^d\to\T^d\qquad\text{defined by}\qquad\mathfrak{g}_{i,l,q}^{(\mathfrak{3})}\big((x_1,\ldots,x_d)\big)=(x_1-\psi_{l,q}^{(\mathfrak{3})}(x_i),x_2,\ldots,x_d)
\end{align*}
with the maps $\psi_{l,q}^{(\mathfrak{i})}$ defined as above. Note that the composition 
\begin{align}
\mathfrak{g}_{i,l,q}:\T^d\to\T^d\qquad\qquad\text{defined by}\qquad\mathfrak{g}_{i,l,q}=\mathfrak{g}_{i,l,q}^{(\mathfrak{3})}\circ\mathfrak{g}_{i,l,q}^{(\mathfrak{2})}\circ\mathfrak{g}_{i,l,q}^{(\mathfrak{1})}
\end{align}
maps the partition $\mathcal{G}_{j,l,q}$ to $\mathcal{G}_{j+1,l,q}$, where 
\begin{align}
& \mathcal{G}_{j,l,q}:=\Big\{\big[\frac{i_1}{l^{j}q},\frac{i_1+1}{l^{j}q}\big)\times\big[\frac{i_2}{l},\frac{i_2+1}{l}\big)\times\ldots\times\big[\frac{i_{d-j+1}}{l},\frac{i_{d-j+1}+1}{l}\big)\times\T^{j-1}:i_1 = 0,1,\ldots,lq-1,\nonumber\\
&\qquad\qquad\qquad\qquad\qquad\qquad\qquad\qquad\qquad\qquad (i_2,\ldots,i_{d-j+1}) \in \{0,1,\ldots,l-1\}^{d-j}\Big\}
\end{align}
So the composition
\begin{align}
\mathfrak{g}_{l,q}:\T^d\to\T^d\qquad\qquad\text{defined by}\qquad\mathfrak{g}_{l,q}=\mathfrak{g}_{d-1,l,q}\circ\ldots\circ\mathfrak{g}_{1,l,q}
\end{align}
maps the partition $\mathcal{G}_{l,q}$ to $\mathcal{T}_{l^dq}=\mathcal{G}_{l^dq}$.

\subsubsection*{Piecing everything together}

Our objective now is to demonstrate that there is a $\frac{1}{q}$-periodic block-slide type of map which maps the partition $\mathcal{G}_{l,q}$ to $\mathcal{R}_{a,k,q}$. Such a map is obtained after taking a composition of some of the maps defined above:

Consider the following composition:
\begin{align}
\mathfrak{h}_{a,k,q}^{(\mathfrak{1})}:\T^d\to\T^d\qquad\qquad\text{defined by}\qquad\mathfrak{h}_{a,k,q}^{(\mathfrak{1})}=\mathfrak{w}_{a(k-1),k-1,k,q}\circ\ldots\circ\mathfrak{w}_{a(1),1,k,q}\circ\mathfrak{w}_{a(0),0,k,q}
\end{align}
and note that the above map maps the partition $\mathcal{R}_{a,k,q}$ to the decomposition $\mathcal{T}_{q}$. Next we define:
\begin{align}
\mathfrak{h}_{a,k,q}^{(\mathfrak{2})}:\T^d\to\T^d\qquad\qquad\text{defined by}\qquad\mathfrak{h}_{l,q}^{(\mathfrak{2})}=\mathfrak{g}_{l,q}
\end{align}
and note that the above map maps the partition $\mathcal{G}_{l,q}$ to $\mathcal{T}_{l^dq}$. So the composition:
\begin{align}
\mathfrak{h}_{a,k,l,q}:\T^d\to\T^d\qquad\qquad\text{defined by}\qquad\mathfrak{h}_{a,k,l,q}=(\mathfrak{h}_{a,k,q}^{(\mathfrak{1})})^{-1}\circ\mathfrak{h}_{kl,q}^{(\mathfrak{2})}
\end{align}
and note that the above map satisfies the following properties:
\begin{enumerate}
\item $\mathfrak{h}_{a,k,l,q}^{-1}(\mathcal{R}_{a,k,q})=\mathcal{T}_q$.
\item $\mathfrak{h}_{a,k,l,q}^{-1}(\mathcal{T}_{l^dk^dq})=\mathcal{G}_{lk,q}$.
\item  { $\mathfrak{h}_{a,k,l,q}\circ\phi^{\a}=\phi^{\a}\circ\mathfrak{h}_{a,k,l,q}$} for any $p$ and $\a=p/q$.
\end{enumerate}

\subsection{Periodic approximation of ergodic translations of the torus}

This entire section is identical to section 6 in \cite{AK}. So we only recall the portions that we need. For exact proofs, one may refer to the original article.

\begin{lemma} \label{8.90}
There exists sequences $\a_n=(\a_n^{(1)},\ldots,\a_n^{(h)})$ and $\gamma_n=(\gamma_n^{(1)},\ldots,\gamma^{(h)}_n)$ satisfying the following properties:
\begin{enumerate}
\item $\gcd(\gamma_n^{(1)},\ldots,\gamma_n^{(h)})=1$
\item There exists integers $p_n,q_n$ such that $\gcd(p_n,q_n)=1$ and $\a_n=(p_n/q_n)\gamma_n$.
\item There exists integers $r_n$ such that $q_n=r_n\gamma_{n-1}^{(h)}$.
\item There exists integers $s_n$ such that $\gamma_{n+1}^{(h)}=s_n\gamma_{n}^{(h)}$.
\item $\gamma_{n+1}^{(i)}\equiv \gamma^{(i)}_n\mod q_{n}$, for  $i=1,\ldots, h$.
\item There exists integers $m_n$ such that 
\begin{align}
\frac{p_{n+1}}{q_{n+1}}=\frac{p_n}{q_n}+\frac{1}{m_ns_nq_n^2}
\end{align}
\item Let $\Gamma_n'\subset \T^{h-1}\times\{0\}\subset \T^h$ be a fundamental domain of the flow $T^{t \gamma_n}$. Let $d_n:=\text{diam}(\Gamma_n),\;$ $\sigma_{n}=\mu_{h-1}(\partial(\Gamma_n))$. Then $d_{n+1}<1/(2^{n}\gamma_{n}^{(h)}\sigma_n)$.
\item \begin{align}
\Big|\frac{\gamma_{n+1}}{\gamma_{n+1}^{(h)}}-\frac{\gamma_{n}}{\gamma_{n}^{(h)}}\Big|<\frac{1}{2^n\sigma_nq_n}
\end{align}
\end{enumerate}
\end{lemma}

%{\rr Benhenda claims that there is a mistake in the computations in the old Anosov-Katok paper and so he gets a different last condition (see the footnote on page 17 of his preprint https://hal.archives-ouvertes.fr/hal-00669027/document ). I haven't checked these computations yet. The different condition wouldn't change the following arguments.}
With $\a_n,p_n,q_n,\gamma_n,r_n,s_n,m_n$ and $\Gamma_n'$ as in lemma \ref{8.90}, we construct the following two sequences of partitions of $\T^{h-1}\times\{0\}\subset\T^h$:
\begin{align}
&\tilde{\mathcal{F}}_{q_n}':=\big\{\Gamma_{i,q_n}':\Gamma_{i,q_n}':=T^{i\gamma_n/\gamma_n^{(h)}}\Gamma_n', i=0,\ldots, q_n-1\big\}  \\
&\tilde{\mathcal{F}}_{q_n,q_{n+1}}':=\big\{\Gamma_{i,q_n,q_{n+1}}':\Gamma_{i,q_n,q_{n+1}}':=T^{i\gamma_n/\gamma_n^{(h)}}(\cup\{\Gamma_{j,q_{n+1}}':T^{j\gamma_{n+1}/\gamma_{n+1}^{(h)}}(0)\in\Gamma_n', \nonumber\\
& \hspace{200pt}  j=0,\ldots, q_{n+1}-1, i=0,\ldots, q_n-1\})\big\}
\end{align}
Note that $\mathcal{F}'_{q_n}>\mathcal{F}'_{q_n,q_{n+1}}$. and they are both generating sequence of partitions. We construct the following two sequence of partitions of $\T^h$ from the above two partitions:
\begin{align}
&\tilde{\mathcal{F}}_{q_n}:=\big\{\Gamma_{i,q_n}:\Gamma_{i,q_n}:=T^{i\gamma_n/q_n}(\cup\{T^{t\gamma_n}\Gamma_n':0\leq t<1/q_n\}),\; i=0,\ldots, q_n-1\big\}\\
& \tilde{\mathcal{F}}_{q_n,q_{n+1}}:=\big\{\Gamma_{i,q_n,q_{n+1}}:\Gamma_{i,q_n,q_{n+1}}:=T^{i\gamma_n/q_n}(\cup\{T^{t\gamma_{n+1}}\Gamma_{0,q_n,q_{n+1}}':0\leq t<1/(r_n\gamma_{n+1}^{(h)})\}),\nonumber\\
&\hspace{280pt} i=0,\ldots, q_n-1\big\}
\end{align}

The following proposition summarizes certain properties of the above partitions. For a proof one can refer to page 28-29 of \cite{AK}.
\begin{proposition}\label{proposition cyclic approximation of translations}
With $\a_n,p_n,q_n,\gamma_n,r_n,s_n,m_n$ as in lemma \ref{8.90} we can conclude the following:
\begin{enumerate}
\item The sequence of partitions $\tilde{\mathcal{F}}_{q_n,q_{n+1}}$ and $\tilde{\mathcal{F}}_{q_n}$ are respectively preserved and permuted by $T^{\a_n}$.
\item $\mu_h(\Gamma_{i,q_n,q_{n+1}} \triangle \Gamma_{i,q_n})<1/(2^{n-3}q_n)$ \footnote{Here $\mu_h$ denotes the standard Lebesgue measure on $\T^h$ }  for any $\Gamma_{i,q_n,q_{n+1}}\in \tilde{\mathcal{F}}_{q_n,q_{n+1}}$ and $\Gamma_{i,q_n}\in \tilde{\mathcal{F}}_{q_n}$ with the same $i$. 
\item The sequence of periodic translations $T^{\a_n}:\T^h\to\T^h$ converges to an ergodic translation $T^{\a}:\T^h\to\T^h$.
\end{enumerate}
\end{proposition}
Next proposition is identical to lemma 6.2 in \cite{AK}.
\begin{proposition} \label{proposition monotonic generating cyclic partition 2}
Under the same hypothesis as proposition \ref{proposition cyclic approximation of translations} we can find a sequence of partitions $\tilde{\mathcal{M}}_n$ of $\T^h$ satisfying the following three properties:
\begin{enumerate}
\item Monotonicity condition: $\tilde{\mathcal{M}}_{n+1}>\tilde{\mathcal{M}}_n$
\item Cyclic permutaion: The diffeomorphims $T^{\a_n}$ cyclically permutes the atoms of $\tilde{\mathcal{M}}_n$.
\item Generating condition: $\tilde{\mathcal{M}}_n\to\e$ as $n\to\infty$. 
\end{enumerate}
\end{proposition}

\begin{proof}
We use a method similar to the proof of \ref{proposition monotonic generating cyclic partition}. We define for any $n$, the following three maps 
\begin{align}
& \mathfrak{q}_{n+1,n}^{(\mathfrak{1})}:\T^h/\tilde{\mathcal{F}}_{q_n}\to\T^h/\tilde{\mathcal{F}}_{q_n,q_{n+1}}\qquad\text{defined by}\quad \mathfrak{q}_{n+1,n}^{(\mathfrak{1})}(\Gamma_{i,q_n}):=\Gamma_{i,q_n,q_{n+1}}\\
& \mathfrak{q}_{n+1,n}^{(\mathfrak{2})}:\T^h/\tilde{\mathcal{F}}_{q_{n+1}}\to\T^h/\tilde{\mathcal{F}}_{q_n,q_{n+1}}\quad\text{defined by}\nonumber\\
&\hspace{150pt}\mathfrak{q}_{n+1,n}^{(\mathfrak{2})}(\Gamma_{i,q_{n+1}}):=\Gamma_{j,q_n,q_{n+1}}\;\text{if}\;\Gamma_{i,q_{n+1}}\subset\Gamma_{j,q_n,q_{n+1}}\\
& \mathfrak{q}_{n+1,n}:\T^h/\tilde{\mathcal{F}}_{q_n}\to\T^h/\tilde{\mathcal{F}}_{q_{n+1}}\qquad\text{defined by}\quad \mathfrak{q}_{n+1,n}:=(\mathfrak{q}_{n+1,n}^{(\mathfrak{2})})^{-1}\circ\mathfrak{q}_{n,n+1}^{(\mathfrak{1})}
\end{align}
and more generally we define the composition map for any $m$ and $n$ with $m>n$ as follows:
\begin{align}
 \mathfrak{q}_{m,n}:\T^h/\tilde{\mathcal{F}}_{q_n}\to\T^h/\tilde{\mathcal{F}}_{q_{m}}\qquad\text{defined by}\quad \mathfrak{q}_{n,m}=\mathfrak{q}_{m,m-1}\circ\ldots\circ\mathfrak{q}_{n+1,n}
\end{align}
%{\rr Definition of $\tilde{\mathcal{M}}_n$ analogous to Prop. 2.34?} Yes
We now define the partition $\tilde{\mathcal{M}}=\{\lim_{m\to\infty}\mathfrak{q}_{m,n}(\Gamma_{i,q_n}): 0\leq i< q_n\}$ (see proposition \ref{proposition monotonic generating cyclic partition}) and finally the correspondence: 
\begin{align}
\mathfrak{q}_{\infty,n}:\T^h/\tilde{\mathcal{F}}_{q_n}\to\T^h/\tilde{\mathcal{M}}\qquad\text{defined by}\quad \mathfrak{q}_{\infty,n}=\lim_{m\to\infty}\mathfrak{q}_{m,n}
\end{align}
The rest of the proof involves proving that $\tilde{\mathcal{M}}$ is indeed a partition satisfying the required conditions. This part can be completed identical to proposition \ref{proposition monotonic generating cyclic partition} and we do not repeat it again.
\end{proof}

\subsection{Analytic diffeomorphisms metrically isomorphic to a shift on a Torus}

Our goal in this section is to prove theorem  \ref{theorem nsr total translations}:

\begin{proof}
First we introduce the following two correspondences
\begin{align}
& K_n:\T^h/\tilde{\mathcal{F}}_{q_n}\to\T^d/\mathcal{F}_{q_n}\qquad\qquad\text{defined by}\qquad K_n(\Gamma_{i,q_n})=H_n^{-1}(\Delta_{i,q_n})\\
& \tilde{K}_n:\T^h/\tilde{\mathcal{M}}_n\to\T^d/\mathcal{M}_n\qquad\qquad\text{defined by}\qquad \tilde{K}_n(\mathfrak{q}_{\infty,n}(\Gamma_{i,q_n}))=\mathfrak{p}_{\infty,n}(H_n^{-1}(\Delta_{i,q_n})) \label{4.5124}
\end{align}
Clearly the above two maps satisfy $\tilde{K}_n=\mathfrak{p}_{n,\infty}\circ K_n\circ\mathfrak{q}_{n,\infty}^{-1}$. We claim that we can choose parameters carefully so that the following condition can be satisfied:
\begin{align} \label{5.436}
K_{n+1}|_{\tilde{\mathcal{F}}_{q_n,q_{n+1}}}=\mathfrak{p}_{n+1,n}\circ K_n\circ\mathfrak{q}_{n+1,n}^{-1}
\end{align}
Before we do that, we make some observation about $K_n$ and $\tilde{K}_n$. First note that we have the following relationship:
\begin{align} 
\tilde{K}_n\circ T^{\a_n}=T_n\circ \tilde{K}_n
\end{align}
Indeed, observe that using proposition \ref{proposition monotonic generating cyclic partition 2}, \ref{4.5124}, proposition \ref{proposition monotonic generating cyclic partition} we get
\begin{align*}
\tilde{K}_n(T^{\a_n}(\mathfrak{q}_{\infty,n}(\Gamma_{i,q_n})))= & \;\tilde{K}_n(\mathfrak{q}_{\infty,n}(\Gamma_{p_n+i\mod q_n,q_n}))\\
= &\; \mathfrak{p}_{\infty,n}(H_n^{-1}(\Delta_{p_n+i\mod q_n,q_n}))\\
= &\;\mathfrak{p}_{\infty,n}(T_n(H_n^{-1}(\Delta_{i,q_n})))\\
= &\; T_n(\mathfrak{p}_{\infty,n}(H_n^{-1}(\Delta_{i,q_n})))\\
= &\;T_n(\tilde{K}_n(\mathfrak{q}_{n,\infty}(\Gamma_{i,q_n})))
\end{align*}
%{\rr It should be $\Gamma_{i,q_n}$ instead of $\Delta_{i,q_n}$ in the last line.}
The next observation we need is 
\begin{align}
\tilde{K}_{n+1}|_{\T^h/\tilde{\mathcal{M}}_n}=\tilde{K}_n
\end{align}
Indeed observe that using \ref{4.5124},
\begin{align*}
\tilde{K}_{n+1}(\mathfrak{q}_{\infty,n}(\Gamma_{i,q_n}))= & \; \tilde{K}_{n+1}(\mathfrak{q}_{\infty,n+1}\circ\mathfrak{q}_{n+1,n}(\Gamma_{i,q_n}))\\
=&\;  \mathfrak{p}_{\infty,n+1}(K_{n+1}(\mathfrak{q}_{n+1,n}(\Gamma_{i,q_n})))\\
=&\;  \mathfrak{p}_{\infty,n+1}(\mathfrak{p}_{n,n+1}(K_n(\Gamma_{i,q_n})))\\
=&\;  \mathfrak{p}_{\infty,n}(K_n(\Gamma_{i,q_n}))\\
=&\;  \tilde{K}_n(\mathfrak{q}_{n,\infty}(\Gamma_{i,q_n}))
\end{align*}

Our job now is to choose parameters correctly in proposition \ref{proposition cyclic approximation of translations} and the  analytic approximation by conjugation scheme and simultaneously construct $K_n$ s satisfying condition \ref{5.436}.

The construction is by induction and assume that we have selected parameters $\a_j,p_j,q_j,\gamma_j,r_j,s_j$ for $j=1,2, \ldots, n $ in proposition \ref{proposition cyclic approximation of translations} so that these satisfy all the conditions in lemma \ref{8.90}. We recall that the parameter $m_n$ in lemma \ref{8.90} can be chosen to be arbitrarily large. We will use this freedom to make the approximation by conjugation scheme work.  Assume the approximation by conjugation scheme has been successfully carried out up to the $n$ th stage. 

At the $n+1$ th stage, we choose parameters and proceed with our construction in the following order:
\begin{enumerate}

\item Choose integer vector $\gamma_{n+1}^h$ and integers $r_n, s_n$ as in lemma \ref{8.90}. 

\item Choose the parameter $k_n:=s_n\gamma_{n+1}^{(h)}$ for the approximation by conjugation scheme.  Next we define the partition:
\begin{align}
&\tilde{\mathcal{F}}_{k_nq_n}\coloneqq\{\Gamma_{i,k_nq_n}:\Gamma_{i,k_nq_n}\coloneqq T^{i/(k_nq_n)}(\cup\{T^{t\gamma_{n+1}}\Gamma_{n+1}: 0\leq t<1/(k_nq_n)\}),\nonumber\\
&\hspace{280pt} i=0,\ldots, k_nq_n-1\}
\end{align}
Note that with this choice of $k_n$ we have $\tilde{\mathcal{F}}_{k_nq_n}>\tilde{\mathcal{F}}_{q_n,q_{n+1}}$. 

\item Choose the functions $a_n$ in the approximation by conjugation scheme. In order to do so, we define the following correspondence:
\begin{align}
\hat{K}_n:\T^h/\tilde{\mathcal{F}}_{k_nq_n}\to\T^d/\mathcal{T}_{k_nq_n}\qquad\qquad\text{defined by}\qquad\hat{K}_n(\Gamma_{i,k_nq_n})=\Delta_{i,k_nq_n}
\end{align}
Now we choose a function $a_n:\{0,\ldots,k_n-1\}\to\{0,\ldots,q_n-1\}$ so that the following equality holds:
\begin{align}
R_{0,q_n}:=\bigcup_{i=0}^{k_n-1}\Delta_{a_n(i)k_n+i,k_nq_n}=\hat{K}_n(\Gamma_{0,q_n,k_nq_{n}})
\end{align}
This allows us to define the partition
\begin{align}
\mathcal{R}_{a_n,k_n,q_n}:=\big\{R_{i,q_n}:R_{i,q_n}:=\phi^{i/q_n}R_{0,q_n}\big\}
\end{align}
Note that $\hat{K}_n(\Gamma_{i,q_n,q_{n+1}})=R_{i,q_n}$ and hence we conclude that $\hat{K}_n\circ T^{\gamma_n/q_n}=\phi^{1/q_n}\circ\hat{K}_n$.

\item  Choose the parameter $l_n$ large enough so that the analytic approximation by conjugation scheme works. 

\item we choose the parameter $m_n$ in proposition \ref{proposition cyclic approximation of translations}. We require that $m_n=\gamma_n^{(h)}l_n$. 

\end{enumerate}

So, after having chosen all the parameters, we can define the partition $\tilde{\mathcal{F}}_{q_{n+1}}$ of $\T^h$ and note that $\tilde{\mathcal{F}}_{q_{n+1}}>\tilde{\mathcal{F}}_{q_n,k_nq_n}$. Now we define the following correspondence:
\begin{align}
\bar{K}_n:\T^h/\tilde{\mathcal{F}}_{q_{n+1}}\to\T^d/\mathcal{T}_{q_{n+1}}\qquad\qquad\text{defined by}\qquad\bar{K}_n(\Gamma_{i,q_{n+1}})=\Delta_{i,q_{n+1}}
\end{align}
and observe that $\bar{K}_n(\Gamma_{i,k_nq_n})=\hat{K}_n(\Gamma_{i,k_nq_n})$. All that remains is to verify that $K_{n+1}$ satisfy \ref{5.436}. So we calculate using definitions and facts from proposition \ref{proposition monotonic generating cyclic partition} and its proof:
\begin{align*}
\mathfrak{p}_{n+1,n}\circ K_{n}\(\Gamma_{i,q_n}\)= &\; \mathfrak{p}_{n+1,n}\circ H_{n}^{-1}\(\Delta_{i,q_n}\)\\
= &\; H_{n+1}^{-1}\circ\mathfrak{c}_{n+1,n}\(\Delta_{i,q_n}\)\\
= &\; H_{n+1}^{-1}\(\bigcup_{j:\Delta_{j,q_{n+1}}\subset R_{i,q_n}}\Delta_{j, q_{n+1}}\)\\
= &\; H_{n+1}^{-1}\(\hat{K}_n(\Gamma_{i,q_n,q_{n+1}})\)\hspace{100pt}\ldots(\text{see item 3 above})\\
= &\; H_{n+1}^{-1}\(\hat{K}_n(\bigcup_{j:\Gamma_{j,k_nq_{n}}\subset  \Gamma_{i,q_n,q_{n+1}}}\Gamma_{j,k_nq_n})\)\\
= &\; H_{n+1}^{-1}\(\bigcup_{j:\Gamma_{j,k_nq_{n}}\subset  \Gamma_{i,q_n,q_{n+1}}}\hat{K}_n(\Gamma_{j,k_nq_n})\)\\
= &\; H_{n+1}^{-1}\(\bigcup_{j:\Gamma_{j,k_nq_{n}}\subset  \Gamma_{i,q_n,q_{n+1}}}\bar{K}_n(\Gamma_{j,k_nq_n})\)\\
= &\; H_{n+1}^{-1}\circ\bar{K}_n\(\Gamma_{i,q_n,q_{n+1}}\)\\
= &\; H_{n+1}^{-1}\circ\bar{K}_n\circ\mathfrak{q}_{n,n+1}\(\Gamma_{i,q_n}\)\\
= &\; {K}_{n+1}\circ\mathfrak{q}_{n,n+1}\(\Gamma_{i,q_n}\)
\end{align*}
This completes the proof
\end{proof}

\section{Minimal diffeomorphisms with a prescribed number of ergodic invariant measures.}
Let $\rho>0$ and $r \in \N$. In order to prove Theorem \ref{theorem prescribed no of measures} we aim at constructing a minimal $T \in\text{Diff }^\omega_\rho(\T^2\, \mu )$ with exactly $r$ ergodic invariant measures. We fix an arbitrary countable set $\Xi= \left\{\rho_i\right\}_{i\in \N}$ of Lipschitz functions that is dense in $C\left(\T^2, \R\right)$. In addition to our usual assumptions we require the number $l_n$ to satisfy
\begin{equation} \label{cond l birk}
l_n > n^2 \cdot \|DH^{-1}_n \|_0 \cdot \max_{i=1,\ldots, n} \text{Lip}(\rho_i),
\end{equation}
where $\text{Lip}(\rho)$ is the Lipschitz constant of $\rho$.

First of all, we show that a permutation $\Pi$ of the partition $\mathcal{S}_{kq,l} $ which commutes with $\phi^{1/q}$ is a block slide type of map. This property will be required in the construction of our conjugation map in subsection \ref{subsec:constrmin}: $h_n = h_{\mathfrak{1},n} \circ h_{\mathfrak{2},n}$. This time there are different parts of the torus $\mathbb{T}^2$ introduced with distinct aims. On the one hand, we will divide it into $r$ sets $N_t$ by requirements on the $x_2$-coordinate. Each set naturally supports an absolutely continuous probability measure $\mu_t$ given by the normalized restriction of the Lebesgue measure $\mu$. These will enable us to build the ergodic invariant measures as the limits $\xi_t$ of the sequence $\xi^n_t \coloneqq \left(H_n\right)^{\ast} \mu_t$. \\
On the other hand, we will use stripes corresponding to small parts of the $x_1$-axis on which the conjugation map $h^{-1}_{\mathfrak{2},n}$ will intermingle the sets $\tilde{N}_t$ to prove minimality of the limit diffeomorphism $T$. These parts are measure theoretically insignificant because the measure of these sets will converge to zero as $n \rightarrow \infty$. \\
In order to achieve these aims we need the so-called trapping map $h^{-1}_{\mathfrak{1},n}$ introduced in subsection \ref{subsec:constrmin}. On the ``minimality'' - part, this map captures parts of every orbit $\left\{\phi^{\alpha_{n}} \circ H_n\left(x\right)\right\}_{k=0,...,q_{n}-1}$ so that the conjugation map $h^{-1}_{\mathfrak{2},n}$ can spread it over the almost whole manifold. Then we can prove minimality in chapter \ref{min} by arguing that every element in a family of sufficiently small cubes covering the whole manifold is met by the orbit $\left\{h^{-1}_{n}  \circ \phi^{k \alpha_{n}} \circ H_n\left(x\right)\right\}_{k=0,...,q_{n}-1}$ and the image of any cube under $H^{-1}_{n-1}$ has a small diameter, which converges to $0$ as $n\rightarrow \infty$. In addition the trapping map is used to gain control of almost everything of every orbit $\left\{H^{-1}_n \circ \phi^{k\alpha_{n}}  \left(x\right)\right\}_{k=0,...,q_{n}-1}$. This allows us to prove a convergence result on Birkhoff sums (see Lemma \ref{lem:birk}), which in turn enables us to exclude the existence of further ergodic invariant measures besides the previously mentioned $\xi_t$.

\subsection{Approximation of arbitrary permutations}

Suppose we have three natural numbers  $l$, $k$ and $q$, and a permutation $\Pi$ of the partition  $\mathcal{S}_{kq,l} $ which commutes with $\phi^{1/q}$. Our objective here is to show that $\Pi$ is a block slide type of map. This will be achieved in two steps. In the first step we show that there exists a product of two 2-cycles and then we will prove that all transpositions are block-slide type of maps:

\subsubsection*{Product of two 2-cycles}

\begin{figure}
\centering
\begin{tabular}{c}
{\includegraphics[width=7.5cm, height=15cm]{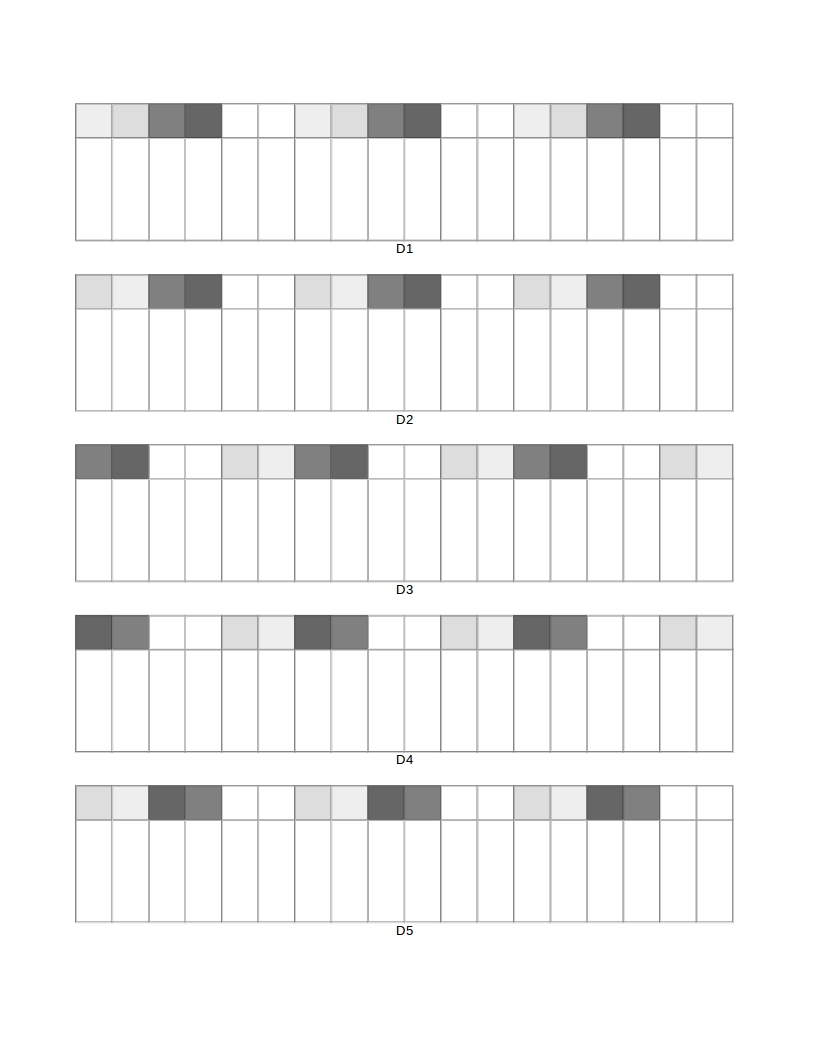}}
\end{tabular}
\caption{Illustration of the action of $\mathfrak{g}_{k,q,l}$ on $\T^2$ with $l=4$, $k=6$ and $q=3$. We note how the blocks are moved by the intermediate maps: $D1\xrightarrow{\mathfrak{g}_{k,q,l}^{(\mathfrak{1})}} D2\xrightarrow{\mathfrak{g}_{k,q,l}^{(\mathfrak{2})}} D3\xrightarrow{\mathfrak{g}_{k,q,l}^{(\mathfrak{3})}} D4\xrightarrow{\mathfrak{g}_{k,q,l}^{(\mathfrak{4})}} D5 $. Note that every other atoms apart from the four we flipped are kept constant at the end.}
\label{2 cycle}
\end{figure}

We now show that for any choice of natural numbers $l,k$ and $q$ there exists a block-slide type of map which has the same effect as the product of two 2-cycles in the symmetric group of $lkq$ elements. 

In order to make this precise we need the following notation. For any $i=0,\ldots, kq-1$ and $j=0,\ldots, s-1$ we define 
\begin{align}
\mathcal{S}_{kq,l}:=\{S_{i,j}^{kq,l}:=[i/(kq),(i+1)/(kq))\times[j/l,(j+1)/l), 0\leq i< kq, 0\leq j< l\}
\end{align}

First we define the following step function:
\begin{align}
& \s^{(\mathfrak{4})}_{kq}:(0,1]\to\R\qquad\qquad\text{defined by}\qquad \s^{(\mathfrak{4})}_{kq}(t)=\begin{cases} 2/(kq) &  \quad\text{if } t \in ((l-1)/l,1] \\ 0 & \quad\text{if } t\in (0,(l-1)/l]\end{cases} 
\end{align}
%{\rr shift by $\frac{2}{kq}$?} Yes
And then the following two maps of block-slide type:
\begin{align}
& \mathfrak{g}_{kq,l}^{(\mathfrak{1})}:\T^2\to\T^2\qquad\qquad\text{defined by}\qquad\mathfrak{g}_{kq,l}^{(\mathfrak{1})}\big((x_1,x_2)\big)=(x_1 - \s^{(\mathfrak{4})}(x_2),x_2)\\
& \mathfrak{g}_{kq,l}^{(\mathfrak{2})}:\T^2\to\T^2\qquad\qquad\text{defined by}\qquad\mathfrak{g}_{kq,l}^{(\mathfrak{2})}\big((x_1,x_2)\big)=(x_1 + \s^{(\mathfrak{4})}(x_2),x_2)
\end{align}
Finally we piece everything together and define the following block-slide type of map
\begin{align}
& \mathfrak{g}_{k,q,l}:\T^2\to\T^2\qquad\qquad\text{defined by}\qquad\mathfrak{g}_{k,q,l}=\mathfrak{g}_{k,q,l}^{(\mathfrak{2})}\circ\mathfrak{f}_{0,k,q}\circ\mathfrak{g}_{k,q,l}^{(\mathfrak{1})}\circ\mathfrak{f}_{0,k,q}
\end{align}
using the map $\mathfrak{f}_{0,k,q}$ from section \ref{constr transl}.

We end this section after noting that the above block-slide type of map takes $S_{0,l-1}^{kq,l}\to S_{1,l-1}^{kq,l}$, $S_{1,l-1}^{kq,l}\to S_{0,l-1}^{kq,l}$, $S_{2,l-1}^{kq,l}\to S_{3,l-1}^{kq,l}$, $S_{3,l-1}^{kq,l}\to S_{2,l-1}^{kq,l}$ and acts as an identity everywhere else. This is the same as the product of two $2$-cycles in the symmetric group on a set of $k\times l$ elements.

\subsubsection*{Transposition}

\begin{figure}
\centering
\begin{tabular}{c c}
{\includegraphics[width=7cm, height=15cm]{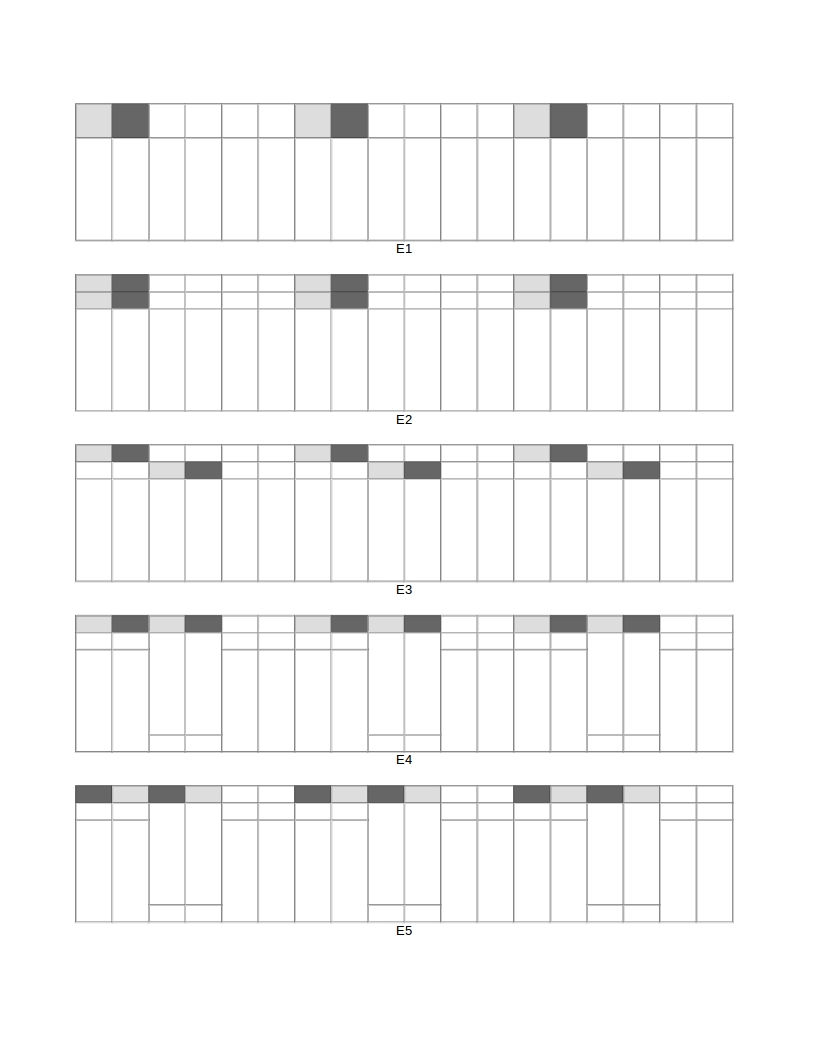}} & {\includegraphics[width=7cm, height=15cm]{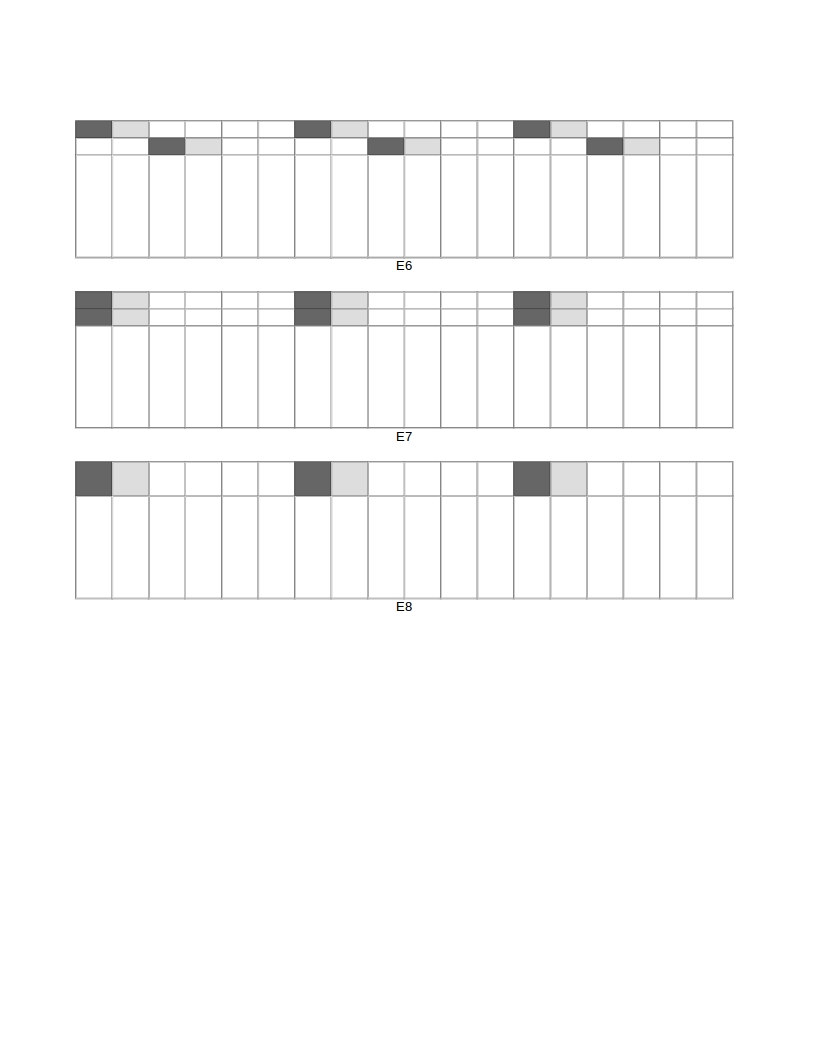}}
\end{tabular}
\caption{Illustration of the action of $\mathfrak{h}_{k,q,l}$ on $\T^2$ with  $l=4$, $k=6$ and $q=3$. We note how the blocks are moved by the intermediate maps: $E2\xrightarrow{\mathfrak{h}_{6,3,4}^{(\mathfrak{1})}} E3\xrightarrow{\mathfrak{h}_{6,3,4}^{(\mathfrak{3})}} E4\xrightarrow{\mathfrak{g}_{6,3,4}} E5\xrightarrow{\mathfrak{h}_{6,3,4}^{(\mathfrak{4})}} E6\xrightarrow{\mathfrak{h}_{6,3,4}^{(\mathfrak{2})}} E7 $. Note that every other atoms apart from the two we flipped are kept constant at the end.}
\label{transposition}
\end{figure}

Finally we show that there exists a block-slide on the torus which switches two blocks and leaves all other invariant. Unfortunately if we work with the partition $\mathcal{S}_{kq,l}$, we do not know if such a map exists. The way we circumnavigate this problem is to go to a finer partition, namely $\mathcal{S}_{kq,2l}$ and show that with some care, a product of two 2-cycles in $\mathcal{S}_{kq,2l}$ is a transposition in $\mathcal{S}_{kq,l}$

First we define the following two step functions: %{\rr In the definition of $\s^{(\mathfrak{5})}_{kq,l}$: $t \in ((2l-1)/(2l),1]$ and so on?} Yes
\begin{align}
& \s^{(\mathfrak{5})}_{kq,l}:(0,1]\to\R\qquad\qquad\text{defined by}\qquad \s^{(\mathfrak{5})}_{kq,l}(t)=\begin{cases} 0 &  \quad\text{if } t \in ((2l-1)/(2l),1] \\ 2/(kq) &  \quad\text{if } t \in ((2l-2)/(2l),(2l-1)/(2l)] \\ 0 & \quad\text{if } t\in (0,(2l-2)/(2l)]\end{cases}\\
& \s^{(\mathfrak{6})}_{kq,l}:(0,1]\to\R\qquad\qquad\text{defined by}\qquad \s^{(\mathfrak{6})}_{kq,l}(t)=\begin{cases} 0 &  \quad\text{if } qt\mod 1 \in (0,2/k] \\ 1/(2l) &  \quad\text{if } qt\mod 1 \in (2/k,4/k] \\ 0 & \quad\text{if } qt\mod 1 \in (4/k,1]\end{cases}  
\end{align}
And then the following four block-slide type of map:
\begin{align}
& \mathfrak{h}_{kq,l}^{(\mathfrak{1})}:\T^2\to\T^2\qquad\qquad\text{defined by}\qquad\mathfrak{h}_{kq,l}^{(\mathfrak{1})}\big((x_1,x_2)\big)=(x_1 + \s^{(\mathfrak{5})}_{kq,l}(x_2),x_2)\\
& \mathfrak{h}_{kq,l}^{(\mathfrak{2})}:\T^2\to\T^2\qquad\qquad\text{defined by}\qquad\mathfrak{h}_{kq,l}^{(\mathfrak{2})}\big((x_1,x_2)\big)=(x_1 - \s^{(\mathfrak{5})}_{kq,l}(x_2),x_2)\\
& \mathfrak{h}_{kq,l}^{(\mathfrak{3})}:\T^2\to\T^2\qquad\qquad\text{defined by}\qquad\mathfrak{h}_{kq,l}^{(\mathfrak{3})}\big((x_1,x_2)\big)=(x_1, x_2 + \s^{(\mathfrak{6})}_{kq,l}(x_1))\\
& \mathfrak{h}_{kq,l}^{(\mathfrak{4})}:\T^2\to\T^2\qquad\qquad\text{defined by}\qquad\mathfrak{h}_{kq,l}^{(\mathfrak{4})}\big((x_1,x_2)\big)=(x_1, x_2 - \s^{(\mathfrak{6})}_{kq,l}(x_1))
\end{align}
Finally we piece everything together and define the following block-slide type of map
\begin{align}
& \mathfrak{h}_{k,q,l}:\T^2\to\T^2\qquad\qquad\text{defined by}\qquad\mathfrak{h}_{k,q,l} \coloneqq  \mathfrak{h}_{k,q,l}^{(\mathfrak{2})}\circ\mathfrak{h}_{k,q,l}^{(\mathfrak{4})}\circ\mathfrak{g}_{k,q,2l}\circ\mathfrak{h}_{k,q,l}^{(\mathfrak{3})}\circ\mathfrak{h}_{k,q,l}^{(\mathfrak{1})}
\end{align}
%{\rr Should it be $\mathfrak{g}_{k,q,2l}$?} Yes
More generally we can define for any $(\mathfrak{i},\mathfrak{j})\neq (0,l-1)$, the following block-slide type of map:
\begin{align}
& \mathfrak{h}_{k,q,l}^{(\mathfrak{i},\mathfrak{j})}:\T^2\to\T^2\qquad\text{defined by}\\
&\hfill\qquad\mathfrak{h}_{k,q,l}^{(\mathfrak{i},\mathfrak{j})} \coloneqq \phi^{(\mathfrak{i}-1)/(kq)}\circ(\mathfrak{h}_{kq,l}^{(\mathfrak{4})})^{2(l-1-\mathfrak{j})}\circ\mathfrak{h}_{k,q,l}\circ(\mathfrak{h}_{k,q,l}^{(\mathfrak{3})})^{2(l-1-\mathfrak{j})}\circ\phi^{-(\mathfrak{i}-1)/(kq)}
\end{align}
%{\rr $i$ and $j$ mixed up? Also $\phi^{\mathfrak{j}/(kq)}$ as first map? }Yes
We end this section by observing that $\mathfrak{h}_{k,q,l}^{(\mathfrak{i},\mathfrak{j})}$ maps $S_{0,l-1}^{kq,l}\to S_{\mathfrak{i},\mathfrak{j}}^{kq,l}$, $S_{\mathfrak{i},\mathfrak{j}}^{kq,l}\to S_{0,l-1}^{kq,l}$ and acts as identity everywhere else. So we obtained all transpositions of the form $(1,n)$ in the symmetric group on a set of $kl$ elements.

\subsubsection*{All permutations are block-slide type of maps}

We now show that any permutation which commutes with $\phi^{1/q}$ is a block-slide type of map.

\begin{maintheorem} \label{permutation = block-slide}
Let $\Pi$ be any permutation of $kql$ elements. We can naturally consider $\Pi$ to be a permutation of the partition $\mathcal{S}_{kq,l}$ of the torus $\T^2$. Assume that $\Pi$ which commutes with $\phi^{1/q}$. Then $\Pi$ is a block-slide type of map.
\end{maintheorem}

\begin{proof}
Follows from the fact that all permutations are generated by transpositions. 
\end{proof}

\subsection{Description of the required combinatorics} \label{subsec:constrmin}

Here we prove theorem \ref{theorem prescribed no of measures}. We begin by describing the combinatorics we need at the $n+1$ th stage of the induction process abstractly.

For $t=0,\ldots, r-1$, we consider the following subsets of $\T^2$:
\begin{align}
    N_t\coloneqq \T^1\times \Big[\frac{t}{r},\frac{t+1}{r}\Big)
\end{align}
We denote the restriction of the Lebesgue measure $\mu$ to $N_t$ by $\mu_t$.

\begin{figure}
\centering
\begin{tabular}{c}
{\includegraphics[scale = .5, trim={0 18cm 0 0},clip]{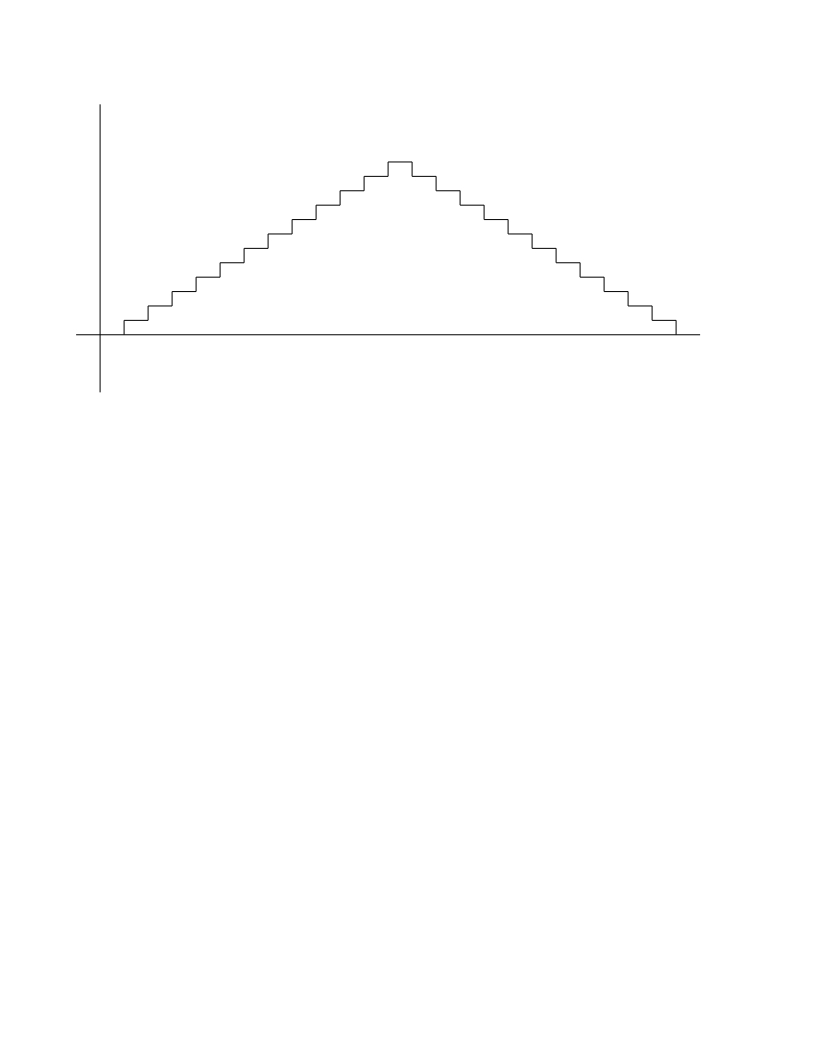}}
\end{tabular}
\caption{Picture of $\kappa^{(\mathfrak{1})}$ with $n=5$ drawn over the interval $(0,1/(n^2l^3q)]$ on the $x$-axis.}
\label{step function}
\end{figure}

For natural numbers $n,l$ and $q$ we define the following partition of the torus $\T^2$:
\begin{align}
    \mathcal{G}_{l^3q}\coloneqq \Big\{G_{i,j, l^3q}:G_{i,j,l^3q}\coloneqq \Big[\frac{i}{l^3q}, \frac{i+1}{l^3q}\Big)\times\Big[\frac{j}{lr},\frac{j+1}{lr}\Big), 0\leq i< l^3q, 0\leq j<lr\Big\}
\end{align}
We define the following permutation of the above partition: 
\begin{align}
    \mathfrak{h}^{(\mathfrak{2})}:\T^2\to\T^2
\end{align}
which acts on the atoms of partition $\mathcal{G}_{l^3q}$ that are contained in $[0,1/(lq)) \times \T^1$ in the following way (for $t=0, \ldots,r-1$) 
%\begin{align}
   % & (\mathfrak{h}^{(\mathfrak{2})})^{-1}\Big(\bigcup_{j=0}^{lr}G_{i,j,n^2l^2q}\Big)   \to    \bigcup_{j'=r(i\text{ mod }l)}^{r(i\text{ mod }l)+r-1}\Big(\bigcup_{i'=\lfloor i/l\rfloor}^{\lfloor i/l\rfloor +l -1}G_{i,j,n^2l^2q}\Big)\qquad\text{if}\quad 0\leq i<l^2\\
  %  & (\mathfrak{h}^{(\mathfrak{2})})^{-1}\Big(\bigcup_{t=1}^{l}G_{i,jl+t,n^2l^2q}\Big)  \to    \Big(\bigcup_{t=\lfloor i/l\rfloor l}^{\lfloor i/l\rfloor l +l -1}G_{t,jl+i\text{ mod } l,n^2l^2q}\Big)\qquad\text{if}\quad l^2\leq i<n^2l^2
%\end{align}
\begin{align}
   &\text{If } 0\leq i <l: \quad  (\mathfrak{h}^{(\mathfrak{2})})^{-1}\Big(G_{i,j,l^3q}\Big) =  G_{i',j',l^3q}, \qquad \text{where } i'= \lfloor \frac{j}{r} \rfloor, \ j'=r \cdot i +j \mod r, \\
  &\text{if } l \leq i <l^2: \quad  (\mathfrak{h}^{(\mathfrak{2})})^{-1}\Big(G_{i,tl+j,l^3q}\Big) =  G_{i',tl+j',l^3q}, \qquad \text{where } i'= \lfloor \frac{i}{l} \rfloor \cdot l + j, \ j'=i \mod l.
\end{align}
We extend this permutation to the whole of $\T^2$ equivariantly. Since the above description is not very clear, we give a somewhat imprecise but more demonstrative description of the above map. Note that the following rectangles get mapped in the following way:
\begin{align*}
     & (\mathfrak{h}^{(\mathfrak{2})})^{-1} \Big(\Big[\frac{i}{l^3q},\frac{i+1}{l^3q}\Big)\times\Big[0,1\Big)\Big)= \Big[0,\frac{1}{l^2q}\Big)\times\Big[\frac{i}{l},\frac{i+1}{l}\Big)\qquad\text{if}\quad 0\leq i<l\\
    & (\mathfrak{h}^{(\mathfrak{2})})^{-1} \Big(\Big[\frac{i}{l^3q},\frac{i+1}{l^3q}\Big)\times\Big[\frac{t}{r},\frac{t+1}{r}\Big)\Big)= \Big[\frac{i'}{l^2q},\frac{i'+1}{l^2q}\Big)\times\Big[\frac{tl+j'}{lr},\frac{(t+1)l+j'+1}{lr}\Big)\quad\text{if}\quad l\leq i<l^2,
\end{align*}
where $i'= \lfloor \frac{i}{l} \rfloor$ and $j'= i \mod l$. Notice that in the first region narrow rectangular stripes of full height get squished and are distributed over the full height of the torus which will allow us to prove minimality. While all other rectangles are mapped to rectangles of small diameter but they remain within the horizontal strip $N_t$ on the torus. These stripes will form the support of a preimage of the invariant measures. \\
By the previous subsection we know that this is a block slide type of map and hence allows good analytic approximations by Proposition \ref{proposition approximation}. We denote this $(\varepsilon, \delta)$-approximation by $h^{(\mathfrak{2})}$, the corresponding ``bad set'' by $E$ and set $F=\T^2 \setminus E$.

\begin{figure}
\centering
\begin{tabular}{c}
{\includegraphics[height =15cm, width = 14cm, trim={0 6cm 3cm 0},clip]{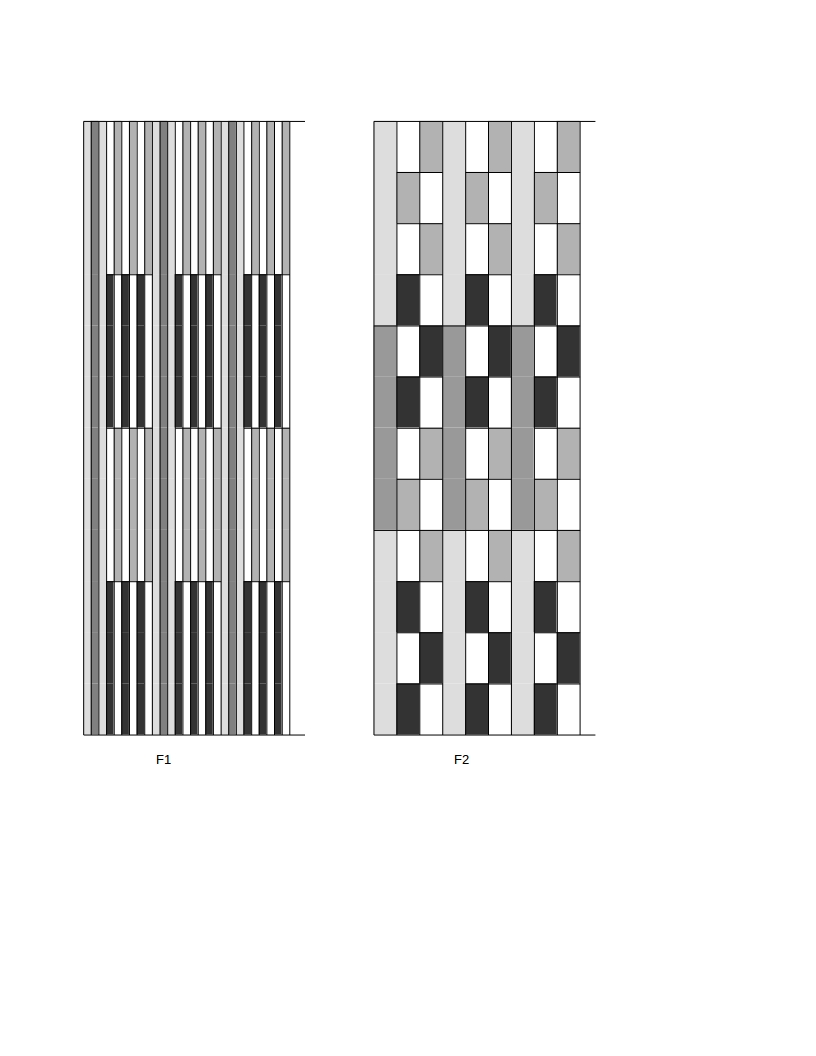}}
\end{tabular}
\caption{Illustration of the action of $(\mathfrak{h}^{(\mathfrak{2})})^{-1}$ on the rectangle $[0,1/q)\times [0,1)$ inside the torus $\T^2$. The combinatorics depicted here is repeated with periodicity $1/q$. We used $r=4$, and $l=3$ for this diagram. }
\label{minimal with specified number of measures}
\end{figure}

For this number $\delta$ and given natural numbers $q,l,n$, we define the following step function: 
\begin{align}
    \tilde{\kappa}^{(\mathfrak{1})}:[0,\frac{1}{l^2q})\to\R\qquad\text{defined by}\qquad \tilde{\kappa}^{(\mathfrak{1})}(x)=\begin{cases} 0 &\quad\text{if}\quad x\in [0,\frac{1}{n^2l^3q})\\\frac{\d}{lr} &\quad\text{if}\quad x\in [\frac{1}{n^2l^3q},\frac{2}{n^2l^3q})\\
    2\frac{\d}{lr} &\quad\text{if}\quad x\in [\frac{1}{n^2l^3q},\frac{2}{n^2l^3q})\\
    \ldots &\quad \ldots\\\ldots & \quad \ldots\\
    (\lfloor \frac{n^2}{2}\rfloor -2)\frac{\d}{lr} &\quad\text{if}\quad x\in [\frac{\lfloor \frac{n^2}{2}\rfloor-1}{n^2l^3q},\frac{\lfloor \frac{n^2}{2}\rfloor}{n^2l^3q})\\
    (\lfloor \frac{n^2}{2}\rfloor -1)\frac{\d}{lr} &\quad\text{if}\quad x\in [\frac{\lfloor \frac{n^2}{2}\rfloor}{n^2l^3q},\frac{\lfloor \frac{n^2}{2}\rfloor+1}{n^2l^3q})\\
    (\lfloor \frac{n^2}{2}\rfloor -2)\frac{\d}{lr} &\quad\text{if}\quad x\in [\frac{\lfloor \frac{n^2}{2}\rfloor+1}{n^2l^3q},\frac{\lfloor \frac{n^2}{2}\rfloor + 2}{n^2l^3q})\\
    \ldots &\quad \ldots\\\ldots & \quad \ldots\\
    \frac{\d}{lr} &\quad\text{if}\quad x\in [\frac{n^2-2}{n^2l^3q},\frac{n^2-1}{n^2l^3q})\\
    0 &\quad\text{if}\quad x\in [\frac{n^2-1}{n^2l^3q},\frac{1}{l^3q})\end{cases}
\end{align}
Let $\kappa^{\mathfrak{1}}$ be the $\frac{1}{l^3q}$-periodic real-analytic $\left( \tilde{\varepsilon}, \tilde{\delta} \right)$-approximation of $\tilde{\kappa}^{(\mathfrak{1})}$. With the aid of this we define
\begin{align}
    h^{(\mathfrak{1})}:\T^2\to\T^2\qquad\text{defined by}\qquad  h^{(\mathfrak{1})}(x_1,x_2)\coloneqq (x_1, x_2+\kappa^{(\mathfrak{1})}(x_1))
\end{align}
and we often refer to the above map as ``trapping map''. The purpose of this map is to capture a large portion of \emph{every} $\phi$ orbit. \\
Hereby we introduce the so-called ``trapping zones'' (for $t=0, \dots,r-1$ and $s=0, \dots, lq-1$)

\begin{align}
    & A_{s,i} = h^{(\mathfrak{1})}\Big( \phi^{\frac{s}{lq}}\big(\bigcup_{j=0}^{lr-1}G_{i,j,l^3q}\big) \cap F \Big)  \qquad\text{if}\quad 0\leq i<l\\
    & B^t_{s,i} = h^{(\mathfrak{1})}\Big(\phi^{\frac{s}{lq}}\big(\bigcup_{j=0}^{l-1}G_{i,tl+j,l^3q} \big) \cap F\Big) \qquad\text{if}\quad l\leq i<l^2
\end{align}

In our specific constructions we define $h_{n+1}=h_{\mathfrak{1},n+1} \circ h_{\mathfrak{2},n+1}$ using the parameters $q=q_n$, $l=l_n$, $\varepsilon<\frac{\varepsilon_n}{2^{l_nq_n}}$, $\delta =\delta_n < \frac{1}{n^4 \cdot 2^{l_nq_n}}$, $\tilde{\varepsilon} = \tilde{\varepsilon}_n < \frac{\delta_n}{2^{l_nq_n}}$ and $\tilde{\delta}= \tilde{\delta}_n < \frac{\delta_n}{2^{l_nq_n}}$.

As announced we have the following trapping property:

\begin{lemma} \label{lem trap}
Let $x \in \T^2$ be arbitrary. Then the orbit $\left\{ \phi^{k \alpha_{n+1}}(x) \right\}_{k=0, \ldots, q_{n+1}}$ meets every set $A_{s,i}$. Moreover, for every $B^t_{s,i}$ at least $\omega^n_t(x) \cdot \frac{\left(1-\frac{8}{n^2}\right) \cdot q_{n+1}}{l^3_n q_n}$ iterates of the orbit $\left\{ \phi^{k \alpha_{n+1}}(x) \right\}_{k=0, \ldots, q_{n+1}}$ lie in $B^t_{s,i}$, where $\omega^n_t(x)$ does not depend on $s,i$. On the contrary at most $\frac{10}{n^2}  \cdot q_{n+1}$ iterates are not captured by the collection of sets $B^t_{s,i}$.
\end{lemma}

\begin{proof}
Let $x=\left(x_1, x_2 \right) \in \T^2$ and $i \in \left\{0,\ldots, l^3q-1\right\}$ be arbitrary. Note that 
\begin{equation}
G_{i,j,l^3q} \cap F \supseteq \left[ \frac{i + \frac{\delta}{2}}{l^3q},  \frac{i + 1- \frac{\delta}{2}}{l^3q} \right] \times \left[ \frac{j+\frac{\delta}{2}}{lr}, \frac{j+1-\frac{\delta}{2}}{lr} \right]
\end{equation}
by our approximation Proposition \ref{proposition approximation}. Due to $\frac{n^2 \cdot \delta_n}{l_nr}< \frac{1}{n^2 l_n r}$ and our choice of the approximative step function $\kappa^{(\mathfrak{1})}$ there are at most four sections $\left[ \frac{i}{l^3q}+ \frac{u+\frac{\tilde{\d}}{2}}{n^2l^3q},  \frac{i}{l^3q}+ \frac{u+1-\frac{\tilde{\d}}{2}}{n^2l^3q} \right]$, where $u \in \left\{1,\ldots,n^2-2\right\}$, on an arbitrary $\left[ \frac{i + \frac{\delta}{2}}{l^3q},  \frac{i + 1- \frac{\delta}{2}}{l^3q} \right]$-section such that $x_2$ does not belong to any of the $h^{(\mathfrak{1})}\Big( \big(G_{i,j,l^3q}\big) \cap F \Big)$-domains for $j=0,\ldots, lr-1$. Since $\left\{ k \cdot \alpha_{n+1} \right\}_{k=0,\ldots, q_{n+1}-1}$ is equidistributed on $\mathbb{S}^1$, the number of iterates $k$, such that $\left\{ \phi^{k \alpha_{n+1}}(x) \right\}_{k=0, \ldots, q_{n+1}-1}$ is captured by one of these domains is at least 
\begin{equation}
\lfloor \frac{\left(1-\frac{6}{n^2}\right) \cdot \left(1-\tilde{\delta}_n\right) \cdot q_{n+1}}{l^3_n q_n} \rfloor \geq \frac{\left(1-\frac{8}{n^2}\right) \cdot q_{n+1}}{l^3_n q_n}.
\end{equation}
Depending on the point $x \in \T^2$ there is a portion $\omega^n_t(x)$ of these iterates spent in trapping regions $B^t_{s,i}$ belonging to $N_t$. This portion does not depend on the indices $s,i$. Since there are $l_nq_n \cdot \left(l^2_n-l_n \right)$ such indices, the last claim follows. 
\end{proof}

\subsection{Proof of minimality} \label{min}

\subsubsection*{Criterion for minimality} 
We recall the notion of a minimal dynamical system:
\begin{definition}
Let $X$ be a topological space and $f: X \rightarrow X$ be a continuous transformation. The map $f$ is called minimal if for every $x \in X$ the orbit $\left\{f^i\left(x\right)\right\}_{i \in \mathbb{N}}$ is dense in $X$.
\end{definition}
Equivalently $f$ is minimal if for every $x \in X$ and every non-empty open set $U \subseteq X$ there is $i \in \mathbb{N}$ such that $f^i\left(x\right) \in U$. In the case of $X$ being a metric space every open set contains an $\gamma$-ball for $\gamma$ sufficiently small. Thus, $f$ is minimal if for every $x \in X$, every $\gamma >0$ and for every $\gamma$-ball $B_{\gamma}$ there is $i \in \mathbb{N}$ such that $f^i\left(x\right) \in B_{\gamma}$. Hereby, we can deduce the subsequent criterion of minimality in the setting of our constructions:
\begin{lemma} \label{lem:critmin}
Suppose that the set of iterates $\left\{h^{-1}_{n+1} \circ \phi^{i \cdot \alpha_{n+1}} \circ H_{n+1}\left(x\right)\right\}_{i=0,...,q_{n+1}-1}$ meets every set of the form $\left[ \frac{j_1}{l_nq_n},  \frac{j_1 + 1}{l_nq_n}\right] \times \left[ \frac{j_2}{l_n},  \frac{j_2 + 1}{l_n}\right]$ for every $x \in \mathbb{T}^2$. Moreover, we assume that the sequence $\left(T_n\right)_{n \in \mathbb{N}}$ constructed as in section \ref{subsec:constrmin} converges to a diffeomorphism $T$ in the Diff$^{\omega}_{\rho}$-topology and satisfies $d_0\left(T^i,T^i_{n+1}\right)<\frac{1}{2^n}$ for all $i=0,...,q_{n+1}-1$. Then $T = \lim_{n\rightarrow \infty} T_n$ is minimal.
\end{lemma}

\begin{proof}
At first we observe that
\begin{equation*}
\text{diam} \left(H^{-1}_{n}\left(\left[ \frac{j_1}{l_nq_n},  \frac{j_1 + 1}{l_nq_n}\right] \times \left[ \frac{j_2}{l_n},  \frac{j_2 + 1}{l_n}\right]\right)\right) \leq \left\|DH_{n}\right\|_0 \cdot \frac{2}{l_n},
\end{equation*}
which converges to $0$ as $n\rightarrow \infty$ (because of $\left\|DH_{n} \right\|_0 < \frac{l_n}{2^n}$ by equation \ref{ln criterion}), and that the family of sets $\left[ \frac{j_1}{l_nq_n},  \frac{j_1 + 1}{l_nq_n}\right] \times \left[ \frac{j_2}{l_n},  \frac{j_2 + 1}{l_n}\right]$ covers the whole space $\T^2$. Hence, for every $\varepsilon >0$ and $y \in \mathbb{T}^m$ there is $M_1 \in \mathbb{N}$ such that for every $n \geq M_1$ there exists a set $H^{-1}_{n}\left(\left[ \frac{j_1}{l_nq_n},  \frac{j_1 + 1}{l_nq_n}\right] \times \left[ \frac{j_2}{l_n},  \frac{j_2 + 1}{l_n}\right]\right) \subseteq B_{\frac{\varepsilon}{2}}\left(y\right)$. \\
Let $x \in \T^2$, $\varepsilon >0$ and an $\varepsilon$-ball $B_{\varepsilon}\left(y\right)$, at which $y \in \mathbb{T}^2$, be arbitrary. Since $d_0\left(T^i,T^i_{n+1}\right)<\frac{1}{2^n}$ for all $i=0,...,q_{n+1}-1$  there is $M_2 \in \mathbb{N}$ such that $d_0\left(T^i,T^i_{n+1}\right)<\frac{\varepsilon}{2}$ for all $i=0,...,q_{n+1}-1$ and $n\geq M_2$. \\
We consider $n\geq \tilde{N} \coloneqq \max\left\{M_1,M_2\right\}$. Then there is a set $H^{-1}_{n}\left(\left[ \frac{j_1}{l_nq_n},  \frac{j_1 + 1}{l_nq_n}\right] \times \left[ \frac{j_2}{l_n},  \frac{j_2 + 1}{l_n}\right]\right)\subseteq B_{\frac{\varepsilon}{2}}\left(y\right)$ and by assumption an $i < q_{n+1}$ such that $T^i_{n+1}\left(x\right) \in H^{-1}_{n}\left(\left[ \frac{j_1}{l_nq_n},  \frac{j_1 + 1}{l_nq_n}\right] \times \left[ \frac{j_2}{l_n},  \frac{j_2 + 1}{l_n}\right]\right)\subseteq B_{\frac{\varepsilon}{2}}\left(y\right)$. By the triangle inequality we obtain
\begin{equation*}
d\left(T^i\left(x\right), y\right) \leq d\left(T^i\left(x\right), T^i_{n+1}\left(x\right)\right) + d\left(T^i_{n+1}\left(x\right), y\right) \leq d_0\left(T^i, T^i_{n+1}\right) + \frac{\varepsilon}{2} < \varepsilon.
\end{equation*}
Thus, we conclude $T^i\left(x\right) \in B_{\varepsilon}\left(y\right)$. Hence, $T$ is minimal. 
\end{proof}

\subsubsection*{Application of the criterion}
The conditions on the convergence of the sequence $\left(T_n\right)_{n \in \mathbb{N}}$ and proximity $d_0\left(T^i,T^i_{n+1}\right)<\frac{1}{2^n}$ for all $i=0,...,q_{n+1}-1$ are fulfilled by Remark \ref{close iterates}. Let $x \in \mathbb{T}^2$ and $\left[ \frac{j_1}{l_nq_n},  \frac{j_1 + 1}{l_nq_n}\right] \times \left[ \frac{j_2}{l_n},  \frac{j_2 + 1}{l_n}\right]$ be arbitrary. We have to show that the orbit $\left\{h^{-1}_{n+1} \circ \phi^{i\alpha_{n+1}} \circ H_{n+1}\left(x\right)\right\}_{i=0,...,q_{n+1}-1}$ meets $\left[ \frac{j_1}{l_nq_n},  \frac{j_1 + 1}{l_nq_n}\right] \times \left[ \frac{j_2}{l_n},  \frac{j_2 + 1}{l_n}\right]$. For this purpose, we note that there is $i \in \left\{0,...,q_{n+1}-1\right\}$ with $\phi^{i\alpha_{n+1}}\circ H_{n+1}\left(x\right) \in A_{j_1,j_2}$ by Lemma \ref{lem trap}. Then we compute
\begin{align*}
h^{-1}_{n+1}\left(A_{j_1,j_2} \right) & = h^{-1}_{2,n+1} \left(\phi^{\frac{j_1}{l_nq_n}}\big(\bigcup_{j=0}^{l_nr-1}G_{j_2,j,l^3_nq_n}\big) \cap F_n\right) = \phi^{\frac{j_1}{l_nq_n}} \circ h^{-1}_{2,n+1} \left(\bigcup_{j=0}^{l_nr-1}G_{j_2,j,l^3_nq_n} \cap F_n\right) \\
& \subset \phi^{\frac{j_1}{l_nq_n}} \left( \left[0, \frac{1}{l^2_nq_n}\right] \times \left[\frac{j_2}{l_n}, \frac{j_2 + 1}{l_n}\right] \right) \subset \left[\frac{j_1}{l_nq_n}, \frac{j_1+1}{l_nq_n}\right] \times \left[\frac{j_2}{l_n}, \frac{j_2 + 1}{l_n}\right]
\end{align*}
and we can apply Lemma \ref{lem:critmin} to prove the minimality of $T$.

\subsection{The ergodic invariant measures}

\subsubsection*{Construction of the measures} \label{subsubsec:constrm}
As announced we will construct the ergodic invariant measures with the aid of the normalized restrictions $\mu_t$ of the Lebesgue measure on the sets $N_t$, i.e. $\mu_t\left(A\right) = \frac{\mu\left(A \cap N_t\right)}{\mu\left(N_t\right)}$ for any measurable set $A \subseteq \mathbb{T}^2$. Since each set $N_t$ is $\phi^{\beta}$-invariant for any $\beta \in \mathbb{S}^1$, we have $\left(\phi^{\beta}\right)_{\ast} \mu_t = \mu_t$. With these we define the measures $\xi^n_t \coloneqq \left(H^{-1}_n\right)_{\ast} \mu_t$ and can prove their $T_n$-invariance:
\begin{equation*}
\left(T_n\right)_{\ast} \xi^n_t = \left(T_n\right)_{\ast} \left( \left(H^{-1}_n\right)_{\ast} \mu_t\right) = \left(T_n \circ H^{-1}_n\right)_{\ast} \mu_t = \left( H^{-1}_n \circ \phi^{\alpha_{n+1}}\right)_{\ast} \mu_t = \left(H^{-1}_n\right)_{\ast} \left(\phi^{\alpha_{n+1}}\right)_{\ast} \mu_t = \xi^n_t.
\end{equation*}
Here we used the relation $f_{\ast}g_{\ast} \mu = \left(f \circ g\right)_{\ast} \mu$ for maps $f,g$. This holds because we have for any measurable set $A$: 
\begin{equation*}
f_{\ast}g_{\ast} \mu\left(A\right) = g_{\ast}\mu\left(f^{-1}\left(A\right)\right) = \mu\left(g^{-1}\left(f^{-1}\left(A\right)\right)\right) = \mu\left(\left(f \circ g\right)^{-1}\left(A\right)\right) = \left(f \circ g \right)_{\ast} \mu\left(A\right).
\end{equation*}
In the next step we want to estimate $\mu\left(H^{-1}_{n+1}\left(N_t\right) \triangle H^{-1}_{n}\left(N_t\right)\right)$. For this purpose, we have to examine which parts of the set $N_t$ are not mapped back to $N_t$ under $h^{-1}_{n+1} = h^{-1}_{\mathfrak{2},n+1} \circ h^{-1}_{\mathfrak{1},n+1}$. The measure difference is composed of our error set $E_n \cap N_t$, the part, where the conjugation map $h^{-1}_{n+1}$ is constructed to prove minimality (i.e. on the $l_nq_n$ sets $\left[\frac{k}{l_nq_n}, \frac{k}{l_nq_n}+\frac{1}{l^2_n \cdot q_n}\right] \times \mathbb{T}$ for $k=0,...,l_nq_n-1$), and the part that is not mapped back to $N_t$ under $h^{-1}_{\mathfrak{1},n+1}$. The last one is caused by the translation about at most $\frac{n^2 \delta_n}{2l_n}$ in the $x_2$-coordinate produced by $h^{-1}_{\mathfrak{1},n+1}$. Altogether, we obtain:
\begin{equation} \label{eq:cauchy}
\mu\left(H^{-1}_{n+1}\left(N_t\right) \triangle H^{-1}_{n}\left(N_t\right)\right) = \mu\left(h^{-1}_{n+1}\left(N_t\right) \triangle N_t\right) \leq \frac{1}{l_n} + \frac{n^2 \delta_n}{l_n} +\mu(E_{n+1}) \leq \frac{1}{l_n}.
\end{equation}
Now we can use the same approach as in \cite{Win}, chapter 7: \\
By equation \ref{eq:cauchy} the sequence $\left\{H^{-1}_n\left(N_t\right)\right\}_{n \in \mathbb{N}}$ is a Cauchy sequence in the metric on the associated measure algebra. Since this space is complete (e.g. \cite{Pet}, Proposition 1.4.3.), there exists a limit $B_t \coloneqq \lim_{n \rightarrow \infty} H^{-1}_n \left(N_t\right)$ in the measure algebra. For this limit we have $\mu\left(B_t\right) = \mu\left(N_t\right)$, because $H^{-1}_n$ is measure-preserving. The sets $B_t$ and $B_s$ are measurably disjoint due to the disjointness of the sets $N_t$ and $N_s$.  Moreover, we have weak convergence of the measures $\left(\xi^n_t\right)_{n \in \mathbb{N}}$ to a measure $\xi_t$, where $\xi_t\left(A\right) = \frac{\mu\left(A \cap B_t\right)}{\mu\left(B_t\right)}$ for any measurable set $A \subseteq \mathbb{T}^2$. For this absolutely continuous measure $\xi_t$ we conclude $\lim_{n \rightarrow \infty}\left(T_n\right)_{\ast} \xi^n_t\left(A\right) = T_{\ast} \xi_t\left(A\right)$ due to the triangel inequality
\begin{align*}
\mu\left(H_n\left(T^{-1}_n A\right) \cap N_t\right) & = \mu\left(T^{-1}_n A \cap H^{-1}_n\left(N_t\right)\right) \leq \mu\left(T^{-1}_n A \cap B_t\right) + \mu\left(H^{-1}_n\left(N_t\right) \triangle B_t\right) \\
& \leq \mu\left(T^{-1}_n A \triangle T^{-1}A\right) + \mu\left(T^{-1} A \cap B_t\right) + \mu\left(H^{-1}_n\left(N_t\right) \triangle B_t\right)
\end{align*}
(where the first summand converges to $0$ as $n \rightarrow \infty$ because of $T_n \rightarrow T$). So we obtain
\begin{equation*}
\xi_t = \lim_{n \rightarrow \infty} \xi^n_t = \lim_{n \rightarrow \infty} \left(T_n\right)_{\ast} \xi^n_t = T_{\ast} \xi_t
\end{equation*}
using the shown $T_n$-invariance of the measure $\xi^n_t$. Thus, the measures $\xi_t$ are $T$-invariant. \\
Furthermore, these measures $\xi_t$ are linearly independent because the sets $B_1,...,B_d$ are measurably disjoint as noted before. Since any non-ergodic invariant measure can be written as a linear combination of ergodic measures (\cite{Wa}, Theorem 5.15), there cannot be less than $d$ ergodic measures.

\subsubsection*{Estimates on Birkhoff sums} \label{trappropmin}
In this subsection we show that the measures $\xi_t$ are the only possible ergodic measures for $T$. For this purpose, we will prove a result on the Birkhoff sums (see Lemma \ref{lem:birk2}) and have to gain control over almost everything of every $\phi$-orbit. In this connection the following sets are useful: In case of $0 \leq s < l_n q_n$, $0 < j_1 < l_n$ and $0 \leq j_2 < l_n$  we introduce
\begin{equation*}
\Delta^t_{s,j_1,j_2} = \left[\frac{s}{l_nq_n} + \frac{j_1}{l^2_n\cdot q_n}, \frac{s}{l_nq_n} + \frac{j_1+1}{l^2_n\cdot q_n} \right] \times \left[\frac{t}{r} + \frac{j_2}{l_nr}, \frac{t}{r} + \frac{j_2+1}{l_nr} \right] 
\end{equation*}
Note that there are $l^3_n q_n \cdot \left(1-\frac{1}{l_n}\right)$ such sets $\Delta^t_{s, j_1, j_2}$ on $N_t$. We denote the family of these sets by $\Omega^t_n$ as well as the union of these sets by $\tilde{\Omega}^t_n$. Then $\mu\left(N_t \setminus \tilde{\Omega}^t_n\right) = \frac{1}{l_nr}$, i.e. $\mu_t\left(N_t \setminus \tilde{\Omega}^t_n\right)\leq \frac{1}{l_n}$.\\
We observe that diam$\left(H^{-1}_{n} \left(\Delta^t_{s,j_1,j_2}\right)\right) < \|DH^{-1}_n\|_0 \cdot \frac{1}{l_n}$. By the requirements on the number $l_n$ in equation \ref{cond l birk} we obtain 
\begin{equation*}
\left| \rho_i\left(H^{-1}_{n} \left(x\right)\right) - \rho_i\left(H^{-1}_{n} \left(y\right)\right) \right| \leq  \textnormal{Lip}\left(\rho_i\right) \cdot \textnormal{diam}\left(H^{-1}_{n} \left(\Delta^t_{s,j_1,j_2}\right)\right)  < \frac{1}{n^2}
\end{equation*}
for every $x,y \in \Delta^t_{s,j_1,j_2}$ and the function $\rho_i \in \Xi$ in case of $i=1,...,n$. Averaging over all $y \in \Delta^t_{s,j_1,j_2}$ we obtain:
\begin{equation} \label{eq:block}
\left| \rho_i \left(H^{-1}_{n}\left(x\right)\right) - \frac{1}{\xi^n_t\left(H^{-1}_{n} \left(\Delta^t_{s,j_1,j_2}\right)\right)} \int_{H^{-1}_{n} \left(\Delta^t_{s,j_1,j_2}\right)} \rho_i \: d\xi^n_t \right| < \frac{1}{n^2}.
\end{equation}
Furthermore, we recall that the image of the trapping region $B^t_{s,i}$ under $h^{-1}_{n+1}$ is contained in $\Delta^t_{s,\lfloor \frac{i}{l_n} \rfloor, i \mod l_n}$. Vice versa, $B^t_{s, j_1 \cdot l_n + j_2}$ is the unique trapping region that is mapped into $\Delta^t_{s,j_1,j_2}$. Hence, we can estimate the number of $i \in \left\{0,...,q_{n+1}-1\right\}$ such that $h^{-1}_{n+1} \circ \phi^{i \cdot \alpha_{n+1}} \left(x\right)$ is contained in $\Delta^t_{s, j_1, j_2}$ by $\varpi^n_t\left(x\right)\cdot \frac{\left(1-\frac{8}{n^2}\right) \cdot q_{n+1}}{l^3_n q_n} $ for arbitrary $x \in \mathbb{T}^2$ using Lemma \ref{lem trap}.

\begin{lemma} \label{lem:birk}
Let $\rho_i \in \Xi$ and $i=1,...,n$. Then for every $y \in \mathbb{T}^2$ we have
\begin{equation*}
\inf_{\xi^n \in \Theta_n} \left| \frac{1}{q_{n+1}} \sum^{q_{n+1}-1}_{k=0} \rho_i\left(T^k_{n+1}y\right) - \int \rho_i \: d\xi^n \right| < \frac{20}{n^2} \cdot \left\|\rho_i\right\|_0 + \frac{1}{n^2},
\end{equation*}
where $\Theta_n$ is the simplex generated by $\left\{ \xi^n_0,..., \xi^n_{r-1}\right\}$.
\end{lemma}

\begin{proof}
Let $x \in \mathbb{T}^2$ be arbitrary. We introduce the measure $\xi^n_x \coloneqq \sum^{r-1}_{t=0} \varpi^n_t\left(x\right) \cdot \xi^n_t \in \Theta_n$. \\
The set of numbers $k \in \left\{0,1,...,q_{n+1}-1\right\}$ such that the iterates $\phi^{k \cdot \alpha_{n+1}}\left(x\right)$ are not contained in one of the trapping regions of the second kind is denoted by $I_{a}$. Referred to Lemma \ref{lem trap} there are at most $\frac{10}{n^2} \cdot q_{n+1}$ numbers in $I_{a}$. We obtain $\left|\sum_{k \in I_a} \rho_i\left( H^{-1}_{n+1} \circ \phi^{k \cdot \alpha_{n+1}} \left(x\right)\right)\right| \leq \left\|\rho_i\right\|_0 \cdot \frac{10}{n^2} \cdot q_{n+1}$. \\
Moreover, we denote the set of $k \in \left\{0,1,...,q_{n+1}-1\right\}$ such that the iterate $h^{-1}_{n+1} \circ \phi^{k\alpha_{n+1}}\left(x\right)$ is contained in the corresponding trapping region $\Delta \in \Omega^t_n$ by $I_{\Delta}$. By the above considerations there are at least $\varpi^n_t\left(x\right) \cdot \frac{\left(1-\frac{8}{n^2}\right) \cdot q_{n+1}}{l^3_n q_n}= \varpi^n_t\left(x\right) \cdot q_{n+1} \cdot \left(1-\frac{8}{n^2}\right) \cdot \mu_t\left(\Delta\right)$ and at most $\varpi^n_t\left(x\right) \cdot q_{n+1} \cdot \mu_t\left(\Delta\right)$ many numbers in $I_{\Delta}$ for an arbitrary $\Delta \in \Omega^t_n$. Thus, we obtain for an arbitrary $\Delta \in \Omega^t_n$ using equation \ref{eq:block}:
\begin{align*}
& \left|\frac{1}{q_{n+1}}\sum_{j \in I_{\Delta}} \rho_i \left( H^{-1}_{n+1} \circ \phi^{j\alpha_{n+1}}\left(x\right)\right) - \int_{H^{-1}_{n} \left(\Delta\right)} \rho_i \: d\left(\varpi^n_t\left(x\right)\xi^n_t\right) \right| \\
\leq & \frac{\left(\varpi^n_t\left(x\right)\mu_t\right)\left(\Delta\right)}{n^2} + \frac{8}{n^2} \cdot \int_{H^{-1}_{n}\left(\Delta\right)} \left|\rho_i \right| \: d\left(\varpi^n_t\left(x\right)\xi^n_t\right) \leq \left(\varpi^n_t\left(x\right)\mu_t\right)\left(\Delta\right) \cdot \left(\frac{1}{n^2} + \frac{8}{n^2} \cdot \left\|\rho_i\right\|_0 \right).
\end{align*}
Altogether, we conclude
\begin{align*}
& \left| \frac{1}{q_{n+1}} \sum^{q_{n+1}-1}_{k=0} \rho_i\left(H^{-1}_{n+1} \circ \phi^{k\alpha_{n+1}}x\right) - \int \rho_i \:d\xi^n_x \right| = \Bigg{|} \frac{1}{q_{n+1}} \sum^{q_{n+1}-1}_{k=0} \rho_i\left(H^{-1}_{n+1} \circ \phi^{k\alpha_{n+1}}(x)\right)  \\
&\qquad\qquad\qquad\qquad\qquad - \sum^{r-1}_{t=0} \left( \sum_{\Delta \in \Omega^t_n} \int_{H^{-1}_{n} \left(\Delta\right)} \rho_i \: d\left(\varpi^n_t\left(x\right)\xi^n_t\right) + \int_{H^{-1}_{n} \left(N_t \setminus \tilde{\Omega}^t_n\right)} \rho_i \: d\left(\varpi^n_t\left(x\right)\xi^n_t\right)\right) \Bigg{|} \\
& \qquad\qquad\qquad \leq  \left| \sum^{r-1}_{t=0}  \sum_{\Delta \in \Omega^t_n} \left(\frac{1}{q_{n+1}}\sum_{j \in I_{\Delta}} \rho_i \left( H^{-1}_{n+1} \circ \phi^{j\alpha_{n+1}}\left(x\right)\right) - \int_{H^{-1}_{n} \left(\Delta\right)} \rho_i  d\left(\varpi^n_t\left(x\right)\xi^n_t\right) \right)\right| \\
& \qquad\qquad\qquad\qquad\qquad\qquad + \frac{1}{q_{n+1}} \cdot \left\|\rho_i\right\|_0 \cdot \frac{10}{n^2} \cdot q_{n+1} + \left\|\rho_i\right\|_0 \cdot \sum^{r-1}_{t=0} \left(\varpi^n_t\left(x\right)\mu_t\right)\left(N_t \setminus \tilde{\Omega}^t_n\right) \\
& \qquad\qquad\qquad \leq   \frac{1}{n^2} + \frac{8}{n^2} \cdot \left\|\rho_i\right\|_0 + \left\|\rho_i\right\|_0 \cdot \frac{10}{n^2}+ \frac{2 \cdot \left\|\rho_i\right\|_0}{l_n} = \frac{20}{n^2} \cdot \left\|\rho_i\right\|_0 + \frac{1}{n^2}.
\end{align*} 
With $x= H_{n+1}\left(y\right)$ we obtain the claim.
\end{proof}

We point out that the measure $\xi^n_x$ used in the above proof was dependent on the point $x$, but independent of the function $\rho \in \Xi$. 
\begin{lemma} \label{lem:birk2}
For every $\rho \in \Xi$ and $y \in \mathbb{T}^2$ we have
\begin{equation*}
\inf_{\xi^n \in \Theta_n} \left| \frac{1}{q_{n+1}} \sum^{q_{n+1}-1}_{k=0} \rho\left(T^k\left(y\right)\right) - \int \rho \: d\xi^n \right| \rightarrow 0 \ \  \text{ as } n\rightarrow \infty,
\end{equation*}
where $\Theta_n$ is the simplex generated by $\left\{ \xi^n_0,..., \xi^n_{r-1}\right\}$.
\end{lemma}

\begin{proof}
By Remark \ref{close iterates} we have
\begin{equation*}
d^{\left(q_{n+1}\right)}_0\left(T,T_{n+1}\right) \coloneqq \max_{i=0,1,...,q_{n+1}-1} d_0\left(T^i, T^i_{n+1}\right) \stackrel{n \rightarrow \infty}{\rightarrow} 0.
\end{equation*}
Then for every $\rho \in \Xi$ we have $\left| \rho\left(T^i\left(x\right)\right)-\rho\left(T^i_{n+1}\left(x\right)\right)\right| \stackrel{n \rightarrow \infty}{\rightarrow} 0$ uniformly for $i=0,1,...,q_{n+1}-1$, because every continuous function on the compact space $\mathbb{T}^2$ is uniformly continuous. Thus, we get: $\left\| \frac{1}{q_{n+1}} \sum^{q_{n+1}-1}_{i=0} \rho\left(T^i\left(x\right)\right) - \frac{1}{q_{n+1}} \sum^{q_{n+1}-1}_{i=0} \rho\left(T^i_n\left(x\right)\right)\right\|_0 \stackrel{n \rightarrow \infty}{\rightarrow} 0$. Applying the previous Lemma \ref{lem:birk} we obtain the claim.
\end{proof}

Since the family $\Xi$ is dense in $C\left(\mathbb{T}^2, \mathbb{R}\right)$, the convergence holds for every continuous function by an approximation argument. \\
Now we can prove that the measures $\xi_0,...,\xi_{r-1}$ are the only possible ergodic ones: Assume that there is another ergodic invariant probability measure $\xi$. By the Birkhoff Ergodic Theorem we have for every $\rho \in C\left(\mathbb{T}^2, \mathbb{R}\right)$
\begin{equation*}
\lim _{n \rightarrow \infty} \frac{1}{n}\sum^{n-1}_{k=0} \rho\left(T^k\left(x\right)\right) = \int_{\mathbb{T}^2} \rho \; d\xi \;\;\; \text{ for $\xi$-a.e. } x \in \mathbb{T}^2.
\end{equation*}
With the aid of Lemma \ref{lem:birk2} we obtain for every $\rho \in C\left(\mathbb{T}^2, \mathbb{R}\right)$ and $x$ in a set of $\xi$-full measure:
\begin{equation*}
\int_{\mathbb{T}^2} \rho \; d\xi = \lim _{n \rightarrow \infty} \frac{1}{n}\sum^{n-1}_{k=0} \rho\left(T^k\left(x\right)\right) = \lim _{n \rightarrow \infty} \frac{1}{q_{n+1}}\sum^{q_{n+1}-1}_{k=0} \rho\left(T^k\left(x\right)\right) = \lim _{n \rightarrow \infty} \int_{\mathbb{T}^2} \rho \;d\xi^n,
\end{equation*}
where $\xi^n$ is in the simplex generated by $\left\{\xi^n_0,...,\xi^n_{r-1}\right\}$. As noted this measure does not depend on the function $\rho$. Thus, we have for every $\rho \in C\left(\mathbb{T}^2, \mathbb{R}\right)$: $\lim _{n \rightarrow \infty} \int_{\mathbb{T}^2} \rho \;d\xi^n=\int_{\mathbb{T}^2} \rho \; d\xi$. Since the simplex generated by $\left\{\xi_0,...,\xi_{r-1}\right\}$ is weakly closed, this implies that $\xi$ is in this simplex. We recall that ergodic measures are the extreme points in the set of invariant Borel probability measures (see \cite{Wa}, Theorem 5.15.). Then $\xi$ has to be one of the measures $\left\{\xi_0,...,\xi_{r-1}\right\}$ and we obtain a contradiction. Hence, the measures $\xi_t$, $t=0,\ldots, r-1$ are the only possible ergodic ones. Since we have already observed that these are linearly independent and any non-ergodic invariant measure can be written as a linear combination of ergodic ones, we conclude that the $\xi_t$ are exactly the ergodic measures of $T$.

\subsection{Possible Generalizations}

\begin{remark}
By putting the combinatorics from \cite[section 5]{AK} on the rectangles $\left[ \frac{i}{l^3q}, \frac{i+1}{l^3q} \right) \times \left[ \frac{t}{r}, \frac{t+1}{r} \right)$ that are not contained in the minimality region, we can construct the diffeomorphism $T$ to be even weakly mixing with respect to each measure $\xi_t$. We can realize these combinatorics with the aid of Theorem \ref{permutation = block-slide}.
\end{remark}

\begin{remark}
With some additional technical and notational effort it is possible to generalize Theorem \ref{theorem prescribed no of measures} and the previous Remark to any torus $\T^d$, $d \geq 2$. Indeed, we construct the $r$ ergodic invariant measures on $N_t = \mathbb{S}^1 \times \left[ \frac{t}{r}, \frac{t+1}{r} \right) \times \T^{d-2}$ and consider partition elements \begin{equation}
\left[ \frac{i_1}{l^{d+1}q}, \frac{i_1+1}{l^{d+1}q} \right) \times \left[ \frac{i_2}{lr}, \frac{i_2+1}{lr} \right) \times \left[\frac{i_3}{l},\frac{i_3+1}{l} \right) \times \ldots \times \left[\frac{i_d}{l}, \frac{i_d+1}{l} \right)
\end{equation}
as building blocks in the description of the combinatorics. Then we use the combinatorics from the beginning of section \ref{constr nsr} to map sets of the form $\left[ \frac{i}{l^{d+1}q}, \frac{i+1}{l^{d+1}q} \right) \times \left[ \frac{j}{lr}, \frac{j+1}{lr} \right) \times [0,1)^{d-2}$ to sets of the form $\left[ \frac{i_1}{l^{3}q}, \frac{i_1+1}{l^{3}q} \right) \times \left[ \frac{j}{lr}, \frac{j+1}{lr} \right) \times \left[\frac{i_3}{l},\frac{i_3+1}{l} \right) \times \ldots \times \left[\frac{i_d}{l}, \frac{i_d+1}{l} \right)$.
\end{remark}

\begin{remark}
In this Remark we present modifications in order to prove the existence of a real-analytic diffeomorphism $T \in \text{Diff }^\omega_\rho(\T^2,\mu)$ which is minimal and has countable many ergodic invariant measures. By weak*-convergence there must be at least one singular ergodic measure. Indeed, we have precisely one singular measure and the other invariant measures are absolutely continuous with respect to Lebesgue measure.

This time we are going to construct the invariant measures on sets $N_t$, $t \in \Z$. For $t \in \N$ we define $N_t = \mathbb{S}^1 \times \left[ \frac{t}{t+1}, \frac{ t+1}{t+2} \right) \subset \T^2$ with $x_2$-length $\frac{1}{(t+2) \cdot (t+1)}$. For each $n \in \N$ we choose $t_n \in \N$ such that
\begin{equation} \label{eq t}
\frac{1}{(t_n + 1) \cdot t_n } < \frac{\delta_{n-1}}{2^n \cdot \|DH^{-1}_n \|_0 \cdot \max_{i=1, \ldots, n} \text{Lip}(\rho_i)}.
\end{equation}
Hereby, we define the further sets
\begin{align*}
    & N_0 = \mathbb{S}^1 \times \left[ \frac{1}{(t_1 + 1) \cdot t_1 }, \frac{1}{2} \right), \\
    & N_{-n} = \mathbb{S}^1 \times \left[ \frac{1}{\left( t_{n+1}+1 \right) \cdot t_{n+1}}, \frac{1}{\left(t_n +1 \right) \cdot t_n} \right) \text{ for every } n \in \N.
\end{align*}
Additionally, for every $n \in \N$ we will use sets
\begin{align*}
    & \bar{N}_{t_n} = \mathbb{S}^1 \times \left[ \frac{t_n}{t_n +1}, 1 \right) = \bigcup_{t \geq t_n} N_t \\
    & N^{(1)}_{-n} = \mathbb{S}^1 \times \left[ 0, \frac{1}{\left( t_n+1 \right) \cdot t_n } \right).
\end{align*}
This time we choose $l_n = \left(t_n + 1 \right)!$. Note that by equation \ref{eq t} the conditions \ref{ln criterion} and \ref{cond l birk} are satisfied. Moreover, this choice of $l_n$ allows us to consider building blocks $\left[ \frac{i}{l^3_n q_n}, \frac{i+1}{l^3_n q_n} \right) \times \left[ \frac{j}{l_n}, \frac{j+1}{l_n} \right)$ for the $\frac{1}{l_nq_n}$-equivariant combinatorics of $h^{-1}_{\mathfrak{2},n+1}$ to be contained in $\bar{N}_{t_n}$, $N^{(1)}_{-n}$ and $N_t$ for $-n< t < t_n$. As before, $h^{-1}_{\mathfrak{2}, n+1}$ is supposed to map long stripes $\left[ \frac{i}{l^3_n q_n}, \frac{i+1}{l^3_n q_n} \right) \times \left[0,1 \right)$ to $\left[0,\frac{1}{l^2_n q_n} \right) \times \left[ \frac{i}{l_n}, \frac{i+1}{l_n} \right)$ for $0 \leq i < l_n$. For $l_n \leq i < l^2_n$ $h^{-1}_{\mathfrak{2}, n+1}$ maps stripes of width $\frac{1}{l^3_nq_n}$ and full height in the particular set $N^{(1)}_{-n}$ or $N_t$ for $-n< t < t_n$ to sets with height $\frac{1}{l_n}$, while on $\bar{N}_{t_n}$ $h^{-1}_{\mathfrak{2}, n+1}$ acts approximately as the identity on the building blocks. For the trapping map $h^{(\mathfrak{1})}$ the step function is constructed with steps of seize $\frac{\delta}{l}$ in our modification.
\end{remark}

\section{Future Work}

Finally we note that Theorem \ref{permutation = block-slide} can be used to upgrade many constructions from the smooth category to the analytic category on the torus. We list some results here.

\subsection{Real-analytic diffeomorphisms with homogeneous spectrum and disjointness of convolutions}

The second author in \cite{Ku-Dc} was able to show that on any smooth compact connected manifold M of dimension $m \geq 2$ admitting a smooth non-trivial circle action, there exists a smooth diffeomorphism $f \in A_\a = \{h \circ \phi^\a \circ h^{-1} : h \in  \text{Diff}^\infty (M,\mu)\}$
for every Liouvillian number $\a$ which admits a good approximation of type $(h, h + 1)$, a maximal spectral type disjoint with its convolutions and a homogeneous spectrum of multiplicity two for the Cartesian square $f \times f$. Its is possible to generalize this result to the analytic category for some Liouvillian numbers.

\begin{theorem}
For any $\rho>0$, there exist real-analytic diffeomorphisms $T\in \text{Diff }_\rho^\omega (\T^2, \mu)$ that have a maximal spectral type disjoint
with its convolutions, a homogeneous spectrum of multiplicity 2 for $T \times T$ and admit a good approximation of type $(h, h + 1)$.
\end{theorem}

\subsection{Coding untwisted AbC diffeomorphisms and the anti-classification problem}

This work is motivated from a series of pioneering work done by Belezney, Foreman, Hjorth, Rudolph and Weiss on the interface of ergodic theory and foundations of mathematics. They were able to show that the conjugacy problem in abstract ergodic theory is non Borel. Later Foreman and Weiss found a method to code a `large' class of smooth diffeomorphisms constructed on $\T^2$ or the annulus or the disk by an untwisted version of the AbC method into some symbolic systems known as \emph{uniform circular systems}. This in particular shows that  the measure isomorphism relation among pairs $(S,T)$ of measure preserving diffeomorphisms of $M$ is not a Borel set with respect to the $C^\infty$ topology. 

The first author was able to show that the constructions we do in the real-analytic category on $\T^2$ are robust enough to construct a large family of untwisted AbC diffeomorphisms measure theoretically isomorphic to  uniform circular systems. Loosely this can be summarized into the following theorem:

\begin{theorem}[\cite{Ba-Sr}]
Let $T$ be an ergodic transformation on a standard measure space. Then the following are equivalent:
\begin{enumerate}
\item $T$ is measure theoretically isomorphic to a real-analytic (untwisted) AbC diffeomorphism (satisfying some requirements).
\item $T$ is isomorphic to a uniform circular system (with `fast' growing parameters).
\end{enumerate}
\end{theorem}

This along with some additional works of Foreman and Weiss would imply an anti-classification result for measure preserving real-analytic diffeomorphisms. More precisely,
\begin{theorem}
The measure-isomorphism relation among pairs $(S,T)\in\text{Diff }^\omega_\rho(\T^2,\mu)\times \text{Diff }^\omega_\rho(\T^2,\mu)$  is not a Borel set with respect to the $\text{Diff }^\omega_\rho(\T^2,\mu)$ topology.
\end{theorem}

\vspace{2cm}

\noindent\emph{Acknowledgement:} The authors would like to thank Anatole Katok for numerous discussions and constant encouragements. Additionally, the second author would like to thank the Center for Dynamics and Geometry at Penn State for hospitality and financial support at a visit in November 2016 when large parts of this paper were completed. He also acknowledges financial support by the ``Forschungsfonds'' at the Department of Mathematics, University of Hamburg.

\end{document}